\documentclass{article}
\usepackage[scale=0.75,a4paper]{geometry}
\usepackage{multirow}
\usepackage{braket,amsfonts}
\usepackage{bm}
\usepackage{amssymb}
\usepackage{stmaryrd}

\usepackage{mathtools} % redundant label removal, adjust limits
\usepackage[colorlinks,hypertexnames=false,citecolor=red,linkcolor=blue]{hyperref}
\usepackage{pgfplots}
\pgfplotsset{compat=1.16}

\usepackage{algorithmic}
\usepackage{comment}
\usepackage{graphicx,epstopdf}
\usepackage{booktabs}

\usepackage[affil-it]{authblk}
\usepackage{amsthm}
\newtheorem{theorem}{Theorem}[section]
\newtheorem{lemma}[theorem]{Lemma}

\usepackage{siunitx}

\date{}
\usepackage{float}
\usepackage{caption}
\usepackage{graphicx,xcolor}
\usepackage{verbatim}
\usepackage{subfigure}
\usepackage{cases}
\usepackage{amssymb,version}
\usepackage{color}
\usepackage{verbatim}
\usepackage{pgfplots}
\usetikzlibrary{shapes,arrows,snakes,calendar,matrix,backgrounds,folding,calc,positioning,spy}
\usepackage{tikz}

\makeatletter
\def\lst@lettertrue{\let\lst@ifletter\iffalse}
\makeatother
\begin{document}
	  \title{High-order, long-time stable and parallel decoupled GBDF$k$ SAV ensemble schemes for the Navier--Stokes--Darcy flow with random hydraulic conductivity tensors
		\thanks{This work of Fukeng Huang was partially supported by NSFC Grant No. 12501554 and  Changxin Qiu was partially supported by Zhejiang Provincial Natural Science Foundation of China under Grant No. LMS26A010009, NSFC Grant No. 12201327 and Natural Science Foundation of Ningbo Municipality Grant No. 2022J087.}}

	\author{Wei-Wei Han\footnote{School of Mathematical Science, Eastern Institute of Technology, Ningbo, 315200, PR China \\({\tt wwhan@eitech.edu.cn}).}
		\quad
		Fukeng Huang\footnote{School of Mathematical Science, Eastern Institute of Technology, Ningbo, 315200, PR China \\({\tt fkhuang@eitech.edu.cn}).}
		\quad
		Changxin Qiu\footnote{Corresponding author. School of Mathematics and Statistics, Ningbo University, Ningbo 315211, PR China. ({\tt qiuchangxin@nbu.edu.cn}).}
		
	}

\date{\today}
\maketitle

\begin{abstract}
We develop and analyze high-order ensemble schemes for the unsteady Navier--Stokes--Darcy system with uncertain initial conditions, forcing terms, hydraulic conductivity tensors, and Lions-Beavers-Joseph-Saffman interface conditions. The proposed schemes which are called GSAV-GBDF$k$-Ensemble schemes integrate a partitioned decoupling strategy, the generalized scalar auxiliary variable (GSAV) approach, and generalized BDF$k$ discretizations. This framework achieves high-order temporal accuracy and long-time stability, permits explicit treatment of the nonlinear term, and facilitates an efficient ensemble implementation for multiple parameter realizations by sharing a single, unified coefficient matrix at each time step. Moreover, the numerical solutions are shown to satisfy uniform-in-time bounds without time-step restrictions. Owing to the ensemble formulation, the resulting linear systems share common coefficient matrices, which significantly improves computational efficiency. We further establish optimal-order error estimates for the proposed high-order schemes.  Numerical results are included to confirm the theoretical analysis and to illustrate the accuracy, stability, and efficiency of the proposed methods.
\end{abstract}

{\normalsize }\textbf{Keywords:} Navier--Stokes--Darcy; GSAV; Ensemble scheme; High-order; Long-time stability

{\normalsize }\textbf{AMS subject classifications:}	 65M12, 65M15, 65M20.

\section{Introduction}
We consider in this paper the construction and analysis of efficient high-order ensemble schemes for the unsteady Navier--Stokes--Darcy system with uncertain initial conditions, forcing terms, and hydraulic conductivity tensors.
To account for these uncertainties, we consider an ensemble of $J$ unsteady Navier--Stokes--Darcy systems corresponding to different parameter sets
$\left( \mathbf{u}_{j}^{0}, \phi_{j}^{0}, \mathbf{f}_{f,j}, f_{p,j}, \mathrm{K}_{j} \right), j=1,\dots,J.$
More precisely, for each ensemble member, we seek the fluid velocity $\mathbf{u}_{j}$, fluid pressure $p_{j}$, and hydraulic head $\phi_{j}$ satisfying
\begin{subequations}
	\begin{align}
		&\frac{\partial \mathbf{u}_{j}}{\partial t} - \nu \Delta \mathbf{u}_{j} + \left( \mathbf{u}_{j} \cdot \nabla \right) \mathbf{u}_{j} + \nabla p_{j} = \mathbf{f}_{f,j}, \quad \text{and} \quad \nabla \cdot \mathbf{u}_{j} = 0, \quad \text{in } \Omega_f \times (0,T], \label{para_stokes} \\
		&S\frac{\partial \phi_{j}}{\partial t} - \nabla \cdot \left( \mathrm{K}_{j} \nabla \phi_{j} \right) = f_{p,j}, \quad \text{in } \Omega_p \times (0,T],\label{para_darcy}
	\end{align}
\end{subequations}
%\begin{equation}\label{para_stokes}
%	\frac{\partial \mathbf{u}_{j}}{\partial t} - \nu \Delta \mathbf{u}_{j} + \left( \mathbf{u}_{j} \cdot \nabla \right) \mathbf{u}_{j} + \nabla p_{j} = \mathbf{f}_{f,j}, \quad \text{and} \quad \nabla \cdot \mathbf{u}_{j} = 0, \quad \text{in } \Omega_f \times (0,T],
%\end{equation}
%and
%\begin{equation}\label{para_darcy}
%	S\frac{\partial \phi_{j}}{\partial t} - \nabla \cdot \left( \mathrm{K}_{j} \nabla \phi_{j} \right) = f_{p,j}, \quad \text{in } \Omega_p \times (0,T],
%\end{equation}
coupled through the \textcolor{black}{Lions--Beavers--Joseph--Saffman interface conditions}
\begin{equation}\label{para_interface}
	\left\{
	\begin{aligned}
		&\mathbf{u}_{j} \cdot \mathbf{n}_f - \mathrm{K}_{j}\nabla \phi_{j} \cdot \mathbf{n}_p = 0, \quad \text{on } \Gamma \times (0,T],\\
		&p_{j} - \nu \mathbf{n}_f \cdot \nabla \mathbf{u}_{j} \cdot \mathbf{n}_f + 1/2\mathbf{u}_{j}\cdot \mathbf{u}_{j} = g\phi_{j}, \quad \text{on } \Gamma \times (0,T],\\
		&-\nu \tau_i \cdot \nabla \mathbf{u}_{j} \cdot \mathbf{n}_f = \frac{\alpha_{BJ}\nu\sqrt{d}}{\sqrt{\tau_i \cdot \mathrm{K}_{j} \cdot \tau_i}}\, \mathbf{u}_{j}\cdot \tau_i, \quad \text{on } \Gamma \times (0,T], \quad i=1,\dots,d-1,
	\end{aligned}
	\right.
\end{equation}
where the initial and boundary conditions are subject to
\begin{displaymath}
	\begin{aligned}
		\mathbf{u}_{j}(0,x) = \mathbf{u}_{j}^0(x) \quad \text{in} \quad \Omega_f \quad \text{and} \quad \phi_{j} (0,x) = \phi_{j}^0(x) \quad \text{in} \quad \Omega_p,\\
		\mathbf{u}_{j}(t,x) = 0 \quad \text{on} \quad \Gamma_f  \quad \text{and} \quad \phi_{j} (t,x) = 0 \quad \text{on} \quad \Gamma_p.
	\end{aligned}
\end{displaymath}
Here, $\Omega=\Omega_f\cup\Omega_p \subset \mathbb{R}^2$, where $\Omega_f$ and $\Omega_p$ denote the free-fluid and porous-media regions, respectively, and $\Gamma$ is the interface between them. We also define $\Gamma_f = \partial \Omega_f \backslash \Gamma$ and $\Gamma_p = \partial \Omega_p \backslash \Gamma$.  The vectors $\mathbf{n}_f$ and $\mathbf{n}_p$ are the unit outward normals on $\partial\Omega_f$ and $\partial\Omega_p$, and $\{\tau_i\}_{i=1}^{d-1}$ are linearly independent unit tangent vectors on $\Gamma$. The constants $\nu$, $T$, $S$, $\alpha_{BJ}$, and $g$ denote the kinematic viscosity, final time, specific mass storativity coefficient, Beavers--Joseph--Saffman condition constant, and gravitational acceleration constant, respectively, and each hydraulic conductivity tensor $\mathrm{K}_{j}$ is assumed to be symmetric positive definite. We denote by $\mathbf{f}_{f,j}$ the external body force density and by $f_{p,j}$ the sink/source term. Specifically, the hydraulic conductivity tensors, denoted by $\mathrm{K}_{j}=\mathrm{K}(\mathbf{x}, \omega_j)$, are modeled as spatially dependent random fields (their explicit statistical characterizations are detailed in Section \ref{sec:experiments}). This parameterization fundamentally transforms the deterministic governing equations into a coupled stochastic partial differential equation (SPDE) system. Consequently, the primary objective of our analytical and numerical efforts is to efficiently and accurately compute the statistical moments of the solution quantities of interest, most notably the mean and variance.

The first condition in \eqref{para_interface} represents mass conservation and enforces the continuity of the normal velocity across the interface. The second equation, known as the Lions interface condition \cite{chidyagwai2009solution,girault2009dg}, imposes the balance of normal stresses. Specifically, it equates the normal stress in the free-fluid region to the corresponding terms in the porous medium, incorporating inertial effects through the dynamic pressure term $\frac{1}{2}\mathbf{u}_{j}\cdot \mathbf{u}_{j}$. The inclusion of this nonlinear term is crucial for establishing the rigorous nonlinear stability of the coupled scheme. Finally, the third equation specifies the well-known Beavers-Joseph-Saffman (BJS) condition \cite{beavers1967boundary,saffman1971boundary}, which dictates that the tangential slip velocity is proportional to the viscous shear stress at the interface.

The Navier--Stokes--Darcy system is a fundamental model for coupling free-fluid flow with porous-media flow and has been widely used in groundwater hydrology, oil reservoir simulation, filtration, and coupled surface--subsurface flow problems. In practical applications, uncertainties stemming from initial conditions, source terms, or heterogeneous hydraulic conductivity tensors usually require repeated simulations for large sets of samples, making uncertainty quantification computationally expensive. The overarching computational challenges are threefold. First, different samples may lead to different coefficient matrices, which greatly increases the cost of ensemble simulations. Second, the velocity, pressure, and hydraulic head are strongly coupled through the equations and interface conditions, so direct discretizations often require solving large coupled systems. Third, the nonlinear convection term is difficult to treat efficiently: an explicit treatment may suffer from stability restrictions, while an implicit treatment leads to nonlinear solves. These challenges motivate the construction of efficient, stable, and high-order ensemble schemes for the Navier--Stokes--Darcy system.

To systematically propagate uncertainties arising from random or parameterized inputs, uncertainty quantification (UQ) frameworks are widely employed. These frameworks typically couple sampling-based strategies, such as the Monte Carlo method and its variants \cite{barth2012multilevel,kuo2012quasi}, stochastic collocation methods \cite{babuvska2007stochastic,ganis2008stochastic,nobile2008sparse,xiu2005high}, or nonintrusive polynomial chaos methods \cite{reagana2003uncertainty}, with deterministic solvers for the underlying PDE system. While highly flexible and straightforward to implement, these approaches incur formidable computational costs, particularly when each sample realization necessitates the resolution of a large-scale, strongly coupled nonlinear system. An alternative paradigm involves substituting the high-fidelity model with reduced-order or surrogate models \cite{owen2017comparison}. Although this significantly mitigates the computational cost per evaluation, it inherently introduces approximation errors and circumvents, rather than accelerates, the repeated solution of the original full-order coupled system.

To directly address the computational bottleneck of repeated full-order simulations, the ensemble method, pioneered by Jiang and Layton \cite{jiang2014algorithm}, offers a highly attractive numerical paradigm. Its fundamental mechanism involves manipulating the temporal discretization to construct linear systems characterized by a shared, parameter-independent coefficient matrix for all ensemble members. This structural uniformity is crucial, as it enables the deployment of efficient block solvers and drastically amortizes the overhead of matrix factorizations or preconditioner setups across all realizations. Building on this profound computational advantage, the ensemble methodology has been rapidly extended to a broad spectrum of complex fluid models, including the Navier--Stokes equations \cite{gunzburger2019efficient,jiang2015higher,jiang2021stabilized}, natural convection \cite{fiordilino2018second}, magnetohydrodynamics (MHD) \cite{mohebujjaman2017efficient}, Stokes--Darcy systems \cite{he2020artificial,jiang2021artificial,jiang2019efficient,jiang2021sav,jiang2024highly}, and dual-porosity Navier--Stokes flows \cite{qiu2025high}.

%To handle the different subdomains for the fluid velocity, hydraulic head, and interface conditions, partitioned methods \cite{mu2010decoupled,chen2013efficient,layton2013analysis,shan2013partitioned,shan2013decoupling,chen2016efficient,wang2024class} have been widely used. These methods decouple the coupled problem into smaller subproblems posed on $\Omega_f$ and $\Omega_p$, allowing one to take advantage of existing solvers for each subproblem. This usually leads to reduced memory usage and lower computational cost.

Another important issue is the treatment of the nonlinear term in the Navier--Stokes equations. Implicit treatments are often robust but computationally expensive, especially in the ensemble setting. Explicit treatments are computationally attractive, but they may compromise stability unless carefully designed. For Stokes--Darcy-type systems, additional restrictions may arise from the partitioned treatment used to decouple the coupled system \cite{he2020artificial,jiang2019efficient}. In recent years, the scalar auxiliary variable (SAV) approach \cite{shen2018scalar,shen2019new} and its variants have provided an effective framework for constructing stable schemes with explicit treatment of nonlinear terms. In particular, it was shown in \cite{jiang2021sav} that the SAV approach can remove the time-step restriction induced by the partitioned treatment for linear Stokes--Darcy ensembles. The generalized SAV (GSAV) approach proposed in \cite{huang2022new,huang2020highly} preserves the main advantages of SAV while being more suitable for efficient implementation and high-order time discretizations. In this work, we further improve the GSAV approach and construct unconditionally long-time stable numerical schemes for the Navier--Stokes--Darcy system. To enhance the stability of high-order schemes, we adopt the recently developed generalized BDF$k$ schemes \cite{huang2024new}, which possess larger stability regions than the classical BDF$k$ schemes and thus permit larger time-step sizes in high-order discretizations for stiff problems. Motivated by these considerations, we combine a partitioned ensemble strategy, the GSAV approach, and the generalized BDF$k$ discretization to develop efficient, high-order, long-time stable schemes for the unsteady Navier--Stokes--Darcy system with uncertain initial conditions, forcing terms, hydraulic conductivity tensors, and Lions-Beavers-Joseph-Saffman interface conditions.

The main contributions of this work can be summarized as follows.
\begin{itemize}
	\item We develop high-order ensemble schemes for the unsteady Navier--Stokes--Darcy system that are based on a partitioned decoupling strategy \cite{chen2013efficient,layton2013analysis,mu2010decoupled,shan2013partitioned,wang2024class}, the GSAV approach, and generalized BDF$k$ formulas. The resulting schemes allow explicit treatment of the nonlinear term and can efficiently handle multiple realizations with varying parameters are solved using a single, unified coefficient matrix at each time step.
	\item We prove that the proposed schemes admit unconditional long-time uniform bounds for the numerical solutions. Thus, the schemes are suitable for long-time simulations of parameterized Navier--Stokes--Darcy ensembles without any time-step restriction arising from stability considerations.	
	\item We establish optimal-order error estimates for the proposed high-order schemes. To the best of our knowledge, this appears to be the first rigorous optimal-order error analysis for high-order GSAV ensemble schemes of this type for the nonlinear Navier--Stokes--Darcy system.
	\item We present extensive numerical examples to validate the theoretical results and to demonstrate the accuracy, stability, and efficiency of the proposed methods.
\end{itemize}

The rest of the paper is organized as follows. In Section~2, we introduce the notation and present the GSAV-GBDF$k$-Ensemble algorithms for the parameterized Navier--Stokes--Darcy system. In Section~3, we prove the long-time stability without time-step constraints of the proposed schemes. Section~4 is devoted to the error analysis, where optimal-order estimates are established. In Section~5, we present numerical experiments to verify the theoretical results and illustrate the efficiency of the proposed methods.

\section{Ensemble algorithm for the parameterized model}
In this section, we present the proposed ensemble algorithm for the unsteady Navier--Stokes--Darcy system. Before this, we first introduce some necessary notations and preliminaries.

\subsection{Notation and preliminaries}
Define the following function spaces:
\begin{displaymath}
	\begin{aligned}
		X_{f} &:= \left\{ \mathbf{u}\in \mathbf{H}^1\left( \Omega_f \right) : \mathbf{u}|_{\Gamma_f} = 0 \right\}, \quad X_p := \left\{ \psi \in H^1\left(\Omega_p \right) : \psi|_{\Gamma_p} = 0  \right\} \\
		Q_{f} &:= L^2\left(\Omega_f \right), \quad V := \left\{ \mathbf{u} \in X_f: \left( \nabla \cdot \mathbf{u} , p \right) = 0,\, \forall p \in Q_f \right\}.\\
	\end{aligned}
\end{displaymath}
We represent by $\| \cdot \|_{\Gamma}$ the $L^2\left( \Gamma \right)$ norm, by $\| \cdot \|_{f/p}$ the $L^2\left(\Omega_{f/p} \right)$ norms, and by $\| \cdot \|_{H^k\left(\Omega_{f/p}\right)}$ the $\mathbf{H}^k\left(\Omega_{f/p} \right)$ norm. For simplicity, we sometimes omit the subscript $f/p$. The trace and Poincar$\acute{e}$ inequalities are introduced as follows:
\begin{equation}\label{poincare}
	\begin{aligned}
		\| \phi \|_{\Gamma} \leq C\left(\Omega_p\right) \sqrt{ \| \phi \|_p  \| \nabla \phi \|_p } , \quad  \| \mathbf{u} \|_{\Gamma} \leq C\left(\Omega_f \right) \sqrt{ \| \mathbf{u} \|_f  \| \nabla \mathbf{u} \|_f } ,  \\
		\| \phi \|  \leq C_{p,p}\| \nabla \phi \| , \quad \forall \phi\in X_p , \quad \| \mathbf{u} \|  \leq C_{p,f}\| \nabla \mathbf{u} \|, \quad \forall \mathbf{u}\in X_f,
	\end{aligned}
\end{equation}
where $C\left(\Omega_{f/p}\right)$, $C_{p,f}$ and $C_{p,p}$ are positive constants depending on the domain $\Omega_{f/p}$.

For simplicity, we also introduce the following notations:
\begin{equation}
	\overline{\mathrm{K}} = \frac{1}{J}\sum_{j=1}^{J} \mathrm{K}_j, \quad \eta _{i,j} = \frac{\alpha _{BJ} \nu \sqrt{d}}{\sqrt{ \tau_i \cdot \mathrm{K}_j \cdot \tau_i } } , \quad \text{and} \quad \overline{\eta}_i = \frac{1}{J}\sum_{j=1}^{J} \eta _{i,j}.
\end{equation}
Then, denote by $|\cdot|_2$ the 2-norm of either vectors or matrices. Let $k_{j,\min}\left(x\right),\overline{k}_{\min}\left(x\right)$ represent the minimum eigenvalue of the hydraulic conductivity tensor $\mathrm{K}_{j}\left(x\right),\mathrm{K}\left(x\right)$ respectively, and $\rho_j^{\prime}\left(x\right)$ be the spectral radius of the fluctuation of hydraulic conductivity tensor $\mathrm{K}_{j}\left(x\right) - \mathrm{K}\left(x\right)$. We can also deduce $| \mathrm{K}_{j}\left(x\right) - \mathrm{K}\left(x\right) |_2 = \rho_j^{\prime}\left(x\right)$ from the symmetry of $\mathrm{K}_{j}\left(x\right)$ and $ \mathrm{K}\left(x\right)$. We need to introduce the following quantities to perform theoretical analysis:
\begin{displaymath}
	\begin{aligned}
		&\rho_j^{\prime \max} = \max_{x\in \Omega_p} \rho_j^{\prime}\left(x\right), \quad \rho^{\prime \max} = \max_{j}  \rho_j^{\prime \max} , \quad \eta_{i,j}^{\prime \max} = \max_{x\in \Omega_p} | \eta_{i,j}\left(x\right) - \overline{\eta}_{i}\left(x\right) |, \\
		& \eta_{i}^{\prime \max} = \max_{j} \eta_{i,j}^{\prime \max}, \quad
		\overline{\eta}_i^{\min} = \min_{x\in \Omega_p} \overline{\eta}_i\left(x\right), \quad \overline{\eta}_i^{\max} = \max_{x\in \Omega_p} \overline{\eta}_i\left(x\right) , \\
		&k_{j,\min} = \min_{x\in \Omega_p} k_{j,\min}\left(x\right) , \quad k_{\min} = \min_{j}k_{j,\min}, \quad \overline{k}_{\min} = \min_{x\in \Omega_p} \overline{k}_{\min}\left(x\right) .
	\end{aligned}
\end{displaymath}

\subsection{The GSAV-GBDF$k$-Ensemble schemes}
Herein, we are devoted to developing the GSAV-GBDF$k$-Ensemble schemes.
We first derive the weak formulation of the unsteady Navier--Stokes--Darcy system (\ref{para_stokes})-(\ref{para_interface}) as follows: find $\mathbf{u}_j \in X_f$, $p_j \in Q_f$ and $\phi_{j} \in X_p$ such that $\forall t \in (0,T]$
\begin{equation}\label{Weak_para_stokes}
	\left\{
	\begin{aligned}	 	
		&\left( \frac{\partial \mathbf{u}_{j}}{\partial t} , \mathbf{v} \right) + a \left(\mathbf{u}_{j} ,\mathbf{v} \right) + d\left(\mathbf{u}_{j},\mathbf{u}_{j},\mathbf{v} \right)  -b\left( \mathbf{v} , p_{j} \right) + \left( g\phi_{j} , \mathbf{v}\cdot \mathbf{n}_f \right)_{\Gamma} \\
		&+ \sum_{i=1}^{d-1} \left( \eta _{i,j} \mathbf{u}_{j}\cdot \tau_i , \mathbf{v}\cdot \tau_i  \right)_{\Gamma}  = \left( \mathbf{f}_{f,j} , \mathbf{v} \right), \quad \forall \mathbf{v} \in X_f,\quad \left( \nabla \cdot \mathbf{u}_{j} ,q   \right)  = 0,  \quad \forall q \in Q_f,
	\end{aligned}
	\right.
\end{equation}
and
\begin{equation}\label{Weak_para_darcy}
	gS\left( \frac{\partial \phi_{j}}{\partial t} , \psi \right) + g\left(  \mathrm{K}_{j}  \nabla \phi_{j} ,\nabla \psi  \right) - \left( g \psi, \mathbf{u}_{j}\cdot \mathbf{n}_f \right)_{\Gamma}  = g\left( f_{p,j} , \psi \right), \quad \forall \psi \in X_p,
\end{equation}
where we define the following notations:
\begin{equation*}
	\begin{aligned}
		a \left(\mathbf{u} ,\mathbf{v} \right) &= \nu \left(\nabla \mathbf{u} , \nabla \mathbf{v} \right) , \quad \forall \mathbf{u},\mathbf{v} \in X_f, \quad
		b\left( \mathbf{v} , p \right) = \left( \nabla \cdot \mathbf{v}, p \right) , \quad \forall \mathbf{v} \in X_f, p \in Q_f, \\
		C_{\Gamma} \left( \mathbf{v}, \psi \right) &= \left( g\psi, \mathbf{v}\cdot \mathbf{n}_f \right)_{\Gamma}, \quad \forall \mathbf{v} \in X_f,\psi \in X_p,  \\
		d\left(\mathbf{u},\mathbf{v},\mathbf{w} \right) &= \left( \left( \mathbf{u} \cdot \nabla  \right) \mathbf{v}, \mathbf{w} \right) -\frac{1}{2}\int_{\Gamma} \left( \mathbf{u} \cdot \mathbf{v} \right) \left( \mathbf{w} \cdot \mathbf{n}_f \right) \mathrm{d}s, \quad \forall \mathbf{u},\mathbf{v},\mathbf{w} \in X_f.
	\end{aligned}
\end{equation*}
%\begin{equation*}
%	\begin{aligned}
	%		a_{f}^j \left(\mathbf{u}_f^{j} ,\mathbf{v} \right) &= \nu \left(\nabla \mathbf{u}_f^j , \nabla \mathbf{v} \right)_{\Omega_f}   \\
	%		a_{p}^j \left( \phi ^j ,\psi \right) &= g\left( \mathrm{K}_j \nabla \phi ^j , \nabla \psi  \right)_{\Omega_p}
	%	\end{aligned}
%\end{equation*}

Furthermore, taking $\mathbf{v} = \mathbf{u}_{j}$ and $\psi = \phi_j$ in (\ref{Weak_para_stokes}) and (\ref{Weak_para_darcy}), and summing the two equations, it yields the following energy equation:
\begin{equation}\label{energy_equ}
	\begin{aligned}
		\frac{\mathrm{d}}{\mathrm{d}t} \left(  \frac{1}{2} \| \mathbf{u}_{j} \|^2 + \frac{1}{2}gS\| \phi_j \|^2 \right) & + \nu \| \nabla \mathbf{u}_{j} \|^2 + g\| \sqrt{ \mathrm{K}_j } \nabla \phi _j \|^2 + \sum_{i=1}^{d-1} \| \sqrt{ \eta_{i,j} } \mathbf{u}_{j} \cdot \tau_i \|^2_{ \Gamma } = \left( \mathbf{f}_{f,j} , \mathbf{u}_{j} \right) + g\left( f_{p,j} , \phi _j \right).
	\end{aligned}
\end{equation}

To introduce the SAV method, we define $E_j\left( \mathbf{u}_{j} , \phi _j \right) = \frac{1}{2} \| \mathbf{u}_{j} \|^2 + \frac{1}{2}gS\| \phi _j \|^2$ and $r_j\left( t \right) = E_j\left( \mathbf{u}_{j} , \phi _j \right) + C_{r,j} >0 $. Consider the following auxiliary variable equation:
\begin{equation}\label{SAV_equation}
	\begin{aligned}
		&\frac{\mathrm{d} r_j\left(t\right)}{\mathrm{d}t} + \gamma _j r_j\left(t\right) = - \Big( \nu \| \nabla \mathbf{u}_{j} \|^2 + g\| \sqrt{ \mathrm{K}_j } \nabla \phi _j \|^2 + \sum_{i=1}^{d-1} \| \sqrt{ \eta_{i,j} } \mathbf{u}_{j} \cdot \tau_i \|^2_{ \Gamma }    - \left( \mathbf{f}_{f,j} , \mathbf{u}_{j} \right) - \left( gf_{p,j} , \phi _j \right) \\
		& + \frac{\alpha _j}{2} \left( \| \mathbf{f}_{f,j} \|^2 + \| g f_{p,j} \|^2 \right)    - \frac{\gamma_j}{2}\left( \| \mathbf{u}_{j} \|^2 + gS\| \phi _j \|^2 \right) \Big) + \frac{\alpha _j}{2} \left( \| \mathbf{f}_{f,j} \|^2 + \| g f_{p,j} \|^2 \right) + \gamma_j C_{r,j},
	\end{aligned}
\end{equation}
where $\gamma _j$ and $\alpha _j$ are positive constants to be specified below. They play a significant role in developing the long-time stable scheme.

Next, we propose the GSAV-GBDF$k$-Ensemble algorithm as follows:
find $\overline{\mathbf{u}}_j^{n+1}$, $\mathbf{u}_j^{n+1}$, $p_j^{n+1}$, $\overline{\phi}_j^{n+1}$, and $\phi_j^{n+1}  $ such that $\forall \mathbf{v}\in X_f,q\in Q_f, \psi \in X_p$
\begin{small}
\begin{equation}\label{ensemble_subsys1}
	\left\{
	\begin{aligned}
		&\left( \frac{ \alpha _k^{\beta} \overline{\mathbf{u}}_{j}^{n+1} + \overline{A}_k^{\beta} \left( \mathbf{u}_{j}^{n} \right) }{ \Delta t } , \mathbf{v} \right)  + \sum_{i=1}^{d-1} \left( \overline{ \eta }_{i}B_k^{\beta}\left( \overline{\mathbf{u}}_{j}^{n+1} \right) \cdot \tau_i , \mathbf{v}
		\cdot \tau_i  \right)_{\Gamma}  + \sum_{i=1}^{d-1} \left( \left(  \eta_{i,j} - \overline{ \eta }_{i} \right) C_k^{\beta}\left( \overline{\mathbf{u}}_{j}^{n} \right) \cdot \tau_i , \mathbf{v}
		\cdot \tau_i  \right)_{\Gamma} \\
		&+ a\left( B_k^{\beta}\left( \overline{\mathbf{u}}_{j}^{n+1} \right) , \mathbf{v} \right) + d\left( C_k^{\beta}\left( \overline{\mathbf{u}}_{j}^{n} \right) , C_k^{\beta}\left( \overline{\mathbf{u}}_{j}^{n} \right) ,\mathbf{v} \right)  -b\left( \mathbf{v} , B_k^{\beta}\left( p_{j}^{n+1} \right) \right) + \left( gC_k^{\beta}\left( \overline{\phi}_{j}^{n} \right) ,\mathbf{v} \cdot \mathbf{n}_f \right)_{\Gamma} = \left( \mathbf{f}_{f,j}^{n+\beta} , \mathbf{v} \right), \\
		& b\left( B_k^{\beta}\left( \overline{\mathbf{u}}_{j}^{n+1} \right) , q \right) = 0,
	\end{aligned}
	\right.
\end{equation}
\end{small}
\begin{equation}\label{ensemble_subsys2}
	\left\{
	\begin{aligned}
		&gS\left( \frac{ \alpha _k^{\beta} \overline{\phi}_{j}^{n+1} + \overline{A}_k^{\beta} \left( \phi_{j}^{n} \right) }{ \Delta t } , \psi \right) + g\left( \overline{ \mathrm{K} } \nabla B_k^{\beta} \left( \overline{\phi}_{j}^{n+1} \right)  ,\nabla \psi \right) \\
		&+ g\left( \left( \mathrm{K}_j - \overline{\mathrm{K}} \right) \nabla C_k^{\beta} \left( \overline{\phi}_{j}^{n} \right), \nabla \psi \right)  - \left( g\psi , C_k^{\beta}\left( \overline{\mathbf{u}}_{j}^{n} \right) \cdot \mathbf{n}_f \right)_{\Gamma}  = g\left( f_{p,j}^{n+\beta} , \psi \right)  ,
	\end{aligned}
	\right.
\end{equation}

\begin{equation}\label{ensemble_subsys3}
	\left\{
	\begin{aligned}
		\frac{ r_j^{n+1} - r_j^{n} }{ \Delta t }& + \gamma _j r_j^{n+1} = - \frac{ r_j^{n+1} }{ E_j\left( \overline{\mathbf{u}}_{j}^{n+1} , \overline{\phi}_{j}^{n+1} \right) + C_{r,j} }  \left\{ \nu \| \nabla \overline{\mathbf{u}}_{j}^{n+1} \|^2 + g\| \sqrt{ \mathrm{K}_j } \nabla \overline{\phi}_{j}^{n+1} \|^2   \right. \\
		& + \sum_{i=1}^{d-1} \| \sqrt{ \eta_{i,j} } \overline{\mathbf{u}}_{j}^{n+1} \cdot \tau_i \|^2_{ \Gamma } - \left( \mathbf{f}_{f,j}^{n+1} , \overline{\mathbf{u}}_{j}^{n+1} \right) - \left( gf_{p,j}^{n+1} , \overline{\phi}_{j}^{n+1} \right) + \frac{\alpha _j}{2} \left( \| \mathbf{f}_{f,j}^{n+1} \|^2 + \| g f_{p,j}^{n+1} \|^2 \right)  \\
		& \left.  - \frac{\gamma_j}{2}\left( \| \overline{\mathbf{u}}_{j}^{n+1} \|^2 + gS\| \overline{\phi}_{j}^{n+1} \|^2 \right)  \right\} + \frac{\alpha _j}{2} \left( \| \mathbf{f}_{f,j}^{n+1} \|^2 + \| g f_{p,j}^{n+1} \|^2 \right) + \gamma_j C_{r,j} ,
	\end{aligned}
	\right.
\end{equation}

\begin{equation}\label{ensemble_subsys4}
	\xi _j^{n+1}  = \frac{ r_j^{n+1} }{ E_j\left( \overline{\mathbf{u}}_{j}^{n+1} , \overline{\phi}_{j}^{n+1} \right) + C_{r,j} } , \quad
	\mathbf{u}_{j}^{n+1} = \eta_{k,j}^{n+1} \overline{\mathbf{u}}_{j}^{n+1}, \quad
	\phi_{j}^{n+1} = \eta_{k,j}^{n+1} \overline{\phi}_{j}^{n+1}, \quad
	\eta_{k,j}^{n+1} = 1 - \left( 1 - \xi _j^{n+1} \right)^{k+1} ,
\end{equation}
where $\beta >1$, $\alpha_k^{\beta}$, the linear operators $A_k^{\beta}$, $\overline{A}_k^{\beta}$, $B_k^{\beta}$ and $C_k^{\beta} \left(k=2,3,4\right)$ are given as follows:

\noindent second-order:
\begin{equation}
	\begin{aligned}
		&\overline{A}_2^{\beta}\left( \mathbf{u}^{n} \right) = -2\beta \mathbf{u}^{n} + \frac{2\beta -1 }{2} \mathbf{u}^{n-1}; \alpha_2^{\beta} = \frac{2\beta + 1}{2},\quad  A_2^{\beta}\left( \mathbf{u}^{n+1} \right) = \alpha_2^{\beta}\mathbf{u}^{n+1} + \overline{A}_2^{\beta}\left( \mathbf{u}^{n} \right);   \\
		&B_2^{\beta}\left( \mathbf{u}^{n} \right) = \beta \mathbf{u}^{n} - \left(\beta - 1 \right)\mathbf{u}^{n-1}, \quad C_2^{\beta}\left( \mathbf{u}^{n} \right) = \left( \beta + 1 \right) \mathbf{u}^{n} - \beta \mathbf{u}^{n-1}.
	\end{aligned}
\end{equation}
third-order:
\begin{equation}\label{3rd-order}
	\begin{aligned}
		&\overline{A}_3^{\beta}\left( \mathbf{u}^{n} \right) = \frac{-\left(  9\beta ^2 + 12\beta -3 \right) }{6} \mathbf{u}^{n} + \frac{9\beta ^2 + 6\beta - 6 }{6} \mathbf{u}^{n-1} + \frac{-\left( 3\beta ^2 -1 \right)}{6}\mathbf{u}^{n-2};\\
		&\alpha_3^{\beta} = \frac{3\beta ^2 + 6\beta + 2}{6}, \quad A_3^{\beta}\left( \mathbf{u}^{n+1} \right) = \alpha_3^{\beta}\mathbf{u}^{n+1} + \overline{A}_3^{\beta}\left( \mathbf{u}^{n} \right);   \\
		& B_3^{\beta}\left( \mathbf{u}^{n} \right) = \frac{\beta ^2 + \beta}{2} \mathbf{u}^{n} - \left(\beta ^2 - 1 \right)\mathbf{u}^{n-1} + \frac{\beta ^2 - \beta}{2}\mathbf{u}^{n-2}; \\
& C_3^{\beta}\left( \mathbf{u}^{n} \right) = \frac{\beta ^2 + 3\beta +2}{2} \mathbf{u}^{n} - \left( \beta ^2 + 2\beta \right)\mathbf{u}^{n-1} + \frac{\beta ^2 + \beta}{2}\mathbf{u}^{n-2}.
	\end{aligned}
\end{equation}
fourth-order:
\begin{equation}
	\begin{aligned}
		\overline{A}_4^{\beta}\left( \mathbf{u}^{n} \right) &= \frac{ -8\beta^3 -30\beta^2 -20\beta + 10 }{12} \mathbf{u}^{n} + \frac{ 12\beta ^3 + 36\beta ^2 + 6\beta - 18 }{12} \mathbf{u}^{n-1}  \\
		&  + \frac{  -8\beta^3-18\beta^2 + 4\beta + 6 }{12}\mathbf{u}^{n-2} + \frac{2\beta^3 +3\beta^2-\beta-1}{12}\mathbf{u}^{n-3};\\
		&\alpha_4^{\beta} = \frac{ 2\beta ^3 + 9\beta ^2 + 11\beta +3 }{12}, \quad A_4^{\beta}\left( \mathbf{u}^{n+1} \right) = \alpha_4^{\beta}\mathbf{u}^{n+1} + \overline{A}_4^{\beta}\left( \mathbf{u}^{n} \right);   \\
		B_4^{\beta}\left( \mathbf{u}^{n} \right) &= \frac{ \beta^3 + 3\beta^2 +2\beta }{6} \mathbf{u}^{n} + \frac{-\beta^3-2\beta^2+\beta+2}{2}\mathbf{u}^{n-1}  + \frac{ \beta ^3 + \beta ^2 - 2\beta}{2}\mathbf{u}^{n-2} + \frac{-\beta^3 + \beta }{6}\mathbf{u}^{n-3}; \\
		C_4^{\beta}\left( \mathbf{u}^{n} \right) &= \frac{ \beta^3 + 6\beta^2 + 11\beta + 6 }{6} \mathbf{u}^{n} + \frac{-\beta^3-5\beta^2-6\beta}{2}\mathbf{u}^{n-1}  + \frac{\beta^3 + 4\beta ^2 + 3\beta}{2}\mathbf{u}^{n-2} + \frac{-\beta^3-3\beta^2-2\beta}{6}\mathbf{u}^{n-3}.
	\end{aligned}
\end{equation}

Herein, we employ the generalized BDF$k$ scheme to discretize the time derivative, which was originally proposed in \cite{huang2024new} and allows the use of a larger time step compared with the classical BDF$k$ scheme. Additionally, the generalized BDF$k$ scheme plays an important role in constructing the high-order decoupled stable scheme for solving the Navier-Stokes equation \cite{huang2023stability,huang2025stability}.

\subsection{Several useful lemmas}
In this subsection, we present several useful lemmas, which will be used in the following proofs.
\begin{lemma}\label{gronwall1}
	(discrete Gronwall lemma 1 \cite{shen1990long}). Let $a^n,b^n$ be two positive sequences satisfying
	\begin{displaymath}
		\frac{a^{n+1}-a^n}{\Delta t} + \lambda a^{n+1} \leq b^n \quad \text{and} \quad b^n \leq b,\forall n\geq 0.
	\end{displaymath}
	Assume that $\Delta t>0$, $1+ \lambda \Delta t >0$. Then
	\begin{displaymath}
		a^n \leq \frac{1}{ (1+ \lambda\Delta t)^n }a^0 + \frac{ 1 + \lambda \Delta t }{ \lambda }\left( 1 - \frac{1}{\left(1+ \lambda \Delta t \right)^{n+1}} \right)b,\, \forall n \geq 0.
	\end{displaymath}
\end{lemma}

\begin{lemma}\label{gronwall2}
	(discrete Gronwall lemma 2 \cite{shen1990long}). Let $a^k,b^k,c^k,d^k$ be four nonnegative sequences satisfying
	\begin{displaymath}
		a^n + \Delta t \sum_{k=0}^{n}b^k \leq B + \Delta t \sum_{k=0}^{n}\left( c^k a^k + d^k \right) \quad \text{with} \quad \Delta t \sum_{k=0}^{T/\Delta t}c^k \leq M,\,\forall 0\leq n \leq T/\Delta t.
	\end{displaymath}
	Assume that $\Delta t c^k <1,\, \forall k$, and let $\Sigma = \max_{0\leq k \leq T/\Delta t}\left( 1-\Delta t c^k \right)^{-1}$. Then
	\begin{displaymath}
		a^n + \Delta t \sum_{k=1}^{n}b^k \leq \exp\left(\Sigma M\right)\left( B + \Delta t \sum_{k=0}^{n}d^k \right)\, \forall n \leq T/\Delta t.
	\end{displaymath}
\end{lemma}

\begin{lemma}\label{gronwall3}
	(discrete Gronwall lemma 3 \cite{shen1990long}). Let $a^k,b^k,c^k,d^k$ be four nonnegative sequences satisfying
	\begin{displaymath}
		a^n + \Delta t \sum_{k=1}^{n}b^k \leq \Delta t \sum_{k=1}^{n-1} c^k a^k + \Delta t \sum_{k=1}^{n-1}d^k + C , n \geq 1,
	\end{displaymath}
	where $\Delta t$ and $C$ are two positive constants. Then
	\begin{displaymath}
		a^n + \Delta t \sum_{k=1}^{n}b^k \leq \exp\left(\Delta t \sum_{k=1}^{n-1} c^k \right)\left( \Delta t \sum_{k=1}^{n-1}d^k + C \right),\, n \geq 1.
	\end{displaymath}
\end{lemma}

\begin{lemma}\label{gronwall4}
	Let $\alpha \left( \varsigma \right) = \alpha_q \varsigma ^q + \cdots + \alpha_0$ and $\mu \left(\varsigma \right) = \mu _q \varsigma ^q + \cdots + \mu_0 $ be polynomials of degree at most $q$ (and at least one of them of degree $q$) that have no common divisors. Denote by $\left(\cdot,\cdot\right)$ an inner product with associated norm $|\cdot|$. If $\text{Re}\frac{ \alpha \left( \varsigma \right) }{ \mu \left(\varsigma \right) } > 0 \quad \text{for} \quad |\varsigma| >1,$
	then there exists a symmetric positive definite matrix $G = \left(g_{ij}\right)\in \mathbb{R}^{q\times q}$ and real $\delta_0,\cdots,\delta _q$ such that for $v^0,\cdots,v^q$ in the inner product space,
	\begin{displaymath}
		\left( \sum_{i=0}^{q}\alpha_iv^i,\sum_{j=0}^{q}\mu_j v^j \right) = \sum_{i,j=1}^{q}g_{ij}\left(v^i,v^j \right) - \sum_{i,j=1}^{q}g_{ij}\left( v^{i-1},v^{j-1}  \right) +  \big|\sum_{i=0}^{q}\delta _i v^i\big|^2.
	\end{displaymath}
\end{lemma}

The following lemma is based on the theoretical results in \cite{huang2024new}, and plays a key role in our proofs.
\begin{lemma}\label{gronwall5} \cite{huang2024new}
	$B_k^{\beta}\left(\mathbf{u}^{n+1}\right)$ can be split into two parts as follows:
%	\begin{equation*}
%		B_k^{\beta}\left(\mathbf{u}^{n+1}\right) = \tau_k^{\beta}C_k^{\beta}\left(\mathbf{u}^{n+1}\right) + D_k^{\beta}\left(\mathbf{u}^{n+1}\right) + F_k^{\beta}\left(\mathbf{u}^{n+1}\right), \quad \text{for}\quad k=2,3,4,
%	\end{equation*}
\begin{equation*}
	B_k^{\beta}\left(\mathbf{u}^{n+1}\right) = \tau_k^{\beta}C_k^{\beta}\left(\mathbf{u}^{n+1}\right) + D_k^{\beta}\left(\mathbf{u}^{n+1}\right), \quad \text{for}\quad k=2,3,4,
\end{equation*}
	with $\tau_k^{\beta}$ given as
 $$\tau_2^{\beta}=\frac{\beta-1}{\beta},\quad \tau_3^{\beta}=\frac{\beta-1}{\beta+1},\quad \tau_4^{\beta}=\frac{\beta-1}{\beta+3}.$$
 Denote by $\left(\cdot,\cdot\right)$ an inner product with associated norm $\|\cdot\|$. Then, given $\beta \ge 2$, we have these properties:
	\begin{itemize}
		\item there exists a symmetric positive definite matrix $G_k = \left(g_{ij}\right)\in \mathbb{R}^{k\times k}$ such that
		\begin{displaymath}
			\left( A_k^{\beta}\left( \mathbf{u}^{n+1} \right) , C_k^{\beta}\left( \mathbf{u}^{n+1} \right) \right) \geq \sum_{i,j=1}^{k}g_{ij}\left( \mathbf{u}^{n+1+i-k} , \mathbf{u}^{n+1+j-k} \right) - \sum_{i,j=1}^{k}g_{ij}\left( \mathbf{u}^{n+i-k} , \mathbf{u}^{n+j-k} \right),
		\end{displaymath}
		\item there exists a symmetric positive definite matrix $H_{k} = \left(h_{ij}\right)\in \mathbb{R}^{\left(k-1\right)\times\left(k-1\right)}$ such that
		\begin{displaymath}
			\left( D_k^{\beta}\left( \mathbf{u}^{n+1} \right) , C_k^{\beta}\left( \mathbf{u}^{n+1} \right) \right) \geq \sum_{i,j=1}^{k-1}h_{ij}\left( \mathbf{u}^{n+2+i-k} , \mathbf{u}^{n+2+j-k} \right) - \sum_{i,j=1}^{k-1}h_{ij}\left( \mathbf{u}^{n+1+i-k} , \mathbf{u}^{n+1+j-k} \right).
		\end{displaymath}
%		\item there exists $U_k\left(\mathbf{u}^i,\cdots,\mathbf{u}^{i+2-k} \right)\geq 0,k=2,3,4$ such that
%		\begin{displaymath}
%			\left( F_k^{\beta}\left( \mathbf{u}^{n+1} \right) , C_k^{\beta}\left( \mathbf{u}^{n+1} \right) \right) \geq \kappa_{k} \| \mathbf{u}^{n+1} \|^2 + U_k\left(\mathbf{u}^{n+1},\cdots,\mathbf{u}^{n+3-k} \right) - U_k\left(\mathbf{u}^{n},\cdots,\mathbf{u}^{n+2-k} \right),
%		\end{displaymath}
%		with $\kappa_{k}$ being a suitable positive number, as given in \cite{huang2025stability}.
	\end{itemize}
\end{lemma}

\begin{lemma}\label{gronwall6}
	Define the truncation error for $k=2,3,4$ as follows:
	\begin{equation*}
		\begin{aligned}
			R_{k,g}^{n+1} &= \frac{ A_k^{\beta}\left( g\left(t^{n+1}\right) \right) }{ \Delta t } - \frac{\partial g\left( t^{n+\beta} \right) }{\partial t},
			P_{k,g}^{n+1}  = B_k^{\beta}\left( g\left(t^{n+1}\right) \right) - g\left( t^{n+\beta} \right), \\
			Q_{k,g}^{n+1} &= C_k^{\beta}\left( g\left(t^{n}\right) \right) - g\left( t^{n+\beta} \right) ,
		\end{aligned}
	\end{equation*}
	where $g$ can represent function $\mathbf{u}$, $p$ and $\phi$, respectively. Then, under the suitable regularity assumptions, there exists a positive constant $C$ such that \cite{huang2024new}:
	\begin{equation*}
		\| R_{k,g}^{n+1} \|, \| P_{k,g}^{n+1} \|,  \| Q_{k,g}^{n+1} \| \leq C \left(\Delta t\right)^{k}, \quad \forall n+1 \leq \frac{T}{\Delta t}.
	\end{equation*}
\end{lemma}

\section{Stability analysis}
In this section, we shall prove that the proposed algorithm (\ref{ensemble_subsys1})-(\ref{ensemble_subsys4}) is long-time stable without any time-step constraints.

\begin{theorem}\label{stableTh}
	Assume that $\| \mathbf{f}_{f,j}\left(t,x\right) \|,\| f_{p,j}\left(t,x\right) \| \leq C_{f}$ for some constant $C_{f}>0$, and we choose $C_{r,j}\geq 1$.  Then  there exist suitable constants $\alpha_j, \gamma_j>0$, such that, given $r^n_{j} >0$, we have $r^{n+1}_j>0,\xi_j^{n+1}>0$, and the following holds
	\begin{equation}
		\| \mathbf{u}_{j}^{n+1} \| , \| \phi_{j}^{n+1} \| \leq M_k,\quad \forall n+1\le \frac{T}{\Delta t},
	\end{equation}
	where $M_k$ is a positive constant, not depending on the final time $T$.
\end{theorem}
\begin{proof}
	It follows from (\ref{ensemble_subsys3}) that
	\begin{equation}\label{stable1}
		\begin{aligned}
			&r_j^{n+1} = \Big \{  1 + \gamma _j \Delta t  +  \frac{ \Delta t }{ E_j\left( \overline{\mathbf{u}}_{j}^{n+1} , \overline{\phi}_{j}^{n+1} \right) + C_{r,j} }  \left( \nu \| \nabla \overline{\mathbf{u}}_{j}^{n+1} \|^2 + g\| \sqrt{ \mathrm{K}_j } \nabla \overline{\phi}_{j}^{n+1} \|^2   \right.  \\
			& + \sum_{i=1}^{d-1} \| \sqrt{ \eta_{i,j} } \overline{\mathbf{u}}_{j}^{n+1} \cdot \tau_i \|^2_{ \Gamma } - \left( \mathbf{f}_{f,j}^{n+1} , \overline{\mathbf{u}}_{j}^{n+1} \right) - \left( gf_{p,j}^{n+1} , \overline{\phi}_{j}^{n+1} \right)  + \frac{\alpha _j}{2} \left( \| \mathbf{f}_{f,j}^{n+1} \|^2 + \| g f_{p,j}^{n+1} \|^2 \right)  \\
			& \left.   - \frac{\gamma_j}{2}\left( \| \overline{\mathbf{u}}_{j}^{n+1} \|^2 + gS\| \overline{\phi}_{j}^{n+1} \|^2 \right)  \right)  \Big \}^{-1} \times \left( r_j^n + \frac{\alpha _j}{2} \Delta t \left( \| \mathbf{f}_{f,j}^{n+1} \|^2 + \| g f_{p,j}^{n+1} \|^2 \right) + \Delta t \gamma_j C_{r,j} \right)  .
		\end{aligned}
	\end{equation}
	Since $\nu ,gk_{j,\min}>0$, we can choose $\alpha _j>0$ large enough and $\gamma _j>0$ small enough such that,
$$\rho_{\max} := \left\{ \min \left\{ \nu ,gk_{j,\min} \right\} - \left( \frac{1}{2\alpha_j } + \frac{\gamma_j}{2}\max \left\{ 1,gS \right\} \right)C^2_{p,\max} \right\} >0$$ with $C_{p,\max} = \max \left\{C_{p,f},C_{p,p}\right\}$, where $C_{p,f},C_{p,p}$ are defined in \eqref{poincare}. Then we can obtain,
	\begin{equation}\label{stable2}
		\begin{aligned}
			& \nu \| \nabla \overline{\mathbf{u}}_{j}^{n+1} \|^2 + g\| \sqrt{ \mathrm{K}_j } \nabla \overline{\phi}_{j}^{n+1} \|^2  + \sum_{i=1}^{d-1} \| \sqrt{ \eta_{i,j} } \overline{\mathbf{u}}_{j}^{n+1} \cdot \tau_i \|^2_{ \Gamma } - \left( \mathbf{f}_{f,j}^{n+1} , \overline{\mathbf{u}}_{j}^{n+1} \right)  \\
			&   - \left( gf_{p,j}^{n+1} , \overline{\phi}_{j}^{n+1} \right) + \frac{\alpha _j}{2} \left( \| \mathbf{f}_{f,j}^{n+1} \|^2 + \| g f_{p,j}^{n+1} \|^2 \right)    - \frac{\gamma_j}{2}\left( \| \overline{\mathbf{u}}_{j}^{n+1} \|^2 + gS\| \overline{\phi}_{j}^{n+1} \|^2 \right)  \\
			&\geq  \nu \| \nabla \overline{\mathbf{u}}_{j}^{n+1} \|^2 + g\| \sqrt{ \mathrm{K}_j } \nabla \overline{\phi}_{j}^{n+1} \|^2  + \sum_{i=1}^{d-1} \| \sqrt{ \eta_{i,j} } \overline{\mathbf{u}}_{j}^{n+1} \cdot \tau_i \|^2_{ \Gamma }   - \frac{\alpha_j}{2}\left( \| \mathbf{f}_{f,j}^{n+1} \|^2 + \| gf_{p,j}^{n+1} \|^2 \right)  \\
			&  + \frac{\alpha _j}{2} \left( \| \mathbf{f}_{f,j}^{n+1} \|^2 + \| g f_{p,j}^{n+1} \|^2 \right)   - \frac{1}{2\alpha_j } \left( \| \overline{\mathbf{u}}_{j}^{n+1} \|^2 + \| \overline{\phi}_{j}^{n+1} \|^2 \right)- \frac{\gamma_j}{2}\left( \| \overline{\mathbf{u}}_{j}^{n+1} \|^2 + gS\| \overline{\phi}_{j}^{n+1} \|^2 \right)   \\
			&\geq \Big( \nu \| \nabla \overline{\mathbf{u}}_{j}^{n+1} \|^2 + g\| \sqrt{ \mathrm{K}_j } \nabla \overline{\phi}_{j}^{n+1} \|^2  + \sum_{i=1}^{d-1} \| \sqrt{ \eta_{i,j} } \overline{\mathbf{u}}_{j}^{n+1} \cdot \tau_i \|^2_{ \Gamma }  \\
			&    - \left( \frac{1}{2\alpha_j } + \frac{\gamma_j}{2} \max \left\{ 1,gS \right\} \right) \left( \| \overline{\mathbf{u}}_{j}^{n+1} \|^2 + \| \overline{\phi}_{j}^{n+1} \|^2 \right)  \Big)  \geq  \rho_{\max} \left( \| \nabla \overline{\mathbf{u}}_{j}^{n+1} \|^2 + \| \nabla \overline{\phi}_{j}^{n+1} \|^2 \right).
		\end{aligned}
	\end{equation}

	Then, given $r_j^n >0$ and using (\ref{stable2}), we can immediately obtain $r_j^{n+1}>0$. Furthermore, we can deduce $\xi_j^{n+1}>0$ from the equation (\ref{ensemble_subsys4}).
	
	Using again (\ref{stable2}), we then derive from (\ref{ensemble_subsys3}) that
	\begin{equation}\label{stable3}
		\begin{aligned}
			\frac{ r_j^{n+1} - r_j^{n} }{\Delta t} + \gamma _j r_j^{n+1}   & \leq  \frac{\alpha _j}{2} \left( \| \mathbf{f}_{f,j}^{n+1} \|^2 + \| g f_{p,j}^{n+1} \|^2 \right) + \gamma_j C_{r,j} -\xi_j^{n+1} \rho_{\max} \left( \| \nabla \overline{\mathbf{u}}_{j}^{n+1} \|^2 + \| \nabla \overline{\phi}_{j}^{n+1} \|^2 \right) \\
			& \leq \frac{\alpha _j}{2} \left( \| \mathbf{f}_{f,j}^{n+1} \|^2 + \| g f_{p,j}^{n+1} \|^2 \right) + \gamma_j C_{r,j}  \leq \alpha_jg^2C_f^2 + \gamma_j C_{r,j}.  \\
		\end{aligned}
	\end{equation}
	
	Applying Lemma \ref{gronwall1} to the above, there exists a positive constant $M$, which is independent of $T$, such that
	\begin{displaymath}
		r_{j}^n \leq M,\quad \forall n \leq T/\Delta t,
	\end{displaymath}
	which, along with $C_{r,j} \geq 1$, leads to
	\begin{equation}\label{stable4}
		|\xi_j^{n+1}| = \frac{ r_j^{n+1} }{ E_j\left( \overline{\mathbf{u}}_{j}^{n+1} , \overline{\phi}_{j}^{n+1} \right) + C_{r,j} }  \leq \frac{ 2M }{ \| \overline{\mathbf{u}}_{j}^{n+1} \|^2 + gS\| \overline{\phi}_{j}^{n+1} \|^2 + 2 }.
	\end{equation}
	
	Since $ \eta_{k,j} ^{n+1} = 1 - \left( 1 - \xi_j^{n+1} \right)^{k+1} $, we can give $ \eta_{k,j} ^{n+1} = \xi_j^{n+1}P_k\left( \xi_j^{n+1} \right) $ with $P_k$ being a polynomial of degree $k$. Then, one can deduce from (\ref{stable4}) that there exists $\tilde{M}_k >0$ such that
	\begin{equation}\label{stable5}
		|\eta_{k,j} ^{n+1}| = |\xi_j^{n+1}P_k\left( \xi_j^{n+1} \right)| \leq \frac{ \tilde{M}_k }{ \| \overline{\mathbf{u}}_{j}^{n+1} \|^2 + gS\| \overline{\phi}_{j}^{n+1} \|^2 + 2 },
	\end{equation}
	which, along with $\mathbf{u}_{j}^{n+1} = \eta_{k,j} ^{n+1}\overline{ \mathbf{u}}_{j}^{n+1}$ and $ \phi_{j}^{n+1} = \eta_{k,j} ^{n+1}\overline{ \phi}_{j}^{n+1} $, implies
	\begin{equation}\label{stable6}
		\begin{aligned}
			&\| \mathbf{u}_{j}^{n+1} \|^2 + gS\| \phi_{j}^{n+1} \|^2 = \left(  \eta_{k,j} ^{n+1} \right)^2 \left( \| \overline{ \mathbf{u}}_{j}^{n+1} \|^2 + gS\| \overline{ \phi}_{j}^{n+1} \|^2 \right) \\
			&\leq \left( \frac{ \tilde{M}_k }{ \| \overline{\mathbf{u}}_{j}^{n+1} \|^2 + gS\| \overline{\phi}_{j}^{n+1} \|^2 + 2 }  \right)^2 \left( \| \overline{ \mathbf{u}}_{j}^{n+1} \|^2 + gS\| \overline{ \phi}_{j}^{n+1} \|^2 \right)  \leq \tilde{M}_k^2 .
		\end{aligned}
	\end{equation}
Finally, by denoting $M_k:=\max \left\{\tilde{M}_k,\frac{\tilde{M}_k}{\sqrt{gS}}\right\}$, the proof is complete.
\end{proof}

\section{Error analysis}
Herein, we present error analysis for our GSAV-GBDF$k$-Ensemble scheme under two parameter conditions as follows:
\begin{equation}\label{para_condition}
	\frac{\eta_{i}^{\prime \max}}{ \overline{\eta}_i^{\min} } < \frac{2\tau_k^{\beta}}{3}\quad \text{and} \quad \frac{ \rho^{\prime \max} }{ \overline{k}_{\min} } < \frac{\tau_k^{\beta}}{3}.
\end{equation}
Let us first introduce the following notations:
\begin{displaymath}
	\begin{aligned}
		\overline{\mathbf{e}}_{j,u}^n &= \mathbf{u}_j(t^n) - \overline{\mathbf{u}}_j^n , \mathbf{e}_{j,u}^n = \mathbf{u}_j(t^n) - \mathbf{u}_j^n , \mathbf{e}_{j,p}^n = p_j(t^n) - p_j^n ,\\
		\overline{\mathbf{e}}_{j,\phi}^n &= \phi_j(t^n) - \overline{\phi}_j^n , \mathbf{e}_{j,\phi}^n = \phi_j(t^n) - \phi_j^n .
	\end{aligned}
\end{displaymath}
%\begin{theorem}
%	Given the initial condition $ \overline{\mathbf{u}}^0 = \mathbf{u}^0 = \mathbf{u}\left(0\right)$, $r^0 = E\left(\mathbf{u}^0\right)$. Let $\overline{\mathbf{u}}^{n+1}$ and $\mathbf{u}^{n+1}$ be computed with $k$th-order scheme (\ref{BDFk-SAV1})-(\ref{BDFk-SAV4}) $\left(1\leq k \leq 4\right)$ with
%	\begin{equation}
	%		\eta_1^{n+1} = 1-\left( 1 - \xi^{n+1} \right)^3, \quad 	\eta_k^{n+1} = 1-\left( 1 - \xi^{n+1} \right)^{k+1} \left(k=2,3,4\right).
	%	\end{equation}
%	Assume that the exact solution $\mathbf{u}\left(t\right)$ satisfy the following regularities
%	\begin{equation*}
	%		\mathbf{u} \in L^2\left(0,T;L^2\left(\Omega \right) \right)  \cap \frac{\partial ^j \mathbf{u}}{\partial t^j} \in L^2\left(0,T;H^1\left( \Omega \right) \right) 1\leq j \leq k+1.
	%	\end{equation*}
%	Then for $n \leq T/\Delta t$ and $\Delta t < \min \left( \frac{1}{2C_0^{k+1}} ,\frac{1}{C_0^{k+2}} \right)$, there holds
%	\begin{equation}
	%		\| \mathbf{e}^{n} \| , 	\| \overline{\mathbf{e}}^{n} \| \leq C\left(\Delta t\right)^{k},
	%	\end{equation}
%	where the positive constants $C_0$ and $C$ depend on the final time $T$, $\Omega$ and the exact solution $\mathbf{u}\left(t\right)$ but are not dependent on $\Delta t$.
%\end{theorem}
\begin{theorem}
	Given the initial conditions $\overline{\mathbf{u}}_j^0=\mathbf{u}_j^0=\mathbf{u}_j(0), \,\overline{\phi}_j^0=\phi_j^0=\phi_j(0),
	\,r_j^0=E_j(\mathbf{u}_j^0,\phi_j^0)+C_{r,j},$
	and suppose that
	$
	\left( \overline{\mathbf{u}}_j^{i},\mathbf{u}_j^{i},p_j^{i},
	\overline{\phi}_j^{i},\phi_j^{i} \right), \,\, i=1,\ldots,k-1,
	$
	are computed by a suitable initialization procedure with the required accuracy.
	Let
	$
	\left( \overline{\mathbf{u}}_j^{n+1},\mathbf{u}_j^{n+1},p_j^{n+1},
	\overline{\phi}_j^{n+1},\phi_j^{n+1} \right)$
	be computed by the scheme \eqref{ensemble_subsys1}--\eqref{ensemble_subsys4} with $2\leq k\leq 4$.
	Assume that the exact solutions are sufficiently smooth such that
	\begin{equation*}
		\begin{aligned}
			&\mathbf{u}_j \in L^2(0,T;\mathbf{H}^2), \quad
			\frac{\partial ^k \mathbf{u}_j}{\partial t^k} \in L^2(0,T;\mathbf{H}^2), \quad
			\frac{\partial ^{k+1} \mathbf{u}_j}{\partial t^{k+1}} \in L^2(0,T;L^2), \quad
			\frac{\partial ^k p_j}{\partial t^k} \in L^2(0,T;H^1), \\
			&\phi_j \in L^2(0,T;H^2), \quad
			\frac{\partial ^k \phi_j}{\partial t^k} \in L^2(0,T;H^2), \quad
			\frac{\partial ^{k+1} \phi_j}{\partial t^{k+1}} \in L^2(0,T;L^2).
		\end{aligned}
	\end{equation*}
	Further assume that \eqref{para_condition} holds and choose $\beta\geq 2$ in all cases.
	Then, for sufficiently small $\Delta t$, the following estimate holds:
	\begin{equation*}
		\| e_{j,u}^{n+1} \|^2 + \| e_{j,\phi}^{n+1} \|^2
		+ \Delta t \sum_{i=k-1}^{n}
		\left(
		\nu \| \nabla C_k^{\beta}( e_{j,u}^{i+1} ) \|^2
		+ \| \nabla C_k^{\beta}( e_{j,\phi}^{i+1} ) \|^2
		\right)
		\leq C(\Delta t)^{2k},
		\qquad n+1 \leq T/\Delta t,
	\end{equation*}
	where the constant $C>0$ is independent of $\Delta t$.
\end{theorem}

\begin{proof}
	Suppose that $\mathbf{u}_j^n$, $\overline{\mathbf{u}}_j^{n}$, $p_j^n$,  $\phi_{j}^n$, and $ \overline{\phi}_{j}^n \left( n = 1,\dots,k-1 \right)$ are computed with a proper initialization procedure such that $ \| \mathbf{u}_j\left( t^n \right) - \mathbf{u}_j^n  \|^2 + \Delta t \| \nabla \left( \mathbf{u}_j\left( t^n \right) - \mathbf{u}_j^n \right)  \|^2 = O\left(\Delta t ^{2k} \right) $, $ \| \mathbf{u}_j\left( t^n \right) - \overline{\mathbf{u}}_j^n  \|^2 + \Delta t \| \nabla \left( \mathbf{u}_j\left( t^n \right) - \overline{\mathbf{u}}_j^n \right)  \|^2 = O\left(\Delta t ^{2k} \right) $, $ \| p_j\left( t^n \right) - p_j^n  \| = O\left(\Delta t ^k \right) $, $ \| \phi_{j}\left(t^n\right) - \phi_{j}^n \|^2 + \| \nabla \left( \phi_{j}\left(t^n\right) - \phi_{j}^n \right) \|^2 = O\left(\Delta t ^{2k} \right) $ and $ \| \phi_{j}\left(t^n\right) - \overline{\phi}_{j}^n \|^2 + \| \nabla \left( \phi_{j}\left(t^n\right) - \overline{\phi}_{j}^n \right) \|^2 = O\left(\Delta t ^{2k} \right) \left( n =0, 1,\dots,k-1 \right).$ To simplify the presentation, we employ $ \mathbf{u}_j^n = \overline{\mathbf{u}}_j^n = \mathbf{u}_j\left( t^n \right) $, $ \phi_{j}^n = \overline{\phi}_j^n = \phi_{j}\left(t^n\right) $ and $ r_j^n = E_j\left( \mathbf{u}_j^n ,\phi_{j}^n \right) + C_{r,j} $ for $ \left( n =0, 1,\dots,k-1 \right) $.
	
	The main purpose is to prove
	\begin{equation}\label{error1}
		| 1-\xi_j^n | \leq C_0 \Delta t, \quad \forall n \leq T/\Delta t,
	\end{equation}
	where the constant $C_0>0$ depends on the final time $T$, $\mathbf{f}_{f,j}$, $f_{p,j}$ and the exact solutions $\left( \mathbf{u}_j\left(t\right) , p_j\left(t\right) , \phi_{j}\left( t \right) \right)$ but is not dependent on $\Delta t$, and will be specified in what follows. Next, we shall prove (\ref{error1}) by induction.
	
	Under the assumption, (\ref{error1}) indeed holds for $n =0$. Now suppose there holds
	\begin{equation}\label{error2}
		| 1 - \xi_j ^n | \leq C_0 \Delta t, \quad \forall n \leq m,
	\end{equation}
	we shall prove below
	\begin{equation}\label{error3}
		| 1 - \xi _j^{m+1} | \leq C_0 \Delta t.
	\end{equation}
	%We shall begin with considering $k = 2,3,4$, and point out the necessary modifications for the case $k = 1$ later.
	The process of proof is divided into the following three parts.
	
	Step 1: Bounds for $ \overline{\mathbf{u}}_j^n, \nabla \overline{\mathbf{u}}_j^n, \nabla \mathbf{u}_j^n, \overline{\phi}_{j}^n, \nabla \overline{\phi}_{j}^n , \nabla \phi_{j}^n,  \forall n\leq m$. For the $k$th-order schemes, we deduce from Theorem \ref{stableTh} that
	\begin{equation}\label{error4}
		\| \mathbf{u}_j^n \|, \| \phi_{j}^n \| \leq  M_k,\quad \forall n \leq T / \Delta t.
	\end{equation}
	Under the assumption (\ref{error2}), if we choose $\Delta t$ small enough such that
	\begin{equation}\label{error5}
		\Delta t \leq \min \left\{ \frac{1}{2C_0^{k+1}},1 \right\},
	\end{equation}
	it implies
	\begin{subequations}
		\begin{align}
			&1 - \frac{1}{2C_0^k} \leq | \xi _j^n | \leq 1 + \frac{1}{2C_0^k}, \quad \forall n \leq m , \label{error6}\\
			&| 1 - \eta_{k,j} ^n | \leq \frac{\Delta t ^k}{2} , \quad \frac{1}{2} \leq 1 - \frac{\Delta t^ k}{2} \leq | \eta_{k,j} ^n | \leq 1 + \frac{\Delta t^ k}{2} \leq \frac{3}{2} , \quad \forall n \leq m , \label{error7} \\
			&\| \overline{\mathbf{u}}_j^n \| = \frac{ \| \mathbf{u}^n \| }{ | \eta_{k,j} ^n | } \leq 2M_k, \| \overline{\phi}_j^n \| = \frac{ \| \phi_{j}^n \| }{ | \eta_{k,j}^n | } \leq 2M_k, \quad \forall n \leq m . \label{error8}
		\end{align}
	\end{subequations}
	%	\begin{equation}\label{error6}
		%		1 - \frac{1}{2C_0^k} \leq | \xi _j^n | \leq 1 + \frac{1}{2C_0^k}, \forall n \leq m ,
		%	\end{equation}
	%	and
	%	\begin{equation}\label{error7}
		%		| 1 - \eta_{k,j} ^n | \leq \frac{\Delta t ^k}{2} , \quad \frac{1}{2} \leq 1 - \frac{\Delta t^ k}{2} \leq | \eta_{k,j} ^n | \leq 1 + \frac{\Delta t^ k}{2} \leq \frac{3}{2}  , \forall n \leq m ,
		%	\end{equation}
	%	and
	%	\begin{equation}\label{error8}
		%		\| \overline{\mathbf{u}}_j^n \| = \frac{ \| \mathbf{u}^n \| }{ | \eta_{k,j} ^n | } \leq 2M_k, \| \overline{\phi}_j^n \| = \frac{ \| \phi_{j}^n \| }{ | \eta_{k,j}^n | } \leq 2M_k, \quad \forall n \leq m .
		%	\end{equation}
	
	Moreover, summing (\ref{ensemble_subsys3}) for $n$ from 0 to $m-1$ and using again (\ref{stable2}),(\ref{error7}), it implies
	\begin{subequations}
		\begin{align}
			& \Delta t \rho_{\max} \sum_{n=0}^{m-1} \left( \| \nabla \overline{\mathbf{u}}_{j}^{n+1} \|^2 + \| \nabla \overline{\phi}_{j}^{n+1} \|^2 \right) \leq \frac{ r^0 + CC_f^2T + CC_{r,j}T }{ \min| \xi_j^n | }   \leq 2\left( r^0 + CC_f^2T + CC_{r,j}T \right) =: C_T ,  \label{error9} \\
			&\Delta t \rho_{\max} \sum_{n=0}^{m-1} \left( \| \nabla \mathbf{u}_{j}^{n+1} \|^2 + \| \nabla \phi_{j}^{n+1} \|^2 \right)  \leq 4 C_T , \quad C_0 \geq 1. \label{error10}
		\end{align}
	\end{subequations}
	%	\begin{equation}\label{error9}
		%		\begin{aligned}
			%			\Delta t \rho_{\max} \sum_{n=0}^{m-1} \left( \| \nabla \overline{\mathbf{u}}_{j}^{n+1} \|^2 + \| \nabla \overline{\phi}_{j}^{n+1} \|^2 \right) & \leq \frac{ r^0 + CC_f^2T + CC_{r,j}T }{ \min| \xi_j^n | }   \leq 2\left( r^0 + CC_f^2T + CC_{r,j}T \right) \equiv C_T , \quad C_0 \geq 1,
			%		\end{aligned}
		%	\end{equation}
	%	and
	%	\begin{equation}\label{error10}
		%		\begin{aligned}
			%			&\Delta t \rho_{\max} \sum_{n=0}^{m-1} \left( \| \nabla \mathbf{u}_{j}^{n+1} \|^2 + \| \nabla \phi_{j}^{n+1} \|^2 \right)  \leq 4 C_T , \quad C_0 \geq 1.
			%		\end{aligned}
		%	\end{equation}
	
	Step 2: Estimates for $ \overline{\mathbf{e}}_{j,u}^{m+1} , \overline{\mathbf{e}}_{j,\phi}^{m+1} , \nabla \overline{\mathbf{e}}_{j,u}^{m+1}, \nabla \overline{\mathbf{e}}_{j,\phi}^{m+1} $. By the assumptions on the exact solutions $\mathbf{u}_j\left(t\right), \phi_{j}\left(t\right)$ and (\ref{error8}), there exists a large enough constant $C_1>0$ such that
	\begin{equation}\label{error11}
		\| \nabla \mathbf{u}_j\left( t \right) \|, \| \nabla \phi_j\left( t \right) \| \leq C_1 \quad \forall t \leq T, \quad \| \overline{\mathbf{u}}_j^n \|, \| \overline{\phi}_j^n \| \leq C_1  \quad \forall n \leq m.
	\end{equation}
	
	We first discrete the equations (\ref{Weak_para_stokes}) and (\ref{Weak_para_darcy}) at time $t^{n+\beta}$, it then yields:
\begin{small}
	\begin{equation}\label{error12_1}
		\begin{aligned}
			&\left( \frac{ A_k^{\beta}\left( \mathbf{u}_j\left(t^{n+1}\right) \right) }{ \Delta t } , \mathbf{v} \right)  + a\left( B_k^{\beta}\left( \mathbf{u}_j\left(t^{n+1}\right) \right) , \mathbf{v} \right) + \sum_{i=1}^{d-1}\left( \overline{\eta}_i B_k^{\beta}\left( \mathbf{u}_j\left(t^{n+1}\right) \right) \cdot \tau_i , \mathbf{v}\cdot \tau_i \right)_{\Gamma}  \\
			& + \sum_{i=1}^{d-1}\left( \left( \eta_{i,j} - \overline{\eta}_{i} \right) C_k^{\beta}\left( \mathbf{u}_j\left(t^{n}\right) \right) \cdot \tau_i , \mathbf{v}\cdot \tau_i \right)_{\Gamma} + d\left( C_k^{\beta}\left( \mathbf{u}_j\left(t^{n}\right) \right) , C_k^{\beta}\left( \mathbf{u}_j\left(t^{n}\right) \right) , \mathbf{v} \right) \\
			& + b\left( \mathbf{v} , B_k^{\beta}\left( p_j\left( t^{n+1} \right) \right) \right) + C_{\Gamma} \left( \mathbf{v} , C_k^{\beta}\left( \phi_j\left(t^{n}\right) \right)  \right) \\
			&= \left( \mathbf{f}_{f,j} ^{n+\beta}  , \mathbf{v} \right) + \left( R_{k,u}^{n+1},\mathbf{v} \right) + a\left( P_{k,u}^{n+1} ,\mathbf{v} \right) + b\left( \mathbf{v} , P_{k,p}^{n+1} \right) +  C_{\Gamma}\left( \mathbf{v}, Q_{k,\phi}^{n+1}  \right)  + \sum_{i=1}^{d-1}\left( \eta_{i,j} P_{k,u}^{n+1} \cdot \tau_i , \mathbf{v} \cdot \tau_i \right) \\
			& - \sum_{i=1}^{d-1} \left( \left( \eta_{i,j} - \overline{\eta}_{i} \right) \left( P_{k,u}^{n+1} - Q_{k,u}^{n+1}  \right) \cdot \tau_i , \mathbf{v}\cdot \tau_i \right)  + d\left( C_k^{\beta}\left( \mathbf{u}_j\left(t^{n}\right) \right) , C_k^{\beta}\left( \mathbf{u}_j\left(t^{n}\right) \right) , \mathbf{v} \right) - d\left( \mathbf{u}_j\left(t^{n+\beta}\right) , \mathbf{u}_j\left(t^{n+\beta}\right) , \mathbf{v} \right),
		\end{aligned}
	\end{equation}
\end{small}
	and
\begin{small}
	\begin{equation}\label{error12_2}
		\begin{aligned}
			& gS\left( \frac{ A_k^{\beta} \left( \phi_{j}\left(t^{n+1} \right) \right) }{ \Delta t } , \psi \right) + g\left(\overline{\mathrm{K}} \nabla B_k^{\beta}\left( \phi_{j}\left(t^{n+1} \right) \right) , \nabla \psi \right) + g\left( \left(\mathrm{K}_j - \overline{\mathrm{K}} \right) \nabla C_k^{\beta}\left( \phi_{j}\left(t^{n} \right) \right) , \nabla \psi \right)  - C_{\Gamma} \left( C_k^{\beta}\left( \mathbf{u}_j\left(t^{n}\right) \right) ,  \psi \right) \\
			& = \left( gf_{p,j}^{n+\beta} , \psi \right) + \left( gR_{k,\phi}^{n+1} , \psi \right) - C_{\Gamma}\left( Q_{k,u}^{n+1} , \psi \right) + g\left( \mathrm{K}_j \nabla P_{k,\phi}^{n+1}   ,  \nabla \psi \right) - g\left(  \left( \mathrm{K}_j - \overline{\mathrm{K}} \right) \nabla \left( P_{k,\phi}^{n+1} - Q_{k,\phi}^{n+1} \right) , \nabla \psi \right) ,
		\end{aligned}
	\end{equation}
\end{small}
	where
\begin{small}
	\begin{equation}\label{error13}
		\begin{aligned}
			R_{k,u}^{n+1} &=  \frac{ A_k^{\beta}\left( \mathbf{u}_j\left(t^{n+1}\right) \right) }{ \Delta t } - \frac{\partial \mathbf{u}_j\left( t^{n+\beta} \right) }{\partial t} ,  \quad   R_{k,\phi}^{n+1} =  \frac{ A_k^{\beta}\left( \phi_j\left(t^{n+1}\right) \right) }{ \Delta t } - \frac{\partial \phi_j\left( t^{n+\beta} \right) }{\partial t} ,  \\
			P_{k,u}^{n+1}  &= B_k^{\beta}\left( \mathbf{u}_j\left(t^{n+1}\right) \right) - \mathbf{u}_j\left( t^{n+\beta} \right) , \quad P_{k,\phi}^{n+1}  = B_k^{\beta}\left( \phi_j\left(t^{n+1}\right) \right) - \phi_j\left( t^{n+\beta} \right) , \\
			Q_{k,u}^{n+1} &= C_k^{\beta}\left( \mathbf{u}_j\left(t^{n}\right) \right) - \mathbf{u}_j\left( t^{n+\beta} \right), \quad Q_{k,\phi}^{n+1} = C_k^{\beta}\left( \phi_j\left(t^{n}\right) \right) - \phi_j\left( t^{n+\beta} \right) , \quad
			P_{k,p}^{n+1} = B_k^{\beta}\left( p_j\left(t^{n+1}\right) \right) - p_j\left( t^{n+\beta} \right) . \\
		\end{aligned}
	\end{equation}
	\end{small}
	To establish the error equation, subtracting (\ref{error12_1}) and (\ref{error12_2}) from (\ref{ensemble_subsys1}) and (\ref{ensemble_subsys2}), respectively, and multiplying by $\Delta t$, we can then get
 \begin{small}
	\begin{equation}\label{error14_1}
		\begin{aligned}
			&\left( A_k^{\beta}\left( \overline{\mathbf{e}}_{j,u}^{n+1} \right) , \mathbf{v} \right)  + \Delta t a\left( B_k^{\beta}\left( \overline{\mathbf{e}}_{j,u}^{n+1} \right) , \mathbf{v} \right) + \Delta t \sum_{i=1}^{d-1} \left( \overline{\eta}_{i} B_k^{\beta}\left( \overline{\mathbf{e}}_{j,u}^{n+1} \right)\cdot \tau_i , \mathbf{v}\cdot \tau_i \right)_{\Gamma}  + \Delta t b\left( \mathbf{v} ,  B_k^{\beta}\left( \mathbf{e}_{j,p}^{n+1} \right) \right)     \\
			&+ \Delta tC_{\Gamma} \left( \mathbf{v} , C_k^{\beta}\left(  \overline{\mathbf{e}}_{j,\phi}^{n} \right)  \right) = \left( \overline{A}_k^\beta \left( \overline{\mathbf{u}}_j^n - \mathbf{u}_j^n \right) , \mathbf{v} \right) + \Delta t \left( R_{k,u}^{n+1},\mathbf{v} \right) + \Delta t a\left( P_{k,u}^{n+1} ,\mathbf{v} \right) + \Delta t b\left( \mathbf{v} , P_{k,p}^{n+1} \right)  \\
			& + \Delta t C_{\Gamma}\left( \mathbf{v}, Q_{k,\phi}^{n+1}  \right) - \Delta t \sum_{i=1}^{d-1}\left( \left( \eta_{i,j} - \overline{\eta}_{i} \right) C_k^{\beta}\left( \overline{\mathbf{e}}^n_{j,u} \right) \cdot \tau_i , \mathbf{v}\cdot \tau_i \right)_{\Gamma}    + \Delta t \sum_{i=1}^{d-1}\left( \eta_{i,j} P_{k,u}^{n+1} \cdot \tau_i , \mathbf{v} \cdot \tau_i \right) \\
			&- \Delta t \sum_{i=1}^{d-1} \left( \left( \eta_{i,j} - \overline{\eta}_{i} \right) \left( P_{k,u}^{n+1} - Q_{k,u}^{n+1}  \right) \cdot \tau_i , \mathbf{v}\cdot \tau_i \right)  + \Delta t d\left( C_k^{\beta}\left( \overline{\mathbf{u}}_j^n \right) , C_k^{\beta}\left( \overline{\mathbf{u}}_j^n \right) , \mathbf{v} \right)  - \Delta t d\left( \mathbf{u}_j\left(t^{n+\beta}\right) , \mathbf{u}_j\left(t^{n+\beta}\right) , \mathbf{v} \right) ,
		\end{aligned}
	\end{equation}
\end{small}
	and
	\begin{equation}\label{error14_2}
		\begin{aligned}
			&gS\left( A_k^{\beta}\left( \overline{\mathbf{e}}_{j,\phi}^{n+1} \right) , \psi \right)  + \Delta t g\left(\overline{\mathrm{K}} \nabla B_k^{\beta}\left( \overline{\mathbf{e}}_{j,\phi}^{n+1} \right) , \nabla \psi \right)  = gS\left( \overline{A}_k^\beta \left( \overline{\phi}_j^n - \phi_j^n \right) , \psi \right) + \Delta t g\left( R_{k,\phi}^{n+1} , \psi \right)  \\
			&+ \Delta t C_{\Gamma} \left( C_k^{\beta}\left( \overline{\mathbf{e}}_{j,u}^{n} \right) ,  \psi \right) - \Delta t C_{\Gamma}\left( Q_{k,u}^{n+1} , \psi \right) - \Delta t g\left( \left(\mathrm{K}_j - \overline{\mathrm{K}} \right) \nabla C_k^{\beta}\left( \overline{\mathbf{e}}_{j,\phi}^{n} \right) , \nabla \psi \right) + \Delta t g\left( \mathrm{K}_j \nabla P_{k,\phi}^{n+1}   ,  \nabla \psi \right)  \\
			& - \Delta t g\left(  \left( \mathrm{K}_j - \overline{\mathrm{K}} \right) \nabla \left( P_{k,\phi}^{n+1} - Q_{k,\phi}^{n+1} \right) , \nabla \psi \right).
		\end{aligned}
	\end{equation}

	Taking $\mathbf{v} = C_k^{\beta}\left(  \overline{\mathbf{e}}_{j,u}^{n+1} \right)$ in (\ref{error14_1}) and $ \psi = C_k^{\beta}\left(  \overline{\mathbf{e}}_{j,\phi}^{n+1} \right) $ in (\ref{error14_2}), we will estimate each term occurring in (\ref{error14_1}) and (\ref{error14_2}) in the following. First, it follows from the Lemma 2.4 and Lemma 2.5 that there exists a symmetric positive definite matrix $G_k = \left(g_{lm} \right) \in \mathbb{R}^{k\times k}$ such that
	%	\begin{equation}\label{error15_1}
		%		\begin{aligned}
			%				& <<  A_k^{\beta}\left( \overline{\mathbf{e}}^{n+1} \right) , C_k^{\beta}\left(  \overline{\mathbf{e}}^{n+1} \right) >>  \\
			%				& = \sum_{i,j=1}^{k} g_{ij}<< \overline{\mathbf{e}}^{n+1+i-k} , \overline{\mathbf{e}}^{n+1+j-k} >> - \sum_{i,j=1}^{k} g_{ij}<< \overline{\mathbf{e}}^{n+i-k} , \overline{\mathbf{e}}^{n+j-k} >> + \| \sum_{i=0}^{k}\delta _i \overline{\mathbf{e}}^{n+1+i-k} \|_S^2 ,
			%			\end{aligned}
		%	\end{equation}
	\begin{subequations}
		\begin{align}
			& \left( A_k^{\beta}\left( \overline{\mathbf{e}}_{j,u}^{n+1} \right) , C_k^{\beta}\left(  \overline{\mathbf{e}}_{j,u}^{n+1} \right) \right)   \geq \sum_{l,m=1}^{k} g_{lm} \left( \overline{\mathbf{e}}_{j,u}^{n+1+l-k} , \overline{\mathbf{e}}_{j,u}^{n+1+m-k} \right) - \sum_{l,m=1}^{k} g_{lm}\left( \overline{\mathbf{e}}_{j,u}^{n+l-k} , \overline{\mathbf{e}}_{j,u}^{n+m-k} \right) , \label{error15_1} \\
			& \left( A_k^{\beta}\left( \overline{\mathbf{e}}_{j,\phi}^{n+1} \right) , C_k^{\beta}\left(  \overline{\mathbf{e}}_{j,\phi}^{n+1} \right) \right)   \geq \sum_{l,m=1}^{k} g_{lm} \left( \overline{\mathbf{e}}_{j,\phi}^{n+1+l-k} , \overline{\mathbf{e}}_{j,\phi}^{n+1+m-k} \right) - \sum_{l,m=1}^{k} g_{lm}\left( \overline{\mathbf{e}}_{j,\phi}^{n+l-k} , \overline{\mathbf{e}}_{j,\phi}^{n+m-k} \right) . \label{error15_2}
		\end{align}
	\end{subequations}
	%\begin{equation}\label{error15_1}
	%	\begin{aligned}
		%		& \left( A_k^{\beta}\left( \overline{\mathbf{e}}_{j,u}^{n+1} \right) , C_k^{\beta}\left(  \overline{\mathbf{e}}_{j,u}^{n+1} \right) \right)   \geq \sum_{l,m=1}^{k} g_{lm} \left( \overline{\mathbf{e}}_{j,u}^{n+1+l-k} , \overline{\mathbf{e}}_{j,u}^{n+1+m-k} \right) - \sum_{l,m=1}^{k} g_{lm}\left( \overline{\mathbf{e}}_{j,u}^{n+l-k} , \overline{\mathbf{e}}_{j,u}^{n+m-k} \right) ,
		%	\end{aligned}
	%\end{equation}
	%and
	%\begin{equation}\label{error15_2}
	%	\begin{aligned}
		%		\left( A_k^{\beta}\left( \overline{\mathbf{e}}_{j,\phi}^{n+1} \right) , C_k^{\beta}\left(  \overline{\mathbf{e}}_{j,\phi}^{n+1} \right) \right)   \geq \sum_{l,m=1}^{k} g_{lm} \left( \overline{\mathbf{e}}_{j,\phi}^{n+1+l-k} , \overline{\mathbf{e}}_{j,\phi}^{n+1+m-k} \right) - \sum_{l,m=1}^{k} g_{lm}\left( \overline{\mathbf{e}}_{j,\phi}^{n+l-k} , \overline{\mathbf{e}}_{j,\phi}^{n+m-k} \right) .
		%	\end{aligned}
	%\end{equation}
	From the Lemma 2.5, we can split $B_k^{\beta}\left( \overline{\mathbf{e}}_{j,u}^{n+1} \right)$ as follows:
	\begin{equation}\label{error16}
		\begin{aligned}
			&  a\left( B_k^{\beta}\left( \overline{\mathbf{e}}_{j,u}^{n+1} \right) , C_k^{\beta}\left(  \overline{\mathbf{e}}_{j,u}^{n+1} \right) \right) = a\left( \tau_k^{\beta} C_k^{\beta}\left( \overline{\mathbf{e}}_{j,u}^{n+1} \right) + D_k^{\beta} \left( \overline{\mathbf{e}}_{j,u}^{n+1} \right) , C_k^{\beta}\left(  \overline{\mathbf{e}}_{j,u}^{n+1} \right) \right) \\
			& = \nu \tau_k^{\beta}  \| \nabla C_k^{\beta}\left(  \overline{\mathbf{e}}_{j,u}^{n+1} \right) \|^2 + a\left( D_k^{\beta} \left( \overline{\mathbf{e}}_{j,u}^{n+1} \right) , C_k^{\beta}\left(  \overline{\mathbf{e}}_{j,u}^{n+1} \right) \right),
		\end{aligned}
	\end{equation}
	It also follows from the Lemma 2.4 and Lemma 2.5 that there exists a symmetric positive definite matrix $H_k = \left( h_{lm} \right) \in \mathbb{R}^{\left( k-1 \right)\times \left( k-1 \right)}$ such that
	\begin{equation}\label{error16_1}
		\begin{aligned}
			& a\left(  D_k^{\beta} \left( \overline{\mathbf{e}}_{j,u}^{n+1} \right) , C_k^{\beta}\left(  \overline{\mathbf{e}}_{j,u}^{n+1} \right) \right) = \nu \left( \nabla D_k^{\beta} \left( \overline{\mathbf{e}}_{j,u}^{n+1} \right) , \nabla C_k^{\beta}\left(  \overline{\mathbf{e}}_{j,u}^{n+1} \right) \right) \\
			& \geq \nu \sum_{l,m=1}^{k-1}h_{lm}\left( \nabla \overline{\mathbf{e}}_{j,u}^{n+2+l-k} , \nabla \overline{\mathbf{e}}_{j,u}^{n+2+m-k} \right) - \nu \sum_{l,m=1}^{k-1}h_{lm}\left( \nabla \overline{\mathbf{e}}_{j,u}^{n+1+l-k} , \nabla \overline{\mathbf{e}}_{j,u}^{n+1+m-k} \right) .
		\end{aligned}
	\end{equation}
%	and the Lemma 2.5 implies that
%	\begin{equation}\label{error16_2}
%		\begin{aligned}
%			& a\left( F_k^{\beta} \left( \overline{\mathbf{e}}_{j,u}^{n+1} \right) , C_k^{\beta}\left(  \overline{\mathbf{e}}_{j,u}^{n+1} \right) \right) = \nu \left( \nabla F_k^{\beta} \left( \overline{\mathbf{e}}_{j,u}^{n+1} \right) , \nabla C_k^{\beta}\left(  \overline{\mathbf{e}}_{j,u}^{n+1} \right) \right) \\
%			& \geq \kappa _{k} \nu \| \nabla \overline{\mathbf{e}}_{j,u}^{n+1} \|^2 + U_{k}\left( \nabla \overline{\mathbf{e}}_{j,u}^{n+1} ,\dots, \nabla \overline{\mathbf{e}}_{j,u}^{n+3-k} \right) - U_{k}\left( \nabla \overline{\mathbf{e}}_{j,u}^{n} ,\dots, \nabla \overline{\mathbf{e}}_{j,u}^{n+2-k} \right).
%		\end{aligned}
%	\end{equation}
	
	Similarly, we can deduce the following estimates:
	\begin{equation}\label{error16_3}
		\begin{aligned}
			& \Delta t \sum_{i=1}^{d-1} \left( \overline{\eta}_{i} B_k^{\beta}\left( \mathbf{e}_{j,u}^{n+1} \right)\cdot \tau_i , C_k^{\beta}\left( \mathbf{e}_{j,u}^{n+1} \right)\cdot \tau_i \right)_{\Gamma} \\
			& = \Delta t \sum_{i=1}^{d-1} \left( \overline{\eta}_{i} \left( \tau_k^{\beta} C_k^{\beta}\left( \overline{\mathbf{e}}_{j,u}^{n+1} \right) + D_k^{\beta} \left( \overline{\mathbf{e}}_{j,u}^{n+1} \right) \right)\cdot \tau_i , C_k^{\beta}\left( \mathbf{e}_{j,u}^{n+1} \right)\cdot \tau_i \right)_{\Gamma} \\
			& \geq \Delta t \sum_{i=1}^{d-1} \tau_k^{\beta} \overline{\eta}_{i}^{\min}\| C_k^{\beta}\left( \mathbf{e}_{j,u}^{n+1} \right)\cdot \tau_i \|_{\Gamma}^2  + \Delta t \sum_{i=1}^{d-1}   \sum_{l,m=1}^{k-1} h_{lm} \left( \overline{\eta}_{i} \overline{\mathbf{e}}_{j,u}^{n+2+l-k} \cdot \tau_i , \overline{\mathbf{e}}_{j,u}^{n+2+m-k} \cdot \tau_i \right)_{\Gamma}  \\
			&- \Delta t \sum_{i=1}^{d-1}  \sum_{l,m=1}^{k-1} h_{lm} \left( \overline{\eta}_{i} \overline{\mathbf{e}}_{j,u}^{n+1+l-k} \cdot \tau_i , \overline{\mathbf{e}}_{j,u}^{n+1+m-k} \cdot \tau_i \right)_{\Gamma} ,
		\end{aligned}
	\end{equation}
	and
	\begin{equation}\label{error16_4}
		\begin{aligned}
			&\left(\overline{\mathrm{K}} \nabla B_k^{\beta}\left( \overline{\mathbf{e}}_{j,\phi}^{n+1} \right) , \nabla C_k^{\beta}\left(  \overline{\mathbf{e}}_{j,\phi}^{n+1} \right) \right) = \left( \overline{\mathrm{K}} \nabla \left(  \tau_k^{\beta} C_k^{\beta}\left( \overline{\mathbf{e}}_{j,\phi}^{n+1} \right) + D_k^{\beta}\left( \overline{\mathbf{e}}_{j,\phi}^{n+1} \right)  \right) , \nabla  C_k^{\beta}\left(  \overline{\mathbf{e}}_{j,\phi}^{n+1} \right)  \right) \\
			& \geq \tau_k^{\beta} \overline{k}_{\min} \| \nabla C_k^{\beta}\left( \overline{\mathbf{e}}_{j,\phi}^{n+1} \right) \|^2  +   \sum_{l,m=1}^{k-1}h_{lm}\left( \overline{\mathrm{K}} \nabla \overline{\mathbf{e}}_{j,\phi}^{n+2+l-k} , \nabla \overline{\mathbf{e}}_{j,\phi}^{n+2+m-k} \right) \\
			&-  \sum_{l,m=1}^{k-1}h_{lm}\left( \overline{\mathrm{K}} \nabla \overline{\mathbf{e}}_{j,\phi}^{n+1+l-k} , \nabla \overline{\mathbf{e}}_{j,\phi}^{n+1+m-k} \right) .
		\end{aligned}
	\end{equation}
	
	Since the term $C_k^{\beta}\left(  \overline{\mathbf{e}}_{j,u}^{n+1} \right)$ belongs to the divergence-free space, it implies
	\begin{equation}\label{error17}
		b\left( C_k^{\beta}\left(  \overline{\mathbf{e}}_{j,u}^{n+1} \right) ,  B_k^{\beta}\left( \mathbf{e}_{j,p}^{n+1} \right) \right)   = 0, \quad b\left( C_k^{\beta}\left(  \overline{\mathbf{e}}_{j,u}^{n+1} \right) , P_{k,p}^{n+1} \right) = 0.
	\end{equation}
	By using the Cauchy-Schwarz inequality and trace inequality, we obtain
	\begin{equation}\label{error18_1}
		\begin{aligned}
			& \Delta t \sum_{i=1}^{d-1}\left( \left( \eta_{i,j} - \overline{\eta}_{i} \right) C_k^{\beta}\left( \overline{\mathbf{e}}^n_{j,u} \right) \cdot \tau_i , C_k^{\beta}\left( \overline{\mathbf{e}}^{n+1}_{j,u} \right)\cdot \tau_i \right)_{\Gamma}  \leq \Delta t \sum_{i=1}^{d-1} \eta_{i,j}^{\prime \max} \left( C_k^{\beta}\left( \overline{\mathbf{e}}^n_{j,u} \right) \cdot \tau_i , C_k^{\beta}\left( \overline{\mathbf{e}}^{n+1}_{j,u} \right)\cdot \tau_i \right)_{\Gamma} \\
			& \leq \Delta t \sum_{i=1}^{d-1} \left( \frac{ \eta_{i}^{\prime \max}  }{2} \| C_k^{\beta}\left( \overline{\mathbf{e}}^{n+1}_{j,u} \right)\cdot \tau_i \|_{\Gamma}^2 + \frac{ \eta_{i}^{\prime \max}  }{2} \| C_k^{\beta}\left( \overline{\mathbf{e}}^{n}_{j,u} \right)\cdot \tau_i \|_{\Gamma}^2 \right),
		\end{aligned}
	\end{equation}
	\begin{equation}\label{error18_2}
		\begin{aligned}
			& \Delta t \sum_{i=1}^{d-1}\left( \eta_{i,j} P_{k,u}^{n+1} \cdot \tau_i , C_k^{\beta}\left( \overline{\mathbf{e}}^{n+1}_{j,u} \right) \cdot \tau_i \right)  = \Delta t \sum_{i=1}^{d-1} \left( \left( \eta_{i,j} - \overline{\eta}_{i} \right) P_{k,u}^{n+1} \cdot \tau_i , C_k^{\beta}\left( \overline{\mathbf{e}}^{n+1}_{j,u} \right) \cdot \tau_i \right) \\
			& + \Delta t \sum_{i=1}^{d-1}\left(  \overline{\eta}_{i} P_{k,u}^{n+1} \cdot \tau_i , C_k^{\beta}\left( \overline{\mathbf{e}}^{n+1}_{j,u} \right) \cdot \tau_i \right)  \leq \Delta t \sum_{i=1}^{d-1} \frac{ \sigma_1 }{2} \eta_{i}^{\prime \max} \| C_k^{\beta}\left( \overline{\mathbf{e}}^{n+1}_{j,u} \right) \cdot \tau_i \|_{\Gamma}^2 \\
			& + \Delta t \sum_{i=1}^{d-1} \frac{ \sigma_2 }{2}  \| \sqrt{ \overline{\eta}_{i} } C_k^{\beta}\left( \overline{\mathbf{e}}^{n+1}_{j,u} \right) \cdot \tau_i \|_{\Gamma}^2  + \Delta t \sum_{i=1}^{d-1} \frac{ 1 }{2 \sigma_1} \eta_{i}^{\prime \max} \| \nabla P_{k,u}^{n+1} \|^2 + \Delta t \sum_{i=1}^{d-1} C\frac{ 1 }{ 2\sigma_2 }  \| \nabla P_{k,u}^{n+1} \|^2 ,
		\end{aligned}
	\end{equation}
	\begin{equation}\label{error18_3}
		\begin{aligned}
			& \Delta t \sum_{i=1}^{d-1} \left( \left( \eta_{i,j} - \overline{\eta}_{i} \right) \left( P_{k,u}^{n+1} - Q_{k,u}^{n+1}  \right) \cdot \tau_i , C_k^{\beta}\left( \overline{\mathbf{e}}^{n+1}_{j,u} \right) \cdot \tau_i \right) \\
			& \leq \Delta t \sum_{i=1}^{d-1} \frac{ \sigma_1 }{2} \eta_{i}^{\prime \max} \| C_k^{\beta}\left( \overline{\mathbf{e}}^{n+1}_{j,u} \right) \cdot \tau_i \|_{\Gamma}^2 + \Delta t \sum_{i=1}^{d-1}  \frac{ 1 }{2 \sigma_1} \eta_{i}^{\prime \max} \left(  \| \nabla P_{k,u}^{n+1} \|^2 + \| \nabla Q_{k,u}^{n+1} \|^2 \right),
		\end{aligned}
	\end{equation}
	\begin{equation}\label{error18_4}
		\begin{aligned}
			& \Delta t C_{\Gamma} \left( C_k^{\beta}\left( \overline{\mathbf{e}}^{n+1}_{j,u} \right) , C_k^{\beta}\left(  \overline{\mathbf{e}}_{j,\phi}^{n} \right)  \right) \leq C\Delta t \| \nabla C_k^{\beta}\left( \overline{\mathbf{e}}^{n+1}_{j,u} \right) \| \| C_k^{\beta}\left(  \overline{\mathbf{e}}_{j,\phi}^{n} \right) \|^{1/2}  \| C_k^{\beta}\left(  \overline{\mathbf{e}}_{j,\phi}^{n} \right) \|^{1/2} \\
			& \leq  \frac{ \Delta t }{2\mu _1} \nu \| \nabla C_k^{\beta}\left( \overline{\mathbf{e}}^{n+1}_{j,u} \right) \|^2 + \frac{\mu_1 \Delta t}{2} \left( \lambda _1 \overline{k}_{\min} \| \nabla C_k^{\beta}\left(  \overline{\mathbf{e}}_{j,\phi}^{n} \right) \|^2  + \frac{C}{ \overline{k}_{\min} \nu ^2 } \| C_k^{\beta}\left(  \overline{\mathbf{e}}_{j,\phi}^{n} \right) \|^2 \right).
		\end{aligned}
	\end{equation}
	%\begin{equation}\label{error18_5}
	%	\begin{aligned}
		%		\Delta t C_{\Gamma}\left( C_k^{\beta}\left( \overline{\mathbf{e}}^{n+1}_{j,u} \right), Q_{k,\phi}^{n+1}  \right)    & \leq C \Delta t \| \nabla C_k^{\beta}\left( \overline{\mathbf{e}}^{n+1}_{j,u} \right) \|  \| \nabla  Q_{k,\phi}^{n+1} \|  \\
		%		& \leq \frac{\epsilon _1 }{8} \Delta t \nu \| \nabla C_k^{\beta}\left( \overline{\mathbf{e}}^{n+1}_{j,u} \right) \|^2 + C\Delta t \| \nabla  Q_{k,\phi}^{n+1} \|^2 ,
		%	\end{aligned}
	%\end{equation}
	
	%\begin{equation}\label{error18}
	%	\begin{aligned}
		%		&\Delta tC_{\Gamma} \left( C_k^{\beta}\left(  \overline{\mathbf{e}}^{n} \right) , C_k^{\beta}\left(  \overline{\mathbf{e}}^{n+1} \right) \right) \leq C\Delta t \| C_k^{\beta}\left(  \overline{\mathbf{e}}^{n} \right)  \|^{\frac{1}{2}}  \| \nabla C_k^{\beta}\left(  \overline{\mathbf{e}}^{n} \right) \|^{\frac{1}{2}} \| \nabla C_k^{\beta}\left(  \overline{\mathbf{e}}^{n+1} \right) \| \\
		%		& \leq \frac{\Delta t}{2\mu} \| \nabla C_k^{\beta}\left(  \overline{\mathbf{e}}^{n+1} \right) \|^2 + \frac{\mu \Delta t}{2} \left(  \lambda \| \nabla C_k^{\beta}\left(  \overline{\mathbf{e}}^{n} \right) \|^2  + C \| C_k^{\beta}\left(  \overline{\mathbf{e}}^{n} \right) \|^2 \right)   .
		%	\end{aligned}
	%\end{equation}
	
	In a similar way, we can obtain the following estimates for the hydraulic conductivity tensor terms:
	\begin{equation}\label{error18_6}
		\begin{aligned}
			& \Delta t g \left( \left(\mathrm{K}_j - \overline{\mathrm{K}} \right) \nabla C_k^{\beta}\left( \overline{\mathbf{e}}_{j,\phi}^{n} \right) , \nabla C_k^{\beta}\left(  \overline{\mathbf{e}}_{j,\phi}^{n+1} \right) \right)  \leq \Delta t g \left( \rho_j^{\prime \max} \nabla C_k^{\beta}\left( \overline{\mathbf{e}}_{j,\phi}^{n} \right) , \nabla C_k^{\beta}\left(  \overline{\mathbf{e}}_{j,\phi}^{n+1} \right) \right) \\
			& \leq \Delta t g \frac{ \rho^{\prime \max} }{2} \| \nabla C_k^{\beta}\left(  \overline{\mathbf{e}}_{j,\phi}^{n+1} \right) \|^2 + \Delta t g \frac{ \rho^{\prime \max} }{2} \| \nabla C_k^{\beta}\left( \overline{\mathbf{e}}_{j,\phi}^{n} \right) \|^2 ,
		\end{aligned}
	\end{equation}
    \begin{small}
	\begin{equation}\label{error18_7}
		\begin{aligned}
			& \Delta t g\left( \mathrm{K}_j \nabla P_{k,\phi}^{n+1}   ,  \nabla C_k^{\beta}\left(  \overline{\mathbf{e}}_{j,\phi}^{n+1} \right) \right)  = \Delta t g\left( \left(\mathrm{K}_j - \overline{\mathrm{K}} \right) \nabla P_{k,\phi}^{n+1}   ,  \nabla C_k^{\beta}\left(  \overline{\mathbf{e}}_{j,\phi}^{n+1} \right) \right) + \Delta t g\left( \overline{\mathrm{K}} \nabla P_{k,\phi}^{n+1}   ,  \nabla C_k^{\beta}\left(  \overline{\mathbf{e}}_{j,\phi}^{n+1} \right) \right) \\
			& \leq \Delta t g\left( \rho^{\prime \max} \nabla P_{k,\phi}^{n+1}   ,  \nabla C_k^{\beta}\left(  \overline{\mathbf{e}}_{j,\phi}^{n+1} \right) \right) + \Delta t g\left( \overline{\mathrm{K}} \nabla P_{k,\phi}^{n+1}   ,  \nabla C_k^{\beta}\left(  \overline{\mathbf{e}}_{j,\phi}^{n+1} \right) \right) \\
			& \leq \Delta t g \frac{ \sigma_3 \rho^{\prime \max} }{2} \| \nabla C_k^{\beta}\left(  \overline{\mathbf{e}}_{j,\phi}^{n+1} \right) \|^2 + \Delta t g \frac{ \sigma _4 }{2} \| \sqrt{\overline{\mathrm{K}}} \nabla C_k^{\beta}\left(  \overline{\mathbf{e}}_{j,\phi}^{n+1} \right) \|^2 \\
			& + \Delta t g \frac{ \rho^{\prime \max} }{ 2\sigma_3 } \| \nabla P_{k,\phi}^{n+1}  \|^2 + \Delta t g \overline{k}_{\max} \frac{ 1 }{2\sigma _4} \| \nabla P_{k,\phi}^{n+1}  \|^2 ,
		\end{aligned}
	\end{equation}
	\begin{equation}\label{error18_8}
		\begin{aligned}
			& \Delta t g\left(  \left( \mathrm{K}_j - \overline{\mathrm{K}} \right) \nabla \left( P_{k,\phi}^{n+1} - Q_{k,\phi}^{n+1} \right) , \nabla C_k^{\beta}\left( \overline{\mathbf{e}}_{j,\phi}^{n+1} \right) \right)  \leq \Delta t g\left(  \rho^{\prime \max} \nabla \left( P_{k,\phi}^{n+1} - Q_{k,\phi}^{n+1} \right) , \nabla C_k^{\beta}\left(  \overline{\mathbf{e}}_{j,\phi}^{n+1} \right) \right) \\
			& \leq \Delta t g \frac{ \sigma_3 \rho^{\prime \max} }{2} \| \nabla C_k^{\beta}\left(  \overline{\mathbf{e}}_{j,\phi}^{n+1} \right) \|^2 + \Delta t g \frac{ \rho^{\prime \max} }{ 2\sigma_3 } \left( \| \nabla  P_{k,\phi}^{n+1} \|^2 + \| \nabla  Q_{k,\phi}^{n+1} \|^2 \right),
		\end{aligned}
	\end{equation}
\end{small}
	\begin{equation}\label{error18_9}
		\begin{aligned}
			& \Delta t C_{\Gamma} \left( C_k^{\beta}\left( \overline{\mathbf{e}}_{j,u}^{n} \right) ,  C_k^{\beta}\left( \overline{\mathbf{e}}_{j,\phi}^{n+1} \right) \right)  \leq C \Delta t \| \nabla C_k^{\beta}\left( \overline{\mathbf{e}}_{j,\phi}^{n+1} \right) \| \| C_k^{\beta}\left( \overline{\mathbf{e}}_{j,u}^{n} \right) \|^{1/2} \| \nabla C_k^{\beta}\left( \overline{\mathbf{e}}_{j,u}^{n} \right) \|^{1/2} \\
			& \leq \frac{ \Delta t }{2\mu_2} \overline{k}_{\min} \| \nabla C_k^{\beta}\left( \overline{\mathbf{e}}_{j,\phi}^{n+1} \right) \|^2 + \frac{\mu_2 \Delta t}{2}\left( \lambda_2 \nu \| \nabla C_k^{\beta}\left( \overline{\mathbf{e}}_{j,u}^{n} \right) \|^2 + \frac{C}{\nu \overline{k}_{\min}^2 }  \| C_k^{\beta}\left( \overline{\mathbf{e}}_{j,u}^{n} \right) \|^2  \right) .
		\end{aligned}
	\end{equation}
	Following from the Lemma 2.6, we can deduce
	\begin{equation}\label{error19}
		\begin{aligned}
			\Delta t \left( R_{k,u}^{n+1} , C_k^{\beta}\left( \overline{\mathbf{e}}^{n+1}_{j,u} \right) \right) & \leq \Delta t \| C_k^{\beta}\left( \overline{\mathbf{e}}^{n+1}_{j,u} \right) \|  \| R_{k,u}^{n+1} \| \leq C_{p,f} \Delta t \| \nabla C_k^{\beta}\left( \overline{\mathbf{e}}^{n+1}_{j,u} \right) \|  \| R_{k,u}^{n+1} \|  \\
			& \leq  \frac{\epsilon_1 }{16} \Delta t \nu  \| C_k^{\beta}\left( \overline{\mathbf{e}}^{n+1}_{j,u} \right) \|^2 + \frac{C}{\nu}\Delta t \| R_{k,u}^{n+1} \|^2  , \\
			%& \leq  {\color{red} \epsilon_3 \Delta t  \| C_k^{\beta}\left( \overline{\mathbf{e}}^{n+1}_{j,u} \right) \|^2 }  + C\Delta t \| R_{k,u}^{n+1} \|^2  , \\
		\end{aligned}
	\end{equation}
	\begin{equation}\label{error20}
		\Delta t a\left( P_{k,u}^{n+1} , C_k^{\beta}\left( \overline{\mathbf{e}}^{n+1}_{j,u} \right) \right)
		\leq \frac{\epsilon_1}{16}\Delta t \nu \| \nabla C_k^{\beta}\left( \overline{\mathbf{e}}^{n+1}_{j,u} \right) \|^2 +  C\Delta t \nu \| \nabla P_{k,u}^{n+1} \|^2 ,
	\end{equation}
	\begin{equation}\label{error21}
		\Delta t C_{\Gamma}\left( C_k^{\beta}\left( \overline{\mathbf{e}}^{n+1}_{j,u} \right) , Q_{k,\phi}^{n+1}  \right)   \leq \frac{\epsilon _1 }{16} \Delta t \nu \| \nabla C_k^{\beta}\left( \overline{\mathbf{e}}^{n+1}_{j,u} \right) \|^2 + \frac{C}{\nu}\Delta t \| \nabla  Q_{k,\phi}^{n+1} \|^2 .
	\end{equation}
	
	%	\begin{equation}\label{error21}
		%		\begin{aligned}
			%			\Delta t C_{\Gamma}\left( Q_k^n  , C_k^{\beta}\left(  \overline{\mathbf{e}}^{n+1} \right) \right) & \leq C\Delta t \| \nabla Q_k^n \|  \| \nabla C_k^{\beta}\left(  \overline{\mathbf{e}}^{n+1} \right) \| \leq C\frac{ \Delta t }{ \sqrt{c_a} } \| \nabla Q_k^n \|  \| \nabla C_k^{\beta}\left(  \overline{\mathbf{e}}^{n+1} \right) \|_a \\
			%			& \leq \epsilon \Delta t \| C_k^{\beta}\left(  \overline{\mathbf{e}}^{n+1} \right) \|_a^2 + C \Delta t \| \nabla Q_k^n \| ^2 \\
			%			& \leq \epsilon \Delta t \| C_k^{\beta}\left(  \overline{\mathbf{e}}^{n+1} \right) \|_a^2 + C \left( \Delta t \right)^{ 2k+1 }.
			%		\end{aligned}
		%	\end{equation}
	
	In a similar way, we can also have the following estimates:
	\begin{subequations}
		\begin{align}
			\Delta t g\left( R_{k,\phi}^{n+1} , C_k^{\beta}\left( \overline{\mathbf{e}}_{j,\phi}^{n+1} \right) \right)  \leq \frac{\epsilon_2}{8} g\Delta t \overline{k}_{\min} \| \nabla C_k^{\beta}\left( \overline{\mathbf{e}}_{j,\phi}^{n+1} \right) \|^2  + \frac{C}{ \overline{k}_{\min} }g\Delta t \| R_{k,\phi}^{n+1} \| ^2 ,  \label{error19_1} \\
			\Delta t C_{\Gamma}\left( Q_{k,u}^{n+1} , C_k^{\beta}\left( \overline{\mathbf{e}}_{j,\phi}^{n+1} \right) \right)  \leq \frac{\epsilon_2}{8} g\Delta t \overline{k}_{\min} \| \nabla C_k^{\beta}\left( \overline{\mathbf{e}}_{j,\phi}^{n+1} \right) \|^2 + \frac{C}{ \overline{k}_{\min} }g\Delta t  \| \nabla Q_{k,u}^{n+1} \|^2 .  \label{error21_1}
		\end{align}
	\end{subequations}
	
	To bound the remaining term, note that
	\begin{displaymath}
		\mathbf{u}_j^n = \eta_{k,j}^n \overline{\mathbf{u}}_j^n , \phi_{j}^n = \eta_{k,j}^n \overline{\phi}_j^n , \quad |\eta_{k,j}^n -1| \leq C_0^{k+1}\Delta t ^{k+1},\quad \forall n\leq m.
	\end{displaymath}
	Hence, it yields
	%\begin{equation}\label{error22}
	%	\begin{aligned}
		%		\left( \overline{A}_k^\beta \left( \overline{\mathbf{u}}_j^n - \mathbf{u}_j^n \right) , C_k^{\beta}\left( \overline{\mathbf{e}}^{n+1}_{j,u} \right) \right) & \leq  \epsilon_3 \Delta t \| C_k^{\beta}\left( \overline{\mathbf{e}}^{n+1}_{j,u} \right) \|^2 + \frac{C}{\Delta t} \| \overline{A}_k^\beta \left( \overline{\mathbf{u}}_j^n - \mathbf{u}_j^n \right)  \|^2  \\
		%		& \leq  \epsilon_3 \Delta t \| C_k^{\beta}\left( \overline{\mathbf{e}}^{n+1}_{j,u} \right) \|^2 + CC_0^{2k+2}\left( \Delta t \right)^{2k+1} .
		%	\end{aligned}
	%\end{equation}
	\begin{equation}\label{error22}
		\begin{aligned}
			&\left( \overline{A}_k^\beta \left( \overline{\mathbf{u}}_j^n - \mathbf{u}_j^n \right) , C_k^{\beta}\left( \overline{\mathbf{e}}^{n+1}_{j,u} \right) \right)  \leq C\| \nabla C_k^{\beta}\left( \overline{\mathbf{e}}^{n+1}_{j,u} \right) \|  \| \overline{A}_k^\beta \left( \overline{\mathbf{u}}_j^n - \mathbf{u}_j^n \right)  \|  \\
			& \leq  \frac{ \epsilon_1 }{ 16 }\Delta t \nu \| \nabla C_k^{\beta}\left( \overline{\mathbf{e}}^{n+1}_{j,u} \right) \|^2 + \frac{C}{\nu \Delta t} \| \overline{A}_k^\beta \left( \overline{\mathbf{u}}_j^n - \mathbf{u}_j^n \right)  \|^2   \leq  \frac{ \epsilon_1 }{ 16 }\Delta t \nu \| \nabla C_k^{\beta}\left( \overline{\mathbf{e}}^{n+1}_{j,u} \right) \|^2 + CC_0^{2k+2}\left( \Delta t \right)^{2k+1} ,
		\end{aligned}
	\end{equation}
	and
	\begin{equation}\label{error22_0}
		\begin{aligned}
			&\left( \overline{A}_k^\beta \left( \overline{\phi}_j^n - \phi_j^n \right) , C_k^{\beta}\left( \overline{\mathbf{e}}^{n+1}_{j,\phi} \right) \right)  \leq C\| \nabla C_k^{\beta}\left( \overline{\mathbf{e}}^{n+1}_{j,\phi} \right) \| \| \overline{A}_k^\beta \left( \overline{\phi}_j^n - \phi_j^n \right) \| \\
			& \leq \frac{\epsilon_2}{8} \Delta t  \overline{k}_{\min} \| \nabla C_k^{\beta}\left( \overline{\mathbf{e}}^{n+1}_{j,\phi} \right) \|^2 + \frac{C}{\Delta t} \| \overline{A}_k^\beta \left( \overline{\phi}_j^n - \phi_j^n \right) \|^2  \leq \frac{\epsilon_2}{8} \Delta t  \overline{k}_{\min} \| \nabla C_k^{\beta}\left( \overline{\mathbf{e}}^{n+1}_{j,\phi} \right) \|^2 + CC_0^{2k+2}\left( \Delta t \right)^{2k+1} . \\
		\end{aligned}
	\end{equation}
	%\begin{equation}\label{error22_0}
	%	\begin{aligned}
		%		\left( \overline{A}_k^\beta \left( \overline{\phi}_j^n - \phi_j^n \right) , C_k^{\beta}\left( \overline{\mathbf{e}}^{n+1}_{j,\phi} \right) \right) & \leq \epsilon_4 \Delta t gS \| C_k^{\beta}\left( \overline{\mathbf{e}}^{n+1}_{j,\phi} \right) \|^2 + \frac{C}{\Delta t} \| \overline{A}_k^\beta \left( \overline{\phi}_j^n - \phi_j^n \right) \|^2 \\
		%		& \leq \epsilon_4 \Delta t gS \| C_k^{\beta}\left( \overline{\mathbf{e}}^{n+1}_{j,\phi} \right) \|^2 + + CC_0^{2k+2}\left( \Delta t \right)^{2k+1} . \\
		%	\end{aligned}
	%\end{equation}
	
	Furthermore, we will deal with the nonlinear terms. First, we rewrite these terms as follows:
	\begin{equation}\label{error22_1}
		\begin{aligned}
			& d\left( C_k^{\beta}\left( \overline{\mathbf{u}}_j^n \right) , C_k^{\beta}\left( \overline{\mathbf{u}}_j^n \right) , C_k^{\beta}\left( \overline{\mathbf{e}}^{n+1}_{j,u} \right) \right)  -  d\left( \mathbf{u}_j\left(t^{n+\beta}\right) , \mathbf{u}_j\left(t^{n+\beta}\right) , C_k^{\beta}\left( \overline{\mathbf{e}}^{n+1}_{j,u} \right) \right) \\
			& = d\left( C_k^{\beta}\left( \overline{\mathbf{u}}_j^n \right) , C_k^{\beta}\left( \overline{\mathbf{u}}_j^n \right) , C_k^{\beta}\left( \overline{\mathbf{e}}^{n+1}_{j,u} \right) \right) - d\left( C_k^{\beta}\left( \mathbf{u}_j\left(t^n\right) \right) , C_k^{\beta}\left( \mathbf{u}_j\left(t^n\right) \right) , C_k^{\beta}\left( \overline{\mathbf{e}}^{n+1}_{j,u} \right) \right) \\
			& + d\left( C_k^{\beta}\left( \mathbf{u}_j\left(t^n\right) \right) , C_k^{\beta}\left( \mathbf{u}_j\left(t^n\right) \right) , C_k^{\beta}\left( \overline{\mathbf{e}}^{n+1}_{j,u} \right) \right) - d\left( \mathbf{u}_j\left(t^{n+\beta}\right) , \mathbf{u}_j\left(t^{n+\beta}\right) , C_k^{\beta}\left( \overline{\mathbf{e}}^{n+1}_{j,u} \right) \right) \\
			& = d\left( C_k^{\beta}\left( \overline{\mathbf{e}}^{n}_{j,u} \right) , C_k^{\beta}\left( \mathbf{u}_j\left(t^n\right) \right) , C_k^{\beta}\left( \overline{\mathbf{e}}^{n+1}_{j,u} \right) \right) + d\left( C_k^{\beta}\left( \overline{\mathbf{e}}^{n}_{j,u} \right) , C_k^{\beta}\left( \overline{\mathbf{e}}^{n}_{j,u} \right) , C_k^{\beta}\left( \overline{\mathbf{e}}^{n+1}_{j,u} \right) \right) \\
			& - d\left( C_k^{\beta}\left( \mathbf{u}_j\left(t^n\right) \right) , C_k^{\beta}\left( \overline{\mathbf{e}}^{n}_{j,u} \right) , C_k^{\beta}\left( \overline{\mathbf{e}}^{n+1}_{j,u} \right) \right) + d\left( Q_{k,u}^{n+1} , C_k^{\beta}\left( \mathbf{u}_j\left(t^n\right) \right) , C_k^{\beta}\left( \overline{\mathbf{e}}^{n+1}_{j,u} \right) \right)\\
 &+ d\left( \mathbf{u}_j\left(t^{n+\beta}\right) , Q_{k,u}^{n+1} , C_k^{\beta}\left( \overline{\mathbf{e}}^{n+1}_{j,u} \right) \right).
		\end{aligned}
	\end{equation}
	Then, we can bound the right-hand side of (\ref{error22_1}) in what follows:
	\begin{equation}\label{error22_1_1}
		\begin{aligned}
			& d\left( C_k^{\beta}\left( \overline{\mathbf{e}}^{n}_{j,u} \right) , C_k^{\beta}\left( \mathbf{u}_j\left(t^n\right) \right) , C_k^{\beta}\left( \overline{\mathbf{e}}^{n+1}_{j,u} \right) \right) - d\left( C_k^{\beta}\left( \mathbf{u}_j\left(t^n\right) \right) , C_k^{\beta}\left( \overline{\mathbf{e}}^{n}_{j,u} \right) , C_k^{\beta}\left( \overline{\mathbf{e}}^{n+1}_{j,u} \right) \right) \\
			& \leq 2 \| C_k^{\beta}\left( \overline{\mathbf{e}}^{n}_{j,u} \right) \|^{1/2} \| \nabla C_k^{\beta}\left( \overline{\mathbf{e}}^{n}_{j,u} \right)  \|^{1/2}  \| \nabla C_k^{\beta}\left( \mathbf{u}_j\left(t^n\right) \right)  \| \| \nabla C_k^{\beta}\left( \overline{\mathbf{e}}^{n+1}_{j,u} \right) \| \\
			& \leq \frac{\epsilon_1}{16}\nu \| \nabla C_k^{\beta}\left( \overline{\mathbf{e}}^{n+1}_{j,u} \right) \|^2 + \epsilon_3 \nu  \| \nabla C_k^{\beta}\left( \overline{\mathbf{e}}^{n}_{j,u} \right) \|^2 + \frac{C}{\nu ^3} \| C_k^{\beta}\left( \overline{\mathbf{e}}^{n}_{j,u} \right) \|^2 \| \nabla C_k^{\beta}\left( \mathbf{u}_j\left(t^n\right) \right)  \|^4 ,
		\end{aligned}
	\end{equation}
	\begin{equation}\label{error22_1_2}
		\begin{aligned}
			&  d\left( C_k^{\beta}\left( \overline{\mathbf{e}}^{n}_{j,u} \right) , C_k^{\beta}\left( \overline{\mathbf{e}}^{n}_{j,u} \right) , C_k^{\beta}\left( \overline{\mathbf{e}}^{n+1}_{j,u} \right) \right) \\
			& \leq \| C_k^{\beta}\left( \overline{\mathbf{e}}^{n}_{j,u} \right) \|^{1/2} \| \nabla C_k^{\beta}\left( \overline{\mathbf{e}}^{n}_{j,u} \right)  \|^{1/2} \| C_k^{\beta}\left( \overline{\mathbf{e}}^{n}_{j,u} \right) \|^{1/2} \| \nabla C_k^{\beta}\left( \overline{\mathbf{e}}^{n}_{j,u} \right)  \|^{1/2} \| \nabla C_k^{\beta}\left( \overline{\mathbf{e}}^{n+1}_{j,u} \right) \| \\
			& \leq \frac{\epsilon_1}{16}\nu \| \nabla C_k^{\beta}\left( \overline{\mathbf{e}}^{n+1}_{j,u} \right) \|^2 + \frac{C}{\nu} \| C_k^{\beta}\left( \overline{\mathbf{e}}^{n}_{j,u} \right) \|^2 \| \nabla C_k^{\beta}\left( \overline{\mathbf{e}}^{n}_{j,u} \right)  \|^2 ,
		\end{aligned}
	\end{equation}
	\begin{equation}\label{error22_1_3}
		\begin{aligned}
			&  d\left( Q_{k,u}^{n+1} , C_k^{\beta}\left( \mathbf{u}_j\left(t^n\right) \right) , C_k^{\beta}\left( \overline{\mathbf{e}}^{n+1}_{j,u} \right) \right)  \leq C \| \nabla Q_{k,u}^{n+1} \| \| \nabla C_k^{\beta}\left( \mathbf{u}_j\left(t^n\right) \right) \| \| \nabla C_k^{\beta}\left( \overline{\mathbf{e}}^{n+1}_{j,u} \right) \| \\
			& \leq \frac{\epsilon_1}{16}\nu \| \nabla C_k^{\beta}\left( \overline{\mathbf{e}}^{n+1}_{j,u} \right) \|^2 + \frac{C}{\nu}\| \nabla Q_{k,u}^{n+1} \|^2  \| \nabla C_k^{\beta}\left( \mathbf{u}_j\left(t^n\right) \right) \|^2,
		\end{aligned}
	\end{equation}
	\begin{equation}\label{error22_1_4}
		\begin{aligned}
			&  d\left( \mathbf{u}_j\left(t^{n+\beta}\right) , Q_{k,u}^{n+1} , C_k^{\beta}\left( \overline{\mathbf{e}}^{n+1}_{j,u} \right) \right)   \leq C \| \nabla \mathbf{u}_j\left(t^{n+\beta}\right)  \|  \| \nabla Q_{k,u}^{n+1} \|   \| \nabla C_k^{\beta}\left( \overline{\mathbf{e}}^{n+1}_{j,u} \right) \|  \\
			& \leq \frac{\epsilon_1}{16}\nu \| \nabla C_k^{\beta}\left( \overline{\mathbf{e}}^{n+1}_{j,u} \right) \|^2 + \frac{C}{\nu}\| \nabla Q_{k,u}^{n+1} \|^2  \| \nabla \mathbf{u}_j\left(t^{n+\beta}\right)  \| ^2.
		\end{aligned}
	\end{equation}
	
	Combining (\ref{error14_1})-(\ref{error22_1_4}), taking $\epsilon_1 = \epsilon_2 = \tau_k^{\beta}$, $\epsilon_3 = \frac{3}{8}\tau_k^{\beta}$, $\mu_1 = \mu_2 = \mu$, $\lambda_1 = \lambda_2 = \lambda$, $\mu = \frac{1}{\sqrt{\lambda}}$, $\frac{1}{8}\tau_k^{\beta} - \sqrt{\lambda} \geq \rho > 0 $, and summing for $n$ from $k-1$ to $m$, noting that the positive definite symmetric matrix $G_k = \left( g_{ij}\right)$ has a minimum eigenvalue $g_k$, we obtain the following inequality after dropping some unnecessary terms:
	\begin{equation}\label{error23}
		\begin{aligned}
			&g_k \| \overline{\mathbf{e}}^{m+1}_{j,u} \|^2 + \Delta t\sum_{n=k-1}^{m} \rho \nu \| \nabla C_k^{\beta}\left( \overline{\mathbf{e}}^{n+1}_{j,u} \right) \|^2  + \Delta t \sum_{n=k-1}^{m} \sum_{i=1}^{d-1}  \rho_1 \| C_k^{\beta}\left( \overline{\mathbf{e}}^{n+1}_{j,u} \right) \cdot \tau_i \|^2_{\Gamma}  \\
			&+ g_k \| \overline{\mathbf{e}}^{m+1}_{j,\phi} \|^2 + \Delta t\sum_{n=k-1}^{m}  \left( \rho _2 + \rho \overline{k}_{\min} \right) \| \nabla C_k^{\beta}\left( \overline{\mathbf{e}}^{n+1}_{j,\phi} \right) \|^2   \\
			&\leq C\Delta t \sum_{n=k-1}^{m} \left( \left( \| \nabla C_k^{\beta}\left( \mathbf{u}_j\left(t^n\right) \right)  \|^4 + \| \nabla C_k^{\beta}\left( \overline{\mathbf{e}}^{n}_{j,u} \right)  \|^2  + 1 \right) \| C_k^{\beta}\left( \overline{\mathbf{e}}^{n}_{j,u} \right) \|^2  \right) \\
			&+ C\Delta t \sum_{n=k-1}^{m}\| C_k^{\beta}\left( \overline{\mathbf{e}}^{n}_{j,\phi} \right) \|^2 + CTC_0^{2k+2}\left( \Delta t \right)^{2k} \\
			&\leq C\Delta t \sum_{n=k-1}^{m} \| \overline{\mathbf{e}}^{n}_{j,u}  \|^2 \left( \sum_{q=0}^{\min \{ k-1,m-n \} } \left( \| \nabla C_k^{\beta}\left( \mathbf{u}_j\left(t^{n+q}\right) \right)  \|^4 + \| \nabla C_k^{\beta}\left( \overline{\mathbf{e}}^{n+q}_{j,u} \right)  \|^2  + 1 \right)  \right) \\
			&+ C\Delta t  \sum_{n=k-1}^{m}\| \overline{\mathbf{e}}^{n}_{j,\phi} \|^2 + CTC_0^{2k+2}\left( \Delta t \right)^{2k}, \\
		\end{aligned}
	\end{equation}
	where $\rho _1 = \left( \left( \tau_k^{\beta}-\sigma_2 \right)\overline{\eta}_{i}^{\min} - \left( 1+\sigma_1 \right)\eta_{i}^{\prime \max} \right)$, $\rho_2 = \left( \frac{ \tau_k^{\beta} }{2} - \frac{\sigma_4}{2} \right)\overline{k}_{\min} - \left( 1 + \sigma_3 \right)\rho^{\prime\max} $, and we used the assumption on the initial step $\overline{\mathbf{u}}_j^{n}, \mathbf{u}_j^{n}, \overline{\phi}_j^{n}, \phi_j^{n},\forall n \leq k-1$.
	
	To guarantee $\rho _1 >0$ and $\rho _2 >0$, we need $0<\sigma_2 , \sigma_4 < \tau_k^{\beta}$, and
	\begin{equation}
		\frac{\eta_{i}^{\prime \max}}{ \overline{\eta}_i^{\min} } < \frac{ \tau_k^{\beta} - \sigma_2 }{ 1 + \sigma_1 } ,\quad \frac{ \rho^{\prime \max} }{ \overline{k}_{\min} } < \frac{ \tau_k^{\beta} - \sigma_4 }{ 2\left( 1+\sigma_3 \right) },
	\end{equation}
	For $\forall \sigma_2,\sigma_4 \in \left(0, \tau_k^{\beta} \right)$, $\forall \sigma_1 >0$, $\forall \sigma_3>0$, we can derive that $\frac{ \tau_k^{\beta} - \sigma_2 }{ 1 + \sigma_1 } \in \left( 0, \tau_k^{\beta} \right)$ and $\frac{ \tau_k^{\beta} - \sigma_4 }{ 2\left( 1+\sigma_3 \right) } \in \left(0, \frac{\tau_k^{\beta}}{2} \right) $. Then it is easy to find $\sigma_2,\sigma_4 \in \left(0, \tau_k^{\beta} \right)$ and $\sigma_1 >0 , \sigma_3>0$ such that $\frac{\eta_{i}^{\prime \max}}{ \overline{\eta}_i^{\min} } < \frac{2\tau_k^{\beta}}{3}$ and $\frac{ \rho^{\prime \max} }{ \overline{k}_{\min} } < \frac{\tau_k^{\beta}}{3}$.
	
	Furthermore, applying the Lemma 2.3 to (\ref{error23}), it yields
	\begin{equation}\label{error24}
		\begin{aligned}
			&g_k \| \overline{\mathbf{e}}^{m+1}_{j,u} \|^2 + \Delta t\sum_{n=k-1}^{m} \rho \nu \| \nabla C_k^{\beta}\left( \overline{\mathbf{e}}^{n+1}_{j,u} \right) \|^2  + \Delta t \sum_{n=k-1}^{m} \sum_{i=1}^{d-1}  \rho_1 \| C_k^{\beta}\left( \overline{\mathbf{e}}^{n+1}_{j,u} \right) \cdot \tau_i \|^2_{\Gamma}  \\
			&+ g_k \| \overline{\mathbf{e}}^{m+1}_{j,\phi} \|^2 + \Delta t\sum_{n=k-1}^{m}  \left( \rho _2 + \rho \overline{k}_{\min} \right) \| \nabla C_k^{\beta}\left( \overline{\mathbf{e}}^{n+1}_{j,\phi} \right) \|^2   \\
			&\leq CT\exp\left( \left( C_1^4 + C_1^2 + 1 \right)T + C_T \right) C_0^{2k+2} \left( \Delta t \right)^{2k} =: C_2C_0^{2k+2}\left( \Delta t \right)^{2k},
		\end{aligned}
	\end{equation}
	where we used the assumption on the exact solution and (\ref{error10}), and the constant $C_2>0$ is independent of $\Delta t$ and $C_0$.
	%+ \frac{3}{8}\tau_k^{\beta} \nu \Delta t \| \nabla C_k^{\beta}\left( \overline{\mathbf{e}}^{m+1}_{j,u} \right) \|^2

	Particularly, one can deduce from (\ref{error24}) that
	\begin{equation}\label{errorCon1}
		\| \overline{\mathbf{e}}^{m+1}_{j,u} \| , \| \overline{\mathbf{e}}^{m+1}_{j,\phi} \| \leq \sqrt{ C_2C_0^{2k+2} } \left( \Delta t\right)^{k}, \quad \forall m+1 \leq T/{\Delta t}.
	\end{equation}
	
	%\begin{equation*}
	%	\frac{\mathrm{d} r\left( t \right)}{\mathrm{d}t} + \gamma r\left( t \right) = -\left( a\left( \mathbf{u}\left(t\right) , \mathbf{u} \left(t\right) \right) - \left( \mathbf{f}\left(t\right) , \mathbf{u}\left( t\right) \right) + \frac{\alpha}{2}\| \mathbf{f} \left(t\right) \|^2 - \frac{ \gamma}{2}\| \mathbf{u} \|_S^2 \right) + \frac{\alpha}{2}\| \mathbf{f}\left(t\right) \|^2 +  \gamma C_r
	%\end{equation*}
	%
	%\begin{equation*}
	%	\begin{aligned}
		%		\frac{ r^{n+1} - r^{n} }{\Delta t} + \gamma r^{n+1} = & -\frac{r^{n+1}}{E\left(\overline{\mathbf{u}}^{n+1} \right) + C_r} \left(  a\left( \overline{\mathbf{u}}^{n+1} , \overline{\mathbf{u}}^{n+1} \right) - \left( \mathbf{f}^{n+1},\overline{\mathbf{u}}^{n+1} \right) + \frac{\alpha}{2}\| \mathbf{f}^{n+1} \|^2 - \frac{\gamma}{2}\| \overline{\mathbf{u}}^{n+1} \|_S^2 \right) \\
		%		& + \frac{\alpha}{2}\| \mathbf{f}^{n+1} \|^2 + \gamma C_r,
		%	\end{aligned}
	%\end{equation*}
	Step 3: Estimate for $ | 1-\xi_j^{m+1} |$. We subtract (\ref{ensemble_subsys3}) from (\ref{SAV_equation}) to obtain the following error equation for $s_j^{n+1} = r_j\left( t^{n+1} \right) - r_j^{n+1} $ at the time $t^{n+1}$:
	\begin{equation}\label{error26}
		\begin{aligned}
			&s_j^{n+1} - s_j^{n} + \Delta t \gamma_j s_j^{n+1}  = -\Delta t \left( \nu \| \nabla \mathbf{u}_{j}\left(t^{n+1} \right) \|^2 + g\| \sqrt{ \mathrm{K}_j } \nabla \phi _j\left(t^{n+1} \right) \|^2  \right.  + \sum_{i=1}^{d-1} \| \sqrt{ \eta_{i,j} } \mathbf{u}_{j}\left(t^{n+1} \right) \cdot \tau_i \|^2_{ \Gamma }  \\
			& - \left( \mathbf{f}_{f,j}\left(t^{n+1} \right) , \mathbf{u}_{j}\left(t^{n+1} \right) \right) - \left( gf_{p,j}\left(t^{n+1} \right) , \phi _j\left(t^{n+1} \right) \right) \\
			& + \frac{\alpha _j}{2} \left( \| \mathbf{f}_{f,j}\left(t^{n+1} \right) \|^2 + \| g f_{p,j}\left(t^{n+1} \right) \|^2 \right)  \left.  - \frac{\gamma_j}{2}\left( \| \mathbf{u}_{j}\left(t^{n+1} \right) \|^2 + gS\| \phi _j\left(t^{n+1} \right) \|^2 \right)  \right) \\
			& + \Delta t \frac{ r_j^{n+1} }{ E_j\left( \overline{\mathbf{u}}_{j}^{n+1} , \overline{\phi}_{j}^{n+1} \right) + C_{r,j} }  \left\{ \nu \| \nabla \overline{\mathbf{u}}_{j}^{n+1} \|^2 + g\| \sqrt{ \mathrm{K}_j } \nabla \overline{\phi}_{j}^{n+1} \|^2   \right. + \sum_{i=1}^{d-1} \| \sqrt{ \eta_{i,j} } \overline{\mathbf{u}}_{j}^{n+1} \cdot \tau_i \|^2_{ \Gamma } \\
			& - \left( \mathbf{f}_{f,j}^{n+1} , \overline{\mathbf{u}}_{j}^{n+1} \right) - \left( gf_{p,j}^{n+1} , \overline{\phi}_{j}^{n+1} \right)   \left. + \frac{\alpha _j}{2} \left( \| \mathbf{f}_{f,j}^{n+1} \|^2 + \| g f_{p,j}^{n+1} \|^2 \right) - \frac{\gamma_j}{2}\left( \| \overline{\mathbf{u}}_{j}^{n+1} \|^2 + gS\| \overline{\phi}_{j}^{n+1} \|^2 \right)  \right\} + T_{j,1}^n,
		\end{aligned}
	\end{equation}
	where
	\begin{equation}\label{error27}
		T_{j,1}^n = r_j\left( t^{n+1} \right) - r_j\left( t^{n} \right) - \Delta t \frac{\mathrm{d} r_j\left( t^{n+1} \right)}{\mathrm{d}t} ,\quad n\leq m.
	\end{equation}
	
	Next, we will bound each term on the right-hand side of (\ref{error26}) as follows. By direct calculation, we obtain
	\begin{equation}\label{error25}
		\begin{aligned}
			r_{j,tt} = \left( \mathbf{u}_{j,t} , \mathbf{u}_{j,t} \right) + \left( \mathbf{u}_j , \mathbf{u}_{j,tt} \right) + gS\left( \phi_{j,t} , \phi_{j,t} \right) + gS\left( \phi_j , \phi_{j,tt} \right),
		\end{aligned}
	\end{equation}
	then, it follows from (\ref{error27}) that
	\begin{equation}\label{error28}
		\begin{aligned}
			|T_{j,1}^n| & \leq C\Delta t \int_{t^n}^{t^{n+1}}| r_{j,tt}\left(s\right) | \mathrm{d}s  \leq C\Delta t \int_{t^n}^{t^{n+1}} \left( \| \mathbf{u}_j \|^2 + \| \mathbf{u}_{j,t} \|^2 + \| \mathbf{u}_{j,tt} \|^2 + \| \phi_j \|^2 + \| \phi_{j,t} \|^2 + \| \phi_{j,tt} \|^2 \right) \mathrm{d}s .
		\end{aligned}
	\end{equation}
	By adding and subtracting the term $ \frac{ r_j^{n+1} }{ E_j\left( \overline{\mathbf{u}}_{j}^{n+1} , \overline{\phi}_{j}^{n+1} \right) + C_{r,j} }  \nu \| \nabla \mathbf{u}_j \left( t^{n+1} \right) \|^2 $, we get
\begin{small}
	\begin{equation}\label{error29}
		\begin{aligned}
			&\left|  \nu \| \nabla \mathbf{u}_j \left( t^{n+1} \right) \|^2 - \frac{ r_j^{n+1} }{ E_j\left( \overline{\mathbf{u}}_{j}^{n+1} , \overline{\phi}_{j}^{n+1} \right) + C_{r,j} }  \nu\| \nabla \overline{\mathbf{u}}_j^{n+1} \|^2  \right|  \leq \nu \| \nabla \mathbf{u}_j \left( t^{n+1} \right) \|^2 \left| 1 - \frac{ r_j^{n+1} }{ E_j\left( \overline{\mathbf{u}}_{j}^{n+1} , \overline{\phi}_{j}^{n+1} \right) + C_{r,j} } \right| \\
			&+ \frac{ r_j^{n+1} }{ E_j\left( \overline{\mathbf{u}}_{j}^{n+1} , \overline{\phi}_{j}^{n+1} \right) + C_{r,j} } \left| \nu \| \nabla \mathbf{u}_j \left( t^{n+1} \right) \|^2 - \nu\| \nabla \overline{\mathbf{u}}_j^{n+1} \|^2 \right| =: P_{11}^n + P_{12}^n ,
		\end{aligned}
	\end{equation}
\end{small}
	We can deduce from Theorem \ref{stableTh}, (\ref{error11}) and $E_j\left( \mathbf{v} , \psi \right)>0,\forall \mathbf{v}$ and $\psi$ that
	\begin{equation}\label{error30}
		\begin{aligned}
			P_{11}^n  &\leq C \left| 1 - \frac{ r_j^{n+1} }{ E_j\left( \overline{\mathbf{u}}_{j}^{n+1} , \overline{\phi}_{j}^{n+1} \right) + C_{r,j} } \right| \\
			& \leq \left| \frac{r_j\left(t^{n+1}\right)}{E_j\left(\mathbf{u}_j\left( t^{n+1}\right), \phi_{j}\left( t^{n+1}\right) \right) + C_{r,j}} - \frac{r_j^{n+1}}{E_j\left(\mathbf{u}_j\left( t^{n+1}\right), \phi_{j}\left( t^{n+1}\right) \right) + C_{r,j}}   \right| \\
			&+ C\left| \frac{r_j^{n+1}}{E_j\left(\mathbf{u}_j\left( t^{n+1}\right), \phi_{j}\left( t^{n+1}\right) \right) + C_{r,j}} - \frac{ r_j^{n+1} }{ E_j\left( \overline{\mathbf{u}}_{j}^{n+1} , \overline{\phi}_{j}^{n+1} \right) + C_{r,j} } \right|   \\
			& \leq C\left( \left| s_j^{n+1} \right| + \left|  E_j\left(\mathbf{u}_j\left( t^{n+1}\right), \phi_{j}\left( t^{n+1}\right) \right) - E_j\left( \overline{\mathbf{u}}_{j}^{n+1} , \overline{\phi}_{j}^{n+1} \right) \right| \right) \leq C \left( \left| s_j^{n+1} \right| + \| \overline{\mathbf{e}}_{j,u}^{n+1} \|  + \| \overline{\mathbf{e}}_{j,\phi}^{n+1} \|  \right) ,
		\end{aligned}
	\end{equation}
	and
	\begin{equation}\label{error31}
		\begin{aligned}
			P_{12}^n \leq C\left(  \| \nabla \mathbf{u}_j \left( t^{n+1} \right) \| +  \| \nabla \overline{\mathbf{u}}_j^{n+1} \| \right)  \| \nabla \overline{\mathbf{e}}_{j,u}^{n+1} \| .
		\end{aligned}
	\end{equation}
	In a similar way, we have
	\begin{equation}\label{error31_1}
		\begin{aligned}
			&\left|  \| \sqrt{ \mathrm{K}_j } \nabla \phi _j\left(t^{n+1} \right) \|^2 - \frac{ r_j^{n+1} }{ E_j\left( \overline{\mathbf{u}}_{j}^{n+1} , \overline{\phi}_{j}^{n+1} \right) + C_{r,j} } \| \sqrt{ \mathrm{K}_j } \nabla \overline{\phi}_{j}^{n+1} \|^2  \right| \\
			&\leq C\| \nabla \phi _j\left(t^{n+1} \right) \|^2 \left| 1 - \frac{ r_j^{n+1} }{ E_j\left( \overline{\mathbf{u}}_{j}^{n+1} , \overline{\phi}_{j}^{n+1} \right) + C_{r,j} } \right| \\
			&+ \frac{ r_j^{n+1} }{ E_j\left( \overline{\mathbf{u}}_{j}^{n+1} , \overline{\phi}_{j}^{n+1} \right) + C_{r,j} } \left| \| \sqrt{ \mathrm{K}_j } \nabla \phi _j\left(t^{n+1} \right) \|^2 - \| \sqrt{ \mathrm{K}_j } \nabla \overline{\phi}_{j}^{n+1} \|^2 \right| \\
			&\leq C \left( \left| s_j^{n+1} \right| + \| \overline{\mathbf{e}}_{j,u}^{n+1} \|  + \| \overline{\mathbf{e}}_{j,\phi}^{n+1} \|  \right) + C\left(  \| \nabla \phi_j \left( t^{n+1} \right) \| +  \| \nabla \overline{\phi}_j^{n+1} \| \right)  \| \nabla \overline{\mathbf{e}}_{j,\phi}^{n+1} \|,
		\end{aligned}
	\end{equation}
	and
	\begin{equation}\label{error31_2}
		\begin{aligned}
			&\left|  \sum_{i=1}^{d-1} \| \sqrt{ \eta_{i,j} } \mathbf{u}_{j}\left(t^{n+1} \right) \cdot \tau_i \|^2_{ \Gamma } - \frac{ r_j^{n+1} }{ E_j\left( \overline{\mathbf{u}}_{j}^{n+1} , \overline{\phi}_{j}^{n+1} \right) + C_{r,j} } \sum_{i=1}^{d-1} \| \sqrt{ \eta_{i,j} } \overline{\mathbf{u}}_{j}^{n+1} \cdot \tau_i \|^2_{ \Gamma }  \right| \\
			& \leq   C\left( \left| s_j^{n+1} \right| + \| \overline{\mathbf{e}}_{j,u}^{n+1} \|  + \| \overline{\mathbf{e}}_{j,\phi}^{n+1} \| \right)  + C\left(  \| \nabla \mathbf{u}_j \left( t^{n+1} \right) \| +  \| \nabla \overline{\mathbf{u}}_j^{n+1} \| \right)  \sum_{i=1}^{d-1} \|  \overline{e}_{j,u}^{n+1} \cdot \tau_i \|^2_{ \Gamma },
		\end{aligned}
	\end{equation}
	and
	\begin{equation}\label{error32}
		\begin{aligned}
			&\left|  \left( \mathbf{f}_{f,j}\left(t^{n+1} \right) , \mathbf{u}_{j}\left(t^{n+1} \right) \right) - \frac{ r_j^{n+1} }{ E_j\left( \overline{\mathbf{u}}_{j}^{n+1} , \overline{\phi}_{j}^{n+1} \right) + C_{r,j} } \left( \mathbf{f}_{f,j}^{n+1} , \overline{\mathbf{u}}_{j}^{n+1} \right) \right| \\
			&\leq \left( \mathbf{f}_{f,j}\left(t^{n+1} \right) , \mathbf{u}_{j}\left(t^{n+1} \right) \right) \left| 1 - \frac{ r_j^{n+1} }{ E_j\left( \overline{\mathbf{u}}_{j}^{n+1} , \overline{\phi}_{j}^{n+1} \right) + C_{r,j} } \right|  \\
			&+ \frac{ r_j^{n+1} }{ E_j\left( \overline{\mathbf{u}}_{j}^{n+1} , \overline{\phi}_{j}^{n+1} \right) + C_{r,j} } \left| \left( \mathbf{f}_{f,j}^{n+1} , \mathbf{u}_{j}\left(t^{n+1} \right) \right)  - \left( \mathbf{f}_{f,j}^{n+1} , \overline{\mathbf{u}}_{j}^{n+1} \right) \right| \\
			&\leq C \left( \left| s_j^{n+1} \right| + \| \overline{\mathbf{e}}_{j,u}^{n+1} \|  + \| \overline{\mathbf{e}}_{j,\phi}^{n+1} \|  \right) + C\| \mathbf{f}_{f,j}^{n+1} \|  \| \overline{\mathbf{e}}_{j,u}^{n+1} \| ,
		\end{aligned}
	\end{equation}
	and
	\begin{equation}\label{error32_1}
    \begin{split}
		&\left|  \left( gf_{p,j}\left(t^{n+1} \right) , \phi _j\left(t^{n+1} \right) \right) - \frac{ r_j^{n+1} }{ E_j\left( \overline{\mathbf{u}}_{j}^{n+1} , \overline{\phi}_{j}^{n+1} \right) + C_{r,j} } \left( gf_{p,j}^{n+1} , \overline{\phi}_{j}^{n+1} \right) \right| \\
& \leq  \left( \left| s_j^{n+1} \right| + \| \overline{\mathbf{e}}_{j,u}^{n+1} \|  + \| \overline{\mathbf{e}}_{j,\phi}^{n+1} \|  \right) + \| \mathbf{f}_{p,j}^{n+1} \|  \| \overline{\mathbf{e}}_{j,\phi}^{n+1} \| ,
	\end{split}
    \end{equation}
	and
	\begin{equation}\label{error35}
		\frac{\alpha _j}{2} \left( \| \mathbf{f}_{f,j}^{n+1} \|^2 + \| g f_{p,j}^{n+1} \|^2 \right) \left| 1 - \frac{ r_j^{n+1} }{ E_j\left( \overline{\mathbf{u}}_{j}^{n+1} , \overline{\phi}_{j}^{n+1} \right) + C_{r,j} } \right|
		\leq C \left( \left| s_j^{n+1} \right| + \| \overline{\mathbf{e}}_{j,u}^{n+1} \|  + \| \overline{\mathbf{e}}_{j,\phi}^{n+1} \|  \right) ,
	\end{equation}
	and
	\begin{equation}\label{error36}
		\begin{aligned}
			& \left| \frac{\gamma_j}{2} \| \mathbf{u}_{j}\left(t^{n+1} \right) \|^2 - \frac{ r_j^{n+1} }{ E_j\left( \overline{\mathbf{u}}_{j}^{n+1} , \overline{\phi}_{j}^{n+1} \right) + C_{r,j} } \frac{\gamma_j}{2} \| \overline{\mathbf{u}}_{j}^{n+1} \|^2  \right| \\
			& \leq \frac{\gamma_j}{2} \| \mathbf{u}_{j}\left(t^{n+1} \right) \|^2  \left| 1 - \frac{ r_j^{n+1} }{ E_j\left( \overline{\mathbf{u}}_{j}^{n+1} , \overline{\phi}_{j}^{n+1} \right) + C_{r,j} } \right| + \frac{ r_j^{n+1} }{ E_j\left( \overline{\mathbf{u}}_{j}^{n+1} , \overline{\phi}_{j}^{n+1} \right) + C_{r,j} } \left| \| \mathbf{u}_{j}\left(t^{n+1} \right) \|^2 - \| \overline{\mathbf{u}}_{j}^{n+1} \|^2  \right|  \\
			& \leq C \left( \left| s_j^{n+1} \right| + \| \overline{\mathbf{e}}_{j,u}^{n+1} \|  + \| \overline{\mathbf{e}}_{j,\phi}^{n+1} \|  \right) + C\left(  \| \mathbf{u}_j \left( t^{n+1} \right) \| +  \| \overline{\mathbf{u}}_j^{n+1} \| \right)  \| \overline{\mathbf{e}}_{j,u}^{n+1} \| ,
		\end{aligned}
	\end{equation}
	and
	\begin{equation}\label{error36_1}
		\begin{aligned}
			& \left| \frac{\gamma_j}{2} \| \phi _j\left(t^{n+1} \right) \|^2 - \frac{ r_j^{n+1} }{ E_j\left( \overline{\mathbf{u}}_{j}^{n+1} , \overline{\phi}_{j}^{n+1} \right) + C_{r,j} } \frac{\gamma_j}{2} \| \overline{\phi}_{j}^{n+1} \|^2  \right| \\
			& \leq C \left( \left| s_j^{n+1} \right| + \| \overline{\mathbf{e}}_{j,u}^{n+1} \|  + \| \overline{\mathbf{e}}_{j,\phi}^{n+1} \|  \right) + C\left(  \| \phi_j \left( t^{n+1} \right) \| +  \| \overline{\phi}_j^{n+1} \| \right)  \| \overline{\mathbf{e}}_{j,\phi}^{n+1} \| .
		\end{aligned}
	\end{equation}
	
	Now, substituting (\ref{error28})-(\ref{error36_1}) into (\ref{error26}) gives
	\begin{equation}\label{error37}
		\begin{aligned}
			&s_j^{n+1} - s_j^{n} + \Delta t \gamma_j s_j^{n+1}  \leq C\Delta t \left( |s_j^{n+1}| + \| \overline{\mathbf{e}}_{j,u}^{n+1} \|  + \| \overline{\mathbf{e}}_{j,\phi}^{n+1} \|  \right) + C\Delta t \left( \| \nabla \phi_j \left( t^{n+1} \right) \| +  \| \nabla \overline{\phi}_j^{n+1} \| \right)  \| \nabla \overline{\mathbf{e}}_{j,\phi}^{n+1} \| \\
			& + C\Delta t \left(  \| \nabla \mathbf{u}_j \left( t^{n+1} \right) \| +  \| \nabla \overline{\mathbf{u}}_j^{n+1} \|  \right)  \left(  \| \nabla \overline{\mathbf{e}}_{j,u}^{n+1} \| + \sum_{i=1}^{d-1} \|  \overline{e}_{j,u}^{n+1} \cdot \tau_i \|^2_{ \Gamma } \right)   \\
			& + C\Delta t \int_{t^n}^{t^{n+1}} \left( \| \mathbf{u}_j \|^2 + \| \mathbf{u}_{j,t} \|^2 + \| \mathbf{u}_{j,tt} \|^2 + \| \phi_j \|^2 + \| \phi_{j,t} \|^2 + \| \phi_{j,tt} \|^2 \right) \mathrm{d}s, \forall n\leq m.
		\end{aligned}
	\end{equation}
	Taking the sum of (\ref{error37}) for $n$ from $k-1$ to $m$ and using (\ref{error9}), (\ref{error11}), (\ref{error24}), it yields
	\begin{equation}\label{error38}
		\begin{aligned}
			&| s_j^{m+1} | \leq C\Delta t \sum_{n=k-1}^{m} \left( |s_j^{n+1}| + \| \overline{\mathbf{e}}_{j,u}^{n+1} \|  + \| \overline{\mathbf{e}}_{j,\phi}^{n+1} \|  \right) + C\Delta t \sum_{n=k-1}^{m} \left( \| \nabla \phi_j \left( t^{n+1} \right) \| +  \| \nabla \overline{\phi}_j^{n+1} \| \right)  \| \nabla \overline{\mathbf{e}}_{j,\phi}^{n+1} \|  \\
			& + \Delta t \sum_{n=k-1}^{m} \left(  \| \nabla \mathbf{u}_j \left( t^{n+1} \right) \| +  \| \nabla \overline{\mathbf{u}}_j^{n+1} \|  \right)  \left(  \| \nabla \overline{\mathbf{e}}_{j,u}^{n+1} \| + \sum_{i=1}^{d-1} \|  \overline{e}_{j,u}^{n+1} \cdot \tau_i \|^2_{ \Gamma } \right)  + C\Delta t \\
			&  \leq C\Delta t \sum_{n=k-1}^{m}   |s_j^{n+1}| + C\sqrt{ C_2C_0^{2k+2} } \left( \Delta t\right)^{k} + C\Delta t.
		\end{aligned}
	\end{equation}
	By the use of Lemma \ref{gronwall2} with $\Delta t < \frac{1}{2C}$, it gives
	\begin{equation}\label{error39}
		| s_j^{m+1} |  \leq  C\exp \left( (1-C\Delta t)^{-1} CT \right) \left( 1 + \sqrt{ C_2C_0^{2k+2} } \left( \Delta t\right)^{k-1} \right)\Delta t
		\leq C_3\left( 1 +  \sqrt{ C_2C_0^{2k+2} } \left( \Delta t\right)^{k-1} \right)\Delta t
	\end{equation}
	where the positive constant $C_3$ is not dependent on $C_0$ and $\Delta t$, and can be defined as
	\begin{equation}\label{error40}
		C_3 := C\max \{ \exp \left( 2T \right) , 2 \},
	\end{equation}
	then $\Delta t < \frac{ 1 }{2C}$ can be guaranteed by $\Delta t < \frac{1}{C_3}$.
	
	Hence, it follows from (\ref{error24}), (\ref{error30}) and (\ref{error39}) that
	\begin{equation}\label{error41}
		\begin{aligned}
			& |1-\xi _j^{m+1}|  \leq C\left( \left| s_j^{m+1} \right| + \left|  E_j\left(\mathbf{u}_j\left( t^{m+1}\right), \phi_{j}\left( t^{m+1}\right) \right) - E_j\left( \overline{\mathbf{u}}_{j}^{m+1} , \overline{\phi}_{j}^{m+1} \right) \right| \right) \\
  & \leq C \left( \left| s_j^{m+1} \right| + \| \overline{\mathbf{e}}_{j,u}^{m+1} \|  + \| \overline{\mathbf{e}}_{j,\phi}^{m+1} \|  \right) \\
			& \leq C\Delta t \left( \sqrt{ C_2C_0^{2k+2} } \left( \Delta t\right)^{k-1} +  C_3\left( 1 +  \sqrt{ C_2C_0^{2k+2} } \left( \Delta t\right)^{k-1} \right) \right) \leq C_4  \left( C_0^{k+1} \left(\Delta t \right)^{k-1} + 1 \right) \Delta t ,
		\end{aligned}
	\end{equation}
	where the positive constant $C_4$ is not dependent on $C_0$ and $\Delta t$. For simplicity, suppose that $C_4 > \max \{ C_2,C_3,C_2 + C_2C_3,1 \}$.
	
	We can deduce $|1-\xi_j^{m+1} | \leq C_0\Delta t$ from (\ref{error41}) if we define $C_0$ such that
	\begin{equation}\label{error42}
		C_4  \left( C_0^{k+1} \left(\Delta t \right)^{k-1} + 1 \right) \leq C_0.
	\end{equation}
	For the cases $2 \leq k \leq 4$, (\ref{error42}) can hold if we choose $C_0 = 2C_4 >1$ and $\Delta t \leq 1/C_0^{k+1}$:
	\begin{displaymath}
		C_4  \left( C_0^{k+1} \left(\Delta t \right)^{k-1} + 1 \right) \leq C_4  \left( C_0^{k+1} \Delta t  + 1 \right) \leq 2C_4 = C_0 .
	\end{displaymath}
	In summary, under the condition $\Delta t \leq 1/C_0^{k+1}\left(2\leq k \leq 4\right)  $, we have $|1-\xi_j^{m+1}| \leq C_0\Delta t$. Now, the induction process for (\ref{error2}) is completed.
	
	Finally, due to (\ref{errorCon1}), there remains to present $\| \mathbf{e}^{m+1}_{j,u} \| , \| \mathbf{e}^{m+1}_{j,\phi} \| \leq C\Delta t^k$.
	
	We deduce from (\ref{error8}) and $\mathbf{u}_j^{m+1} = \eta_{k,j}^{m+1}\overline{\mathbf{u}}_j^{m+1}$ that
	\begin{equation}\label{error44}
		\begin{aligned}
			\| \mathbf{u}_j^{m+1} - \overline{\mathbf{u}}_j^{m+1} \| \leq | 1-\eta_{k,j}^{m+1} | \| \overline{\mathbf{u}}_j^{m+1} \|  \leq M_k | 1-\eta_{k,j}^{m+1} |.
		\end{aligned}
	\end{equation}
	In addition, we can derive from (\ref{error1}) that
	\begin{equation}\label{error45}
		| 1 - \eta_{k,j}^{m+1} |  \leq C_0^{k+1} \left( \Delta t \right)^{k+1} .
	\end{equation}
	Then, under the condition on $\Delta t$, it follows from (\ref{errorCon1}), (\ref{error44}) and (\ref{error45}) that
	\begin{equation}\label{error46}
		\| \mathbf{e}_{j,u}^{m+1} \|  \leq \| \mathbf{u}_j^{m+1} - \overline{\mathbf{u}}_j^{m+1} \|  + \| \overline{\mathbf{e}}_j^{m+1} \|  \leq \sqrt{ C_2C_0^{2k+2} } \left( \Delta t\right)^{k} + M_k C_0^{k+1} \left( \Delta t \right)^{k+1} \leq C\left( \Delta t \right)^k.
	\end{equation}
	In addition, we can also deduce from (\ref{error9}), (\ref{error24}), and (\ref{error45}) that
	\begin{equation}
		\begin{aligned}
			&\Delta t \sum_{i=k-1}^{m}\| \nabla C_k^{\beta} \left( \mathbf{e}_{j,u}^{i+1} \right)  \|^2 \leq 2\Delta t \sum_{i=k-1}^{m} \| \nabla C_k^{\beta}\left( \overline{\mathbf{e}}_{j,u}^{i+1} \right)  \|^2 + 2\Delta t \sum_{i=k-1}^{m} \| \nabla C_k^{\beta} \left( \overline{\mathbf{u}}_j^{i+1} - \mathbf{u}_j^{i+1} \right)  \|^2 \\
			& \leq 2\Delta t \sum_{i=k-1}^{m} \| \nabla C_k^{\beta} \left( \overline{\mathbf{e}}_{j,u}^{i+1} \right)  \|^2 + 2\Delta t \sum_{i=k-1}^{m} | 1-\eta_{k,j}^{i+1} |^2 \| \nabla C_k^{\beta} \left( \overline{\mathbf{u}}_j^{i+1} \right)  \|^2 \\
			&\leq 2\Delta t \sum_{i=k-1}^{m} \| \nabla C_k^{\beta} \left( \overline{\mathbf{e}}_{j,u}^{i+1} \right)  \|^2 + 2\max_{i}\left\{ | 1-\eta_{k,j}^{i+1} |^2 \right\}\Delta t \sum_{i=k-1}^{m}  \| \nabla  C_k^{\beta} \left( \overline{\mathbf{u}}_j^{i+1} \right) \|^2 \\
			&\leq 2\Delta t \sum_{i=k-1}^{m} \| \nabla C_k^{\beta} \left( \overline{\mathbf{e}}_{j,u}^{i+1} \right)  \|^2 + CC_T \max_{i}\left\{ | 1-\eta_{k,j}^{i+1} |^2 \right\} \\
			& \leq  C_2C_0^{2k+2}  \left( \Delta t\right)^{2k} + CC_{T} C_0^{2k+2} \left( \Delta t \right)^{2k+2} \leq C\left( \Delta t \right)^{2k}.
		\end{aligned}
	\end{equation}
%	\begin{equation}
%		\begin{aligned}
%			&\Delta t \sum_{i=k-1}^{m}\| \nabla \mathbf{e}_{j,u}^{i+1}  \|^2 \leq 2\Delta t \sum_{i=k-1}^{m} \| \nabla \overline{\mathbf{e}}_{j,u}^{i+1}  \|^2 + 2\Delta t \sum_{i=k-1}^{m} \| \nabla \left( \overline{\mathbf{u}}_j^{i+1} - \mathbf{u}_j^{i+1} \right)  \|^2 \\
%			& \leq 2\Delta t \sum_{i=k-1}^{m} \| \nabla \overline{\mathbf{e}}_{j,u}^{i+1}  \|^2 + 2\Delta t \sum_{i=k-1}^{m} | 1-\eta_{k,j}^{i+1} |^2 \| \nabla  \overline{\mathbf{u}}_j^{i+1} \|^2 \\
%			&\leq 2\Delta t \sum_{i=k-1}^{m} \| \nabla \overline{\mathbf{e}}_{j,u}^{i+1}  \|^2 + 2\max_{i}\left\{ | 1-\eta_{k,j}^{i+1} |^2 \right\}\Delta t \sum_{i=k-1}^{m}  \| \nabla  \overline{\mathbf{u}}_j^{i+1} \|^2 \\
%			&\leq 2\Delta t \sum_{i=k-1}^{m} \| \nabla \overline{\mathbf{e}}_{j,u}^{i+1}  \|^2 + CC_T \max_{i}\left\{ | 1-\eta_{k,j}^{i+1} |^2 \right\}  \leq  C_2C_0^{2k+2}  \left( \Delta t\right)^{2k} + CC_{T} C_0^{2k+2} \left( \Delta t \right)^{2k+2} \leq C\left( \Delta t \right)^{2k}.
%		\end{aligned}
%	\end{equation}
	Similarly, we can obtain $ \| \mathbf{e}_{j,\phi}^{m+1} \|^2 + \Delta t \sum_{i=k-1}^{m}\| \nabla C_k^{\beta}\left(  \mathbf{e}_{j,\phi}^{i+1} \right)  \|^2 \leq C\left( \Delta t \right)^{2k} $. Herein, we complete the proof.
	
\end{proof}

\section{Numerical results}
\label{sec:experiments}
\subsection{Accuracy test}

We first evaluate our proposed method on the coupled Navier--Stokes--Darcy system~\eqref{para_stokes}-\eqref{para_interface}. The problem is defined on a rectangular domain $\Omega = [0, 1] \times [-1, 1]$, which is partitioned into the porous medium domain $\Omega_p = [0, 1] \times [-1, 0]$ and the free-flow domain $\Omega_f = [0, 1] \times [0, 1]$.
The physical parameters are non-dimensionalized to unity: $\alpha_{BJ} = \nu = g = S = 1$, and the permeability tensor is isotropic with $\mathrm{K} = k\mathbf{I}$.
To verify the convergence order, we construct the source terms and boundary conditions based on the following exact analytical solution:
\begin{displaymath}
	\left\{
	\begin{aligned}
		\phi(x,y,t) &= (\sin(\pi x)\cos(\pi y)+1)e^t, \quad p(x,y,t) = \cos(\pi x)\cos(\pi y)e^t \\
		\mathbf{u}(x,y,t) &= [\pi\sin(\pi x)cos(\pi y)+x+2y, -\pi\cos(\pi x)\sin(\pi y)-y-2x]^{T}e^t.
	\end{aligned}
	\right.
\end{displaymath}
Then the source terms and Dirichlet boundary
conditions of the model are chosen such that the above functions are the exact solutions of the model. It is acknowledged that this solution was artificially constructed for the purpose of testing the algorithm. Therefore, it does not satisfy the homogeneous BJS interface condition exactly. The introduction of boundary residuals as source terms into the algorithm is therefore required.
In order to check the convergence order, we use uniform grids with grid size $h=1/8  ,1/16  ,1/24  ,1/32  $. The terminal time is set to $T=0.5$. The time step size
$\Delta t$ is set to $\Delta t= h$ for the second-order scheme. We examine our scheme by a set of random parameters with $J=3$, $k_j=[1.1, 2.1, 3.3]$.

Tables \ref{tab:convergence-2nd} and \ref{tab:convergence-4th} present the numerical results for the proposed algorithms. To assess temporal accuracy, we evaluate the $L^2$-norm errors of the velocity ($\mathbf{u}$), pressure ($p$), and hydraulic head $\phi$  under uniform time step refinement ($\Delta t = h$), utilizing $j \in \{1, 2, 3\}$ as representative ensemble realizations. As shown in Table \ref{tab:convergence-2nd}, the GSAV-GBDF2-Ensemble scheme with $\beta=3$ exhibits consistent second-order convergence across all tested realizations, confirming that the ensemble averaging process introduces no degradation to the theoretical temporal accuracy. Similarly, the data in Table \ref{tab:convergence-4th} empirically validate that the GSAV-GBDF4-Ensemble scheme with $\beta=3$ preserves optimal fourth-order accuracy, rendering it highly suitable for long-time simulations requiring stringent error control. Consequently, these numerical experiments robustly corroborate the theoretical convergence properties of the proposed ensemble algorithms, highlighting their ability to efficiently achieve high-order accuracy.

\begin{table}[ht]
	\caption{Convergence results of GSAV-GBDF2-Ensemble scheme, $\Delta t=h$}
	\centering
	\begin{scriptsize}
		\centering
		\begin{tabular}{c||c|c||c|c||c|c}
			\hline
			\multicolumn{7}{c}{2nd-order algorithm, $j=1$}\\
			% \Xhline{1pt}
			\hline
			h & $||\mathbf{u}-\mathbf{u}_{h}||$ & rate & $||p-p_{h}||$ & rate& $||\phi-\phi_{h}||$ & rate  \\
			%\Xhline{1pt}
			\hline
			1/4   & 6.81e-02 &  $\backslash$ & 1.58e-02 & $\backslash$ & 1.94e-01  & $\backslash$   \\
			1/8   & 1.68e-02 & 2.02    & 4.06e-03 & 1.96    & 4.78e-02   & 2.02       \\
			1/16  & 4.04e-03 & 2.05    & 1.01e-03 & 2.00    & 1.19e-02   & 2.00       \\
			1/32  & 1.02e-03 & 2.00    & 2.48e-04 & 2.02    & 3.01e-03   & 1.99      \\
			\hline
			\multicolumn{7}{c}{2nd-order algorithm, $j=2$}\\
			% \Xhline{1pt}
			\hline
			h &  $||\mathbf{u}-\mathbf{u}_{h}||$ & rate & $||p-p_{h}||$ & rate& $||\phi-\phi_{h}||$ & rate  \\
			%\Xhline{1pt}
			\hline
			1/4 & 6.55e-02 & $\backslash$ &1.61e-02 & $\backslash$ &2.01e-01 & $\backslash$\\
			1/8 & 1.58e-02 & 2.05 &  4.02e-03 & 2.00 & 4.81e-02 & 2.06 \\
			1/16 & 4.01e-03 & 1.98 &  9.96e-04 & 2.01 & 1.21e-02 & 1.99 \\
			1/32 & 9.915e-04 & 2.02 &  2.51e-04 & 1.99 & 3.02e-03 & 2.00 \\
			\hline
			\multicolumn{7}{c}{2nd-order algorithm, $j=3$}\\
			% \Xhline{1pt}
			\hline
			h &  $||\mathbf{u}-\mathbf{u}_{h}||$ & rate & $||p-p_{h}||$ & rate& $||\phi-\phi_{h}||$ & rate  \\
			%\Xhline{1pt}
			\hline
			1/4 & 6.88e-02 & $\backslash$ &1.65e-02 & $\backslash$ &2.10e-01 & $\backslash$\\
			1/8 & 1.71e-02 & 2.01 & 4.08e-03 & 2.02 & 4.99e-02 & 2.07 \\
			1/16 & 4.06e-03 & 2.07 & 1.01e-03 & 2.01 & 1.29e-02 & 1.95 \\
			1/32 & 1.02e-03 & 1.99 & 2.48e-04 & 2.03 & 3.14e-03 & 2.03 \\
			\hline
		\end{tabular}
	\end{scriptsize}\label{tab:convergence-2nd}
\end{table}

\begin{table}[ht]
	\caption{Convergence results of GSAV-GBDF4-Ensemble scheme, $\Delta t=h$}
	\centering
	\begin{scriptsize}
		\centering
		\begin{tabular}{c||c|c||c|c||c|c}
			\hline
			\multicolumn{7}{c}{4th-order scheme, $j=1$}\\
			\hline
			h & $||\mathbf{u}-\mathbf{u}_{h}||$ & rate & $||p-p_{h}||$ & rate& $||\phi-\phi_{h}||$ & rate  \\
			\hline
			1/4   & 6.59e-02 & $\backslash$ & 2.79e-02 & $\backslash$ & 5.47e-02  & $\backslash$   \\
			1/8   & 4.25e-03 & 3.95 & 1.77e-03 & 3.98 & 3.83e-03  & 3.84 \\
			1/16  & 2.59e-04 & 4.04 & 1.01e-04 & 4.13 & 2.38e-04  & 4.01 \\
			1/32  & 1.61e-05 & 4.01 & 6.29e-06 & 4.01 & 1.48e-05  & 4.01 \\
			\hline
			\multicolumn{7}{c}{4th-order scheme, $j=2$}\\
			\hline
			h &  $||\mathbf{u}-\mathbf{u}_{h}||$ & rate & $||p-p_{h}||$ & rate& $||\phi-\phi_{h}||$ & rate  \\
			\hline
			1/4 & 6.12e-02 & $\backslash$ & 2.82e-02 & $\backslash$ & 6.01e-02 & $\backslash$ \\
			1/8 & 3.85e-03 & 3.99 & 1.78e-03 & 3.99 & 3.78e-03 & 3.99 \\
			1/16 & 2.38e-04 & 4.02 & 1.10e-04 & 4.02 & 2.36e-04 & 4.00 \\
			1/32 & 1.47e-05 & 4.02 & 6.89e-06 & 4.00 & 1.46e-05 & 4.01 \\
			\hline
			\multicolumn{7}{c}{4th-order scheme, $j=3$}\\
			\hline
			h &  $||\mathbf{u}-\mathbf{u}_{h}||$ & rate & $||p-p_{h}||$ & rate& $||\phi-\phi_{h}||$ & rate  \\
			\hline
			1/4 & 6.48e-02 & $\backslash$ & 2.75e-02 & $\backslash$ & 5.98e-02 & $\backslash$ \\
			1/8 & 4.08e-03 & 3.99 & 1.79e-03 & 3.94 & 3.72e-03 & 4.01 \\
			1/16 & 2.53e-04 & 4.01 & 1.11e-04 & 4.01 & 2.27e-04 & 4.04 \\
			1/32 & 1.56e-05 & 4.02 & 6.86e-06 & 4.02 & 1.40e-05 & 4.02 \\
			\hline
		\end{tabular}
	\end{scriptsize}\label{tab:convergence-4th}
\end{table}

Next, we investigate the computational efficiency of the proposed GSAV-GBDF2-Ensemble scheme. Numerical experiments are conducted for varying numbers of realizations, $J \in \{1, 100, 500, 1000\}$, and the corresponding CPU elapsed times are recorded. For benchmarking, we compare our method against two alternatives: the GSAV-Individual approach (where each realization is solved independently using the BDF2 method) and a traditional fully implicit solver. The performance metrics are summarized in Table \ref{eoc_time_compare}. The results demonstrate that the GSAV-Ensemble scheme exhibits a significant advantage in terms of time efficiency as the number of realizations $J$ increases. Compared to the GSAV-Individual approach, the ensemble strategy achieves a consistent time-saving gain, reaching approximately 33.1\% for $J = 1000$. More notably, when compared to the Traditional method, the proposed algorithm reduces the computational cost by over 90\% across all tested cases, with the efficiency gain peaking at 95.0\% for large-scale simulations ($J = 1000$). This substantial reduction in CPU time confirms the superior scalability of the ensemble framework, which effectively bypasses the heavy computational burden associated with traditional implicit iterations in stochastic simulations.

\begin{table}[ht]
	\caption{Computational Efficiency and Speed-up Analysis ($h=1/32$)}
	\centering
	\small
	\begin{tabular}{lcccc}
		\toprule
		$J$ (Realizations) & 1 & 100 & 500 & 1000 \\
		\midrule
		GSAV-Ensemble (s)   & 3.1 & 256.5 & 1281.9 & 2575.1 \\
		GSAV-Individual (s) & 3.1 & 310.1 & 1714.7 & 3847.2 \\
		Traditional (s)     & 32.8 & 3488.2 & 20122.1 & 51021.4 \\
		\midrule
		\textbf{VS Individual (Gain)} & 0\% & 17.3\% & 25.2\% & 33.1\% \\
		\textbf{VS Traditional (Gain)} & 90.5\% & 92.6\% & 93.6\% & 95.0\% \\
		\bottomrule
	\end{tabular}
	\label{eoc_time_compare}
\end{table}

\subsection{Simulation of subsurface flow in a simplified Y-shaped domain}

In this example, we apply the proposed method to a local simulation of subsurface flow, specifically modeling the merging of two conduits into a single channel within a coupled porous medium. As illustrated in Fig. \ref{Y_style_domain_figure}, the computational domain $\Omega$ is a unit square partitioned into a porous media subdomain, $\Omega_p$, and a free-flow subdomain, $\Omega_f$. The region $\Omega_f$ is defined by a decagon with vertices $\overline{ABCDEFGHIJ}$. Consequently, the porous-media domain is defined as $\Omega_p = \Omega \setminus \Omega_f$. The boundary segments are specified as $S_0 = \overline{AB} \cup \overline{JA}$ for the inlets, and $S_1 = \overline{DE}$ and $S_2 = \overline{GH}$ for the respective outlets.
\begin{figure}[ht]
	\centering
	\includegraphics[height=1.5in]{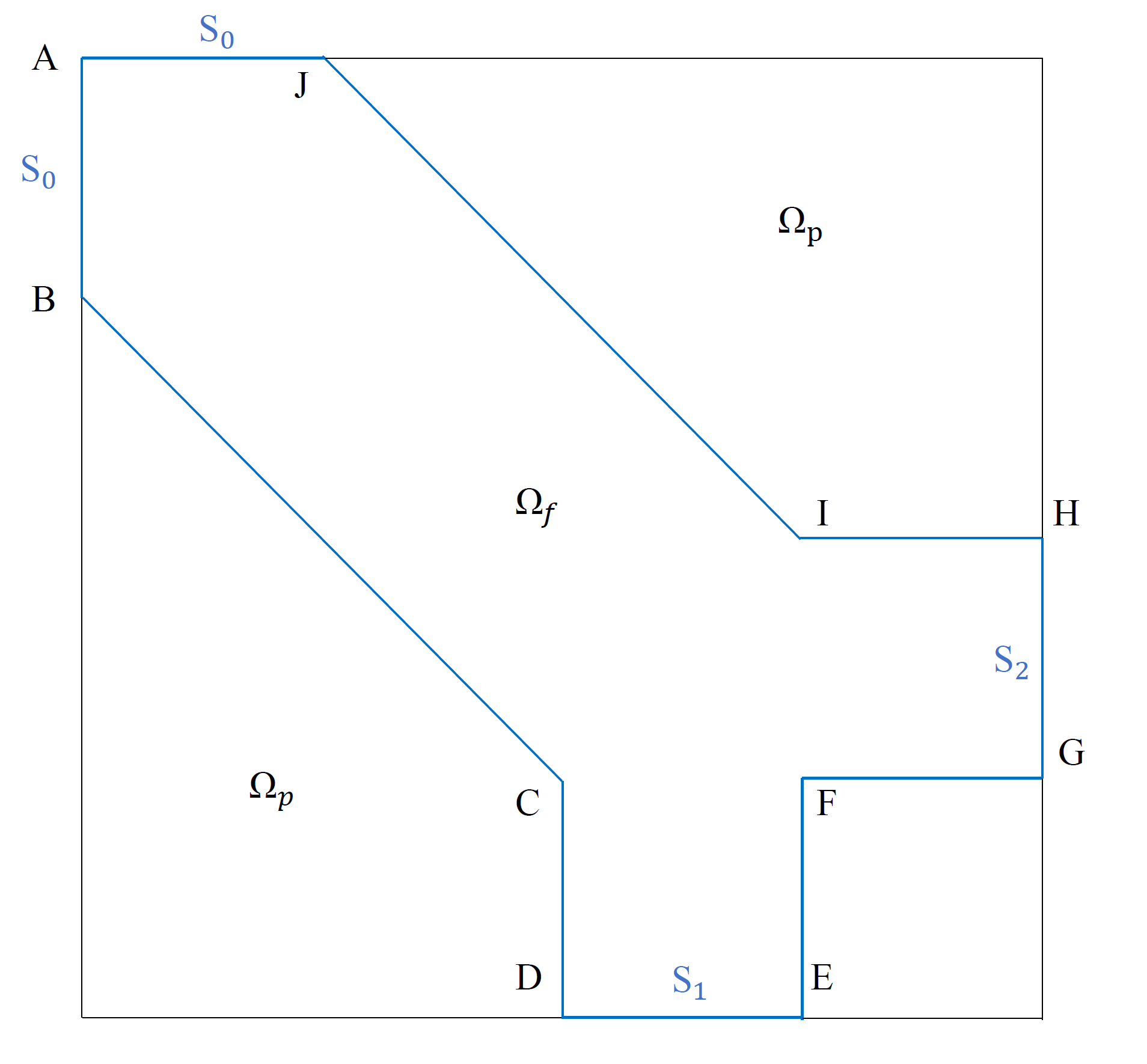}
	\caption{Domain $\Omega_p$ for the porous media and $\Omega_f$ for the free flow. $A=(0,1), B=(0,3/4), C=(1/2,1/4), D=(1/2,0), E=(3/4,0), F=(3/4,1/4), G=(1,1/4), H=(1,1/2), I=(3/4,1/2)$ and $J=(1/4,1)$.}
	\label{Y_style_domain_figure}
\end{figure}

For the numerical simulations, we set the physical parameters $T=1$, $\alpha_{BJ}=1$, $\nu=1$, $g=1$, with the permeability tensor defined as $\mathrm{K}_j=k_jI$ with $j=1, 2, 3,..., J$. All boundary data and source terms are set to zero, with the exception of the fluxes $Q_i$ prescribed on the boundary segments $S_i$ for $i=0, 1, 2$. The spatial domain $\Omega$ is partitioned into a uniform rectangular mesh with characteristic mesh size $h=1/32$. Each rectangle is further bisected into two triangles along its diagonal to form the final triangulation. In this experiment, we enforce a coupling between the spatial and temporal steps by choosing $\Delta t=h$. Here, the GSAV-GBDF3-Ensemble scheme with $\beta=3$ defined by \eqref{3rd-order} is employed. The numerical results obtained at the final time $T=1$ are presented below.

We construct the random hydraulic conductivity tensor as
$k_j(\mathbf{x}, \omega)=(1+\omega)\times10^{-2}$, where $\omega\sim U(0,1)$.
%\begin{eqnarray*}
%	\mathrm{K}(\mathbf{x}, \omega)= a_0+ exp\left\{\left[ Y_1(\omega)cos(\pi y)+Y_3(\omega)sin(\pi y)\right]e^{-\frac{1}{8}}+\left[Y_2(\omega)cos(\pi x)+Y_4(\omega)sin(\pi x)\right] e^{-\frac{1}{8}} \right\}.
%\end{eqnarray*}
%where $\mathbf{x}=(x,y)^T, a_0=1/100$, and $Y_1, \dots, Y_{4}$ are independent and identically distributed with zero mean and unit variance.
We perform a series of ensemble simulations with $J=100$ realizations to assess the model's response to varying flux conditions. In the first scenario, we set $Q_1 = Q_2 = -1$ and $Q_0 = 2$, ensuring that the total inflow rate perfectly balances the total outflow rate. In the second scenario, we maintain $Q_1 = Q_2 = -1$ but reduce the outflow to $Q_0 = 1$. This net positive inflow mimics a recharge-dominant state, such as a rainy season, where excess pressure drives fluid from the conduits into the surrounding porous matrix. Conversely, the third scenario sets $Q_0 = 3$, resulting in a net outflow. This induces a discharge-dominant state, characteristic of a dry season, where the porous media acts as a source feeding the conduits.
The mean velocity profiles computed via the GSAV-GBDF3-Ensemble scheme at the final time $T=1$ for these three cases are illustrated in Figure \ref{GSAV-Ensemble_figure}. A comparative analysis with the results obtained via the traditional fully implicit solver (shown in Figure \ref{Traditional_figure}) confirms that both methods yield nearly identical flow patterns and qualitative behavior. However, the proposed GSAV-Ensemble scheme achieves these results with significantly higher computational efficiency.

\begin{figure}[ht]
	\centering
	\includegraphics[width=0.3\linewidth]{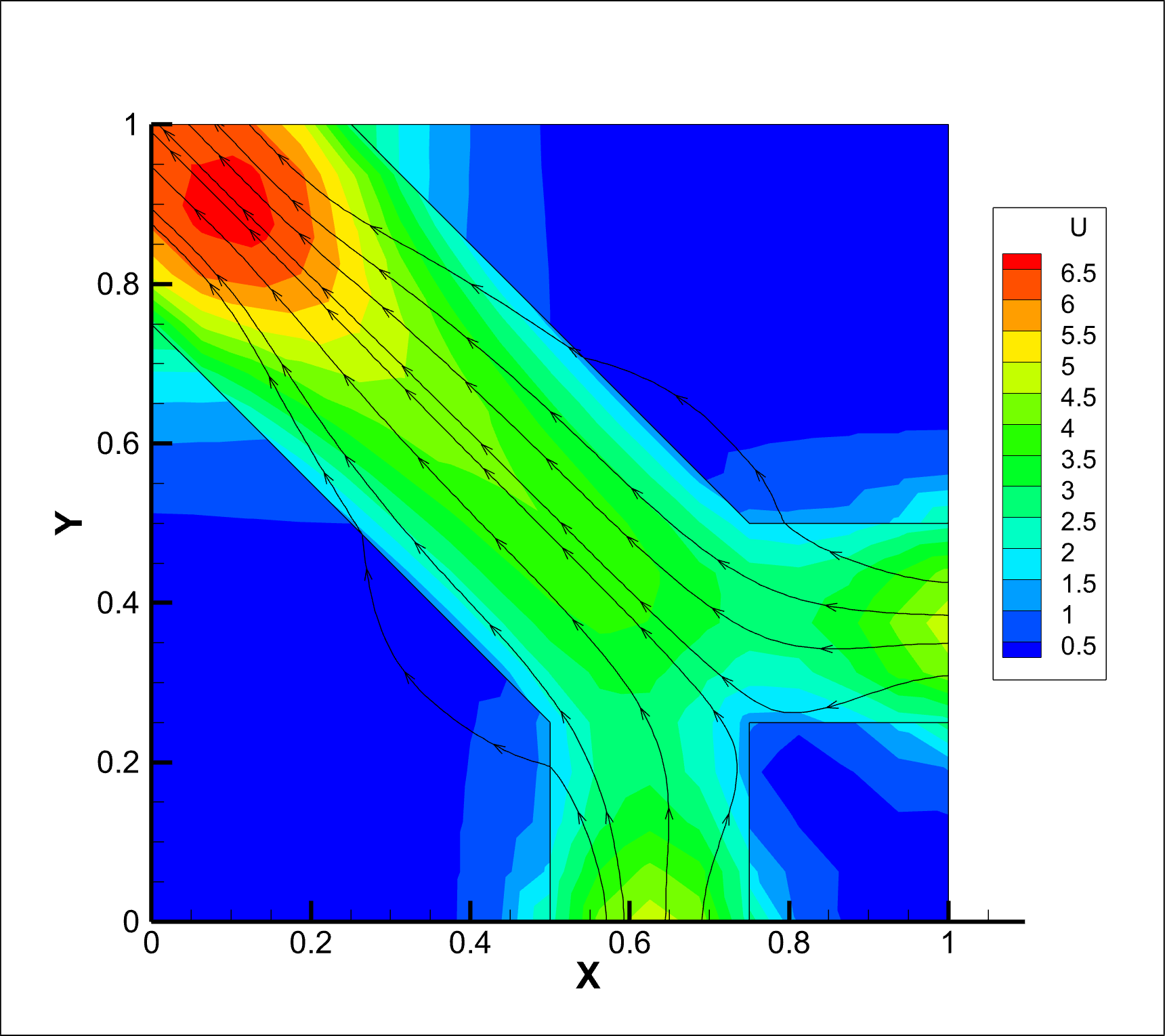}
	\includegraphics[width=0.3\linewidth]{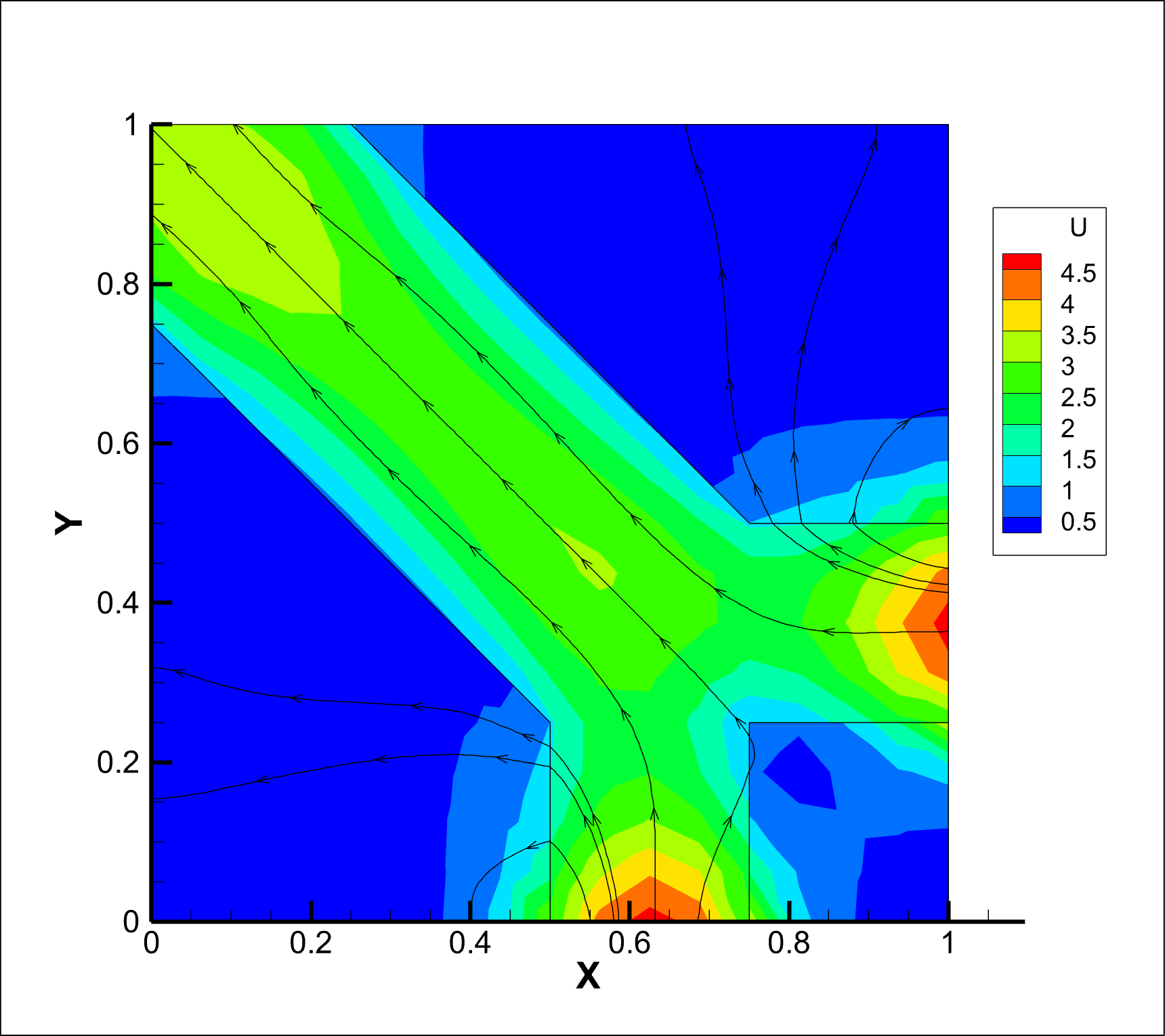}
	\includegraphics[width=0.3\linewidth]{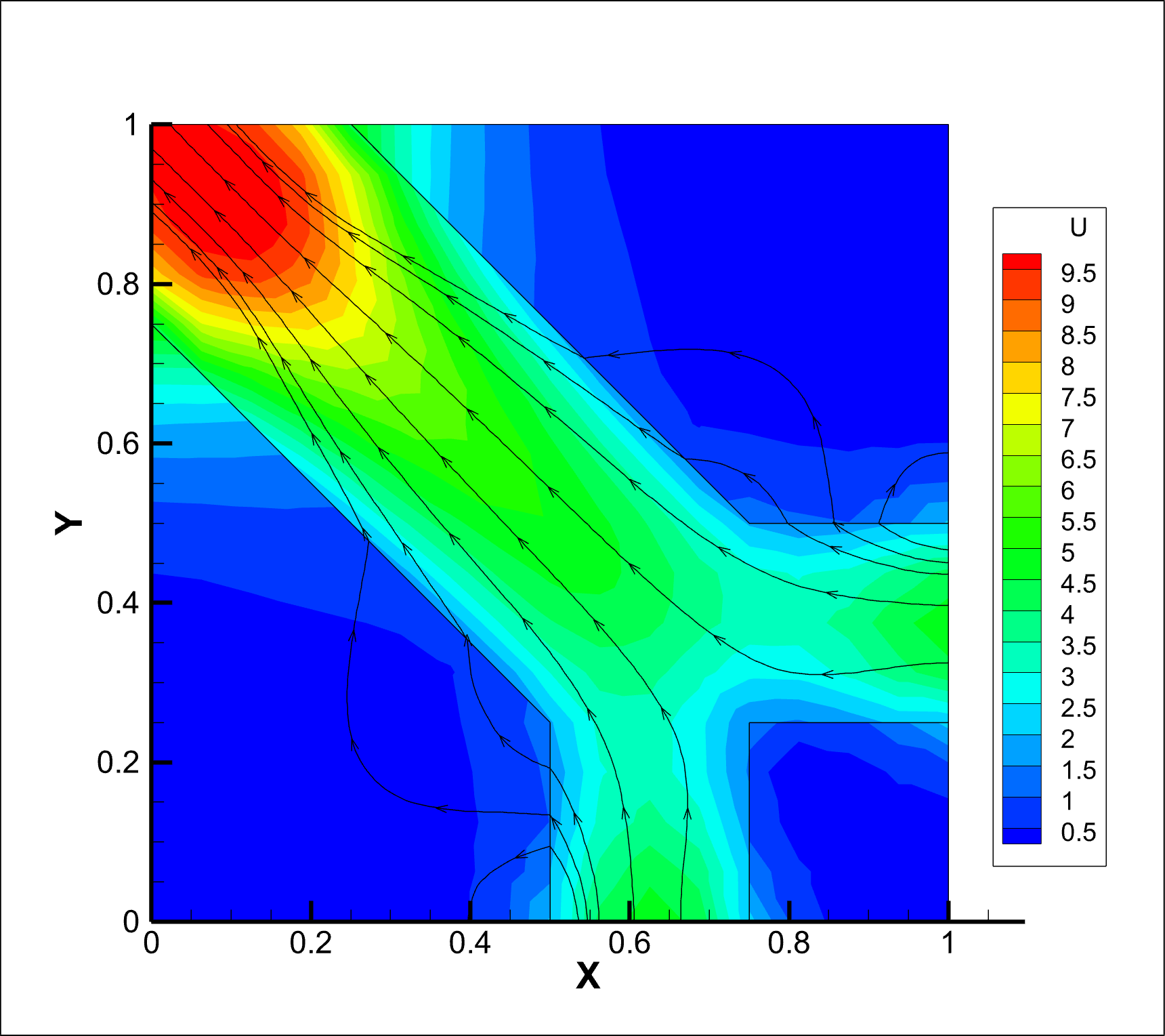}
	\caption{Mean velocity from GSAV-GBDF3-Ensemble algorithm for
		$Q_1=-1$, $Q_2=-1$, and different $Q_0$: $Q_0=2$ (left), $Q_0=1$ (middle) , and $Q_0=3$ (right).}
	\label{GSAV-Ensemble_figure}
\end{figure}
\begin{figure}[ht]
	\centering
	\includegraphics[width=0.3\linewidth]{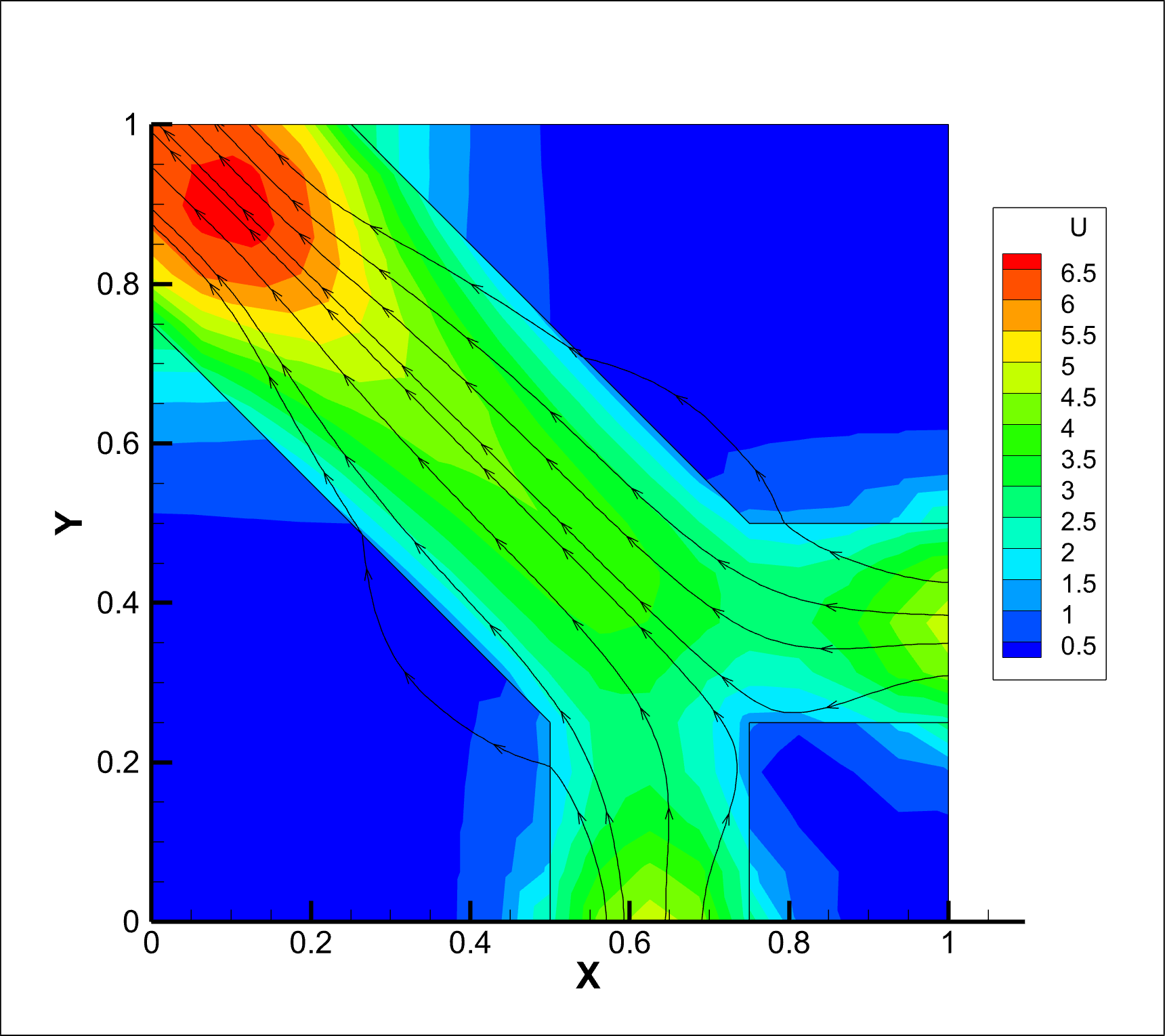}
	\includegraphics[width=0.3\linewidth]{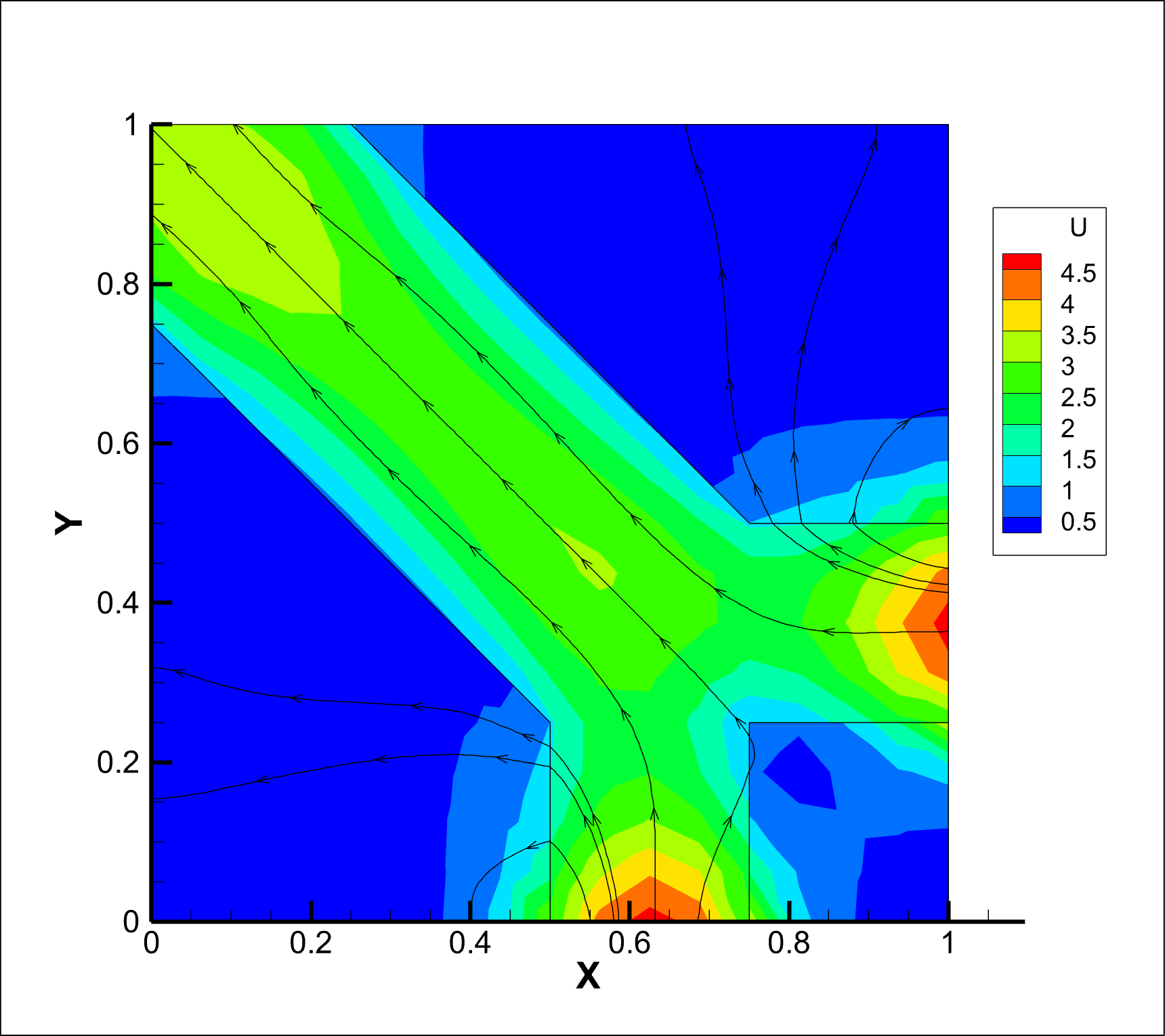}
	\includegraphics[width=0.3\linewidth]{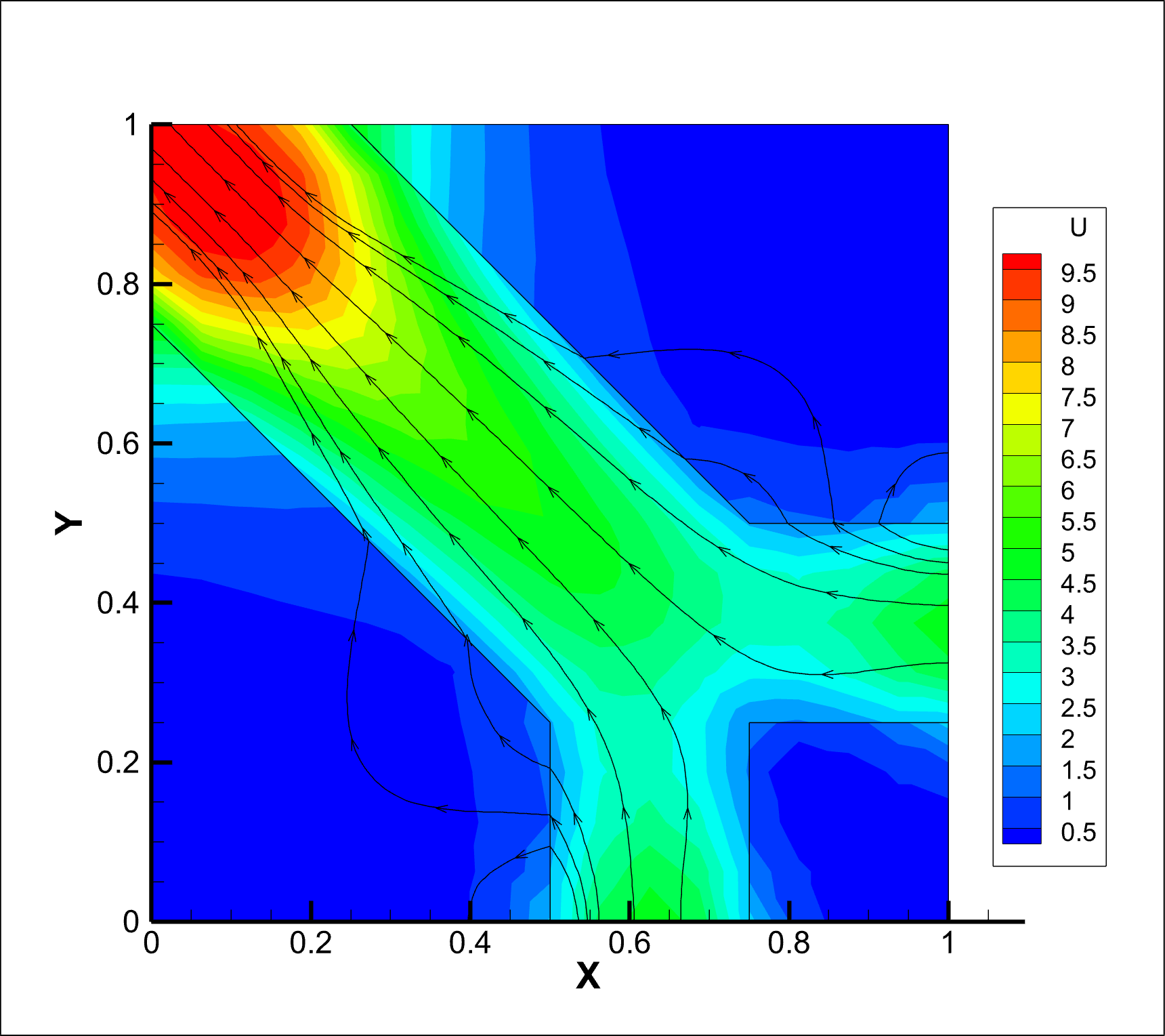}
	\caption{Mean velocity from tradition algorithm for
		$Q_1=-1$, $Q_2=-1$, and different $Q_0$: $Q_0=2$ (left), $Q_0=1$ (middle) , and $Q_0=3$ (right).}
	\label{Traditional_figure}
\end{figure}

\section{Conclusions}\label{sec:conclusions}

In this paper, we developed a class of efficient high-order, long-time stable, and parallel decoupled ensemble schemes for the unsteady Navier--Stokes--Darcy system with uncertain initial data, forcing terms, and hydraulic conductivity tensors. The proposed methods combine a partitioned decoupling strategy, the GSAV approach, and generalized BDF$k$ time discretizations. This design enables explicit treatment of the nonlinear convection term, decouples the free-flow and porous-media subproblems, and allows all ensemble members to share common coefficient matrices, thereby significantly improving computational efficiency. We proved unconditional long-time stability of the proposed schemes by establishing uniform-in-time bounds for the numerical solutions, without imposing time-step restrictions caused by the explicit nonlinear treatment or the partitioned coupling. We also derived optimal-order error estimates, providing a rigorous analysis of the high-order methods. Numerical experiments confirmed the theoretical convergence rates and demonstrated the stability, accuracy, and efficiency of the proposed schemes.

\bibliographystyle{siamplain}
\bibliography{references.bib}

@article{girault2009dg,
	title={DG approximation of coupled Navier--Stokes and Darcy equations by Beaver--Joseph--Saffman interface condition},
	author={Girault, Vivette and Rivi{\`e}re, B{\'e}atrice},
	journal={SIAM Journal on Numerical Analysis},
	volume={47},
	number={3},
	pages={2052--2089},
	year={2009},
	publisher={SIAM}
}

@article{chidyagwai2009solution,
	title={On the solution of the coupled Navier--Stokes and Darcy equations},
	author={Chidyagwai, Prince and Rivi{\`e}re, B{\'e}atrice},
	journal={Computer Methods in Applied Mechanics and Engineering},
	volume={198},
	number={47-48},
	pages={3806--3820},
	year={2009},
	publisher={Elsevier}
}

@article{saffman1971boundary,
	title={On the boundary condition at the surface of a porous medium},
	author={Saffman, Philip Geoffrey},
	journal={Studies in applied mathematics},
	volume={50},
	number={2},
	pages={93--101},
	year={1971},
	publisher={Wiley Online Library}
}

@article{beavers1967boundary,
	title={Boundary conditions at a naturally permeable wall},
	author={Beavers, Gordon S and Joseph, Daniel D},
	journal={Journal of fluid mechanics},
	volume={30},
	number={1},
	pages={197--207},
	year={1967},
	publisher={Cambridge University Press}
}

@Article{shen1990long,
	title={Long time stability and convergence for fully discrete nonlinear {Galerkin} methods},
	author={Shen, Jie},
	journal={Applicable Analysis},
	volume={38},
	number={4},
	pages={201--229},
	year={1990},
	publisher={Taylor \& Francis}
}

@Article{huang2025stability,
	title={Stability and error analysis of a new class of higher-order consistent splitting schemes for the {Navier-Stokes} equations},
	author={Huang, Fukeng and Shen, Jie},
	journal={Mathematics of Computation, https://doi.org/10.1090/mcom/4132},
	year={2025}
}

@Article{huang2024new,
	title={On a new class of {BDF and IMEX} schemes for parabolic type equations},
	author={Huang, Fukeng and Shen, Jie},
	journal={SIAM Journal on Numerical Analysis},
	volume={62},
	number={4},
	pages={1609--1637},
	year={2024},
	publisher={SIAM}
}

@Article{huang2023stability,
	title={Stability and error analysis of a second-order consistent splitting scheme for the {Navier--Stokes} equations},
	author={Huang, Fukeng and Shen, Jie},
	journal={SIAM Journal on Numerical Analysis},
	volume={61},
	number={5},
	pages={2408--2433},
	year={2023},
	publisher={SIAM}
}

@Article{wang2024class,
	title={A Class of New Linear, Efficient and High-Order Implicit-Explicit Methods for the Unsteady {Navier--Stokes-Darcy} Model Based on Nonlinear {Lions} Interface Condition},
	author={Wang, Xinhui and Guo, Xu and Li, Xiaoli},
	journal={Journal of Scientific Computing},
	volume={101},
	number={2},
	pages={40},
	year={2024},
	publisher={Springer}
}

@Article{shan2013partitioned,
	title={Partitioned time stepping method for fully evolutionary {Stokes--Darcy} flow with {Beavers--Joseph} interface conditions},
	author={Shan, Li and Zheng, Haibiao},
	journal={SIAM Journal on Numerical Analysis},
	volume={51},
	number={2},
	pages={813--839},
	year={2013},
	publisher={SIAM}
}

@Article{layton2013analysis,
	title={Analysis of long time stability and errors of two partitioned methods for uncoupling evolutionary groundwater--surface water flows},
	author={Layton, William and Tran, Hoang and Trenchea, Catalin},
	journal={SIAM Journal on Numerical Analysis},
	volume={51},
	number={1},
	pages={248--272},
	year={2013},
	publisher={SIAM}
}

@Article{chen2013efficient,
	title={Efficient and long-time accurate second-order methods for the {Stokes--Darcy} system},
	author={Chen, Wenbin and Gunzburger, Max and Sun, Dong and Wang, Xiaoming},
	journal={SIAM Journal on Numerical Analysis},
	volume={51},
	number={5},
	pages={2563--2584},
	year={2013},
	publisher={SIAM}
}

@Article{mu2010decoupled,
	title={Decoupled schemes for a non-stationary mixed {Stokes-Darcy} model},
	author={Mu, Mo and Zhu, Xiaohong},
	journal={Mathematics of Computation},
	volume={79},
	number={270},
	pages={707--731},
	year={2010}
}

@Article{huang2022new,
	title={A new class of implicit--explicit {BDFk SAV} schemes for general dissipative systems and their error analysis},
	author={Huang, Fukeng and Shen, Jie},
	journal={Computer Methods in Applied Mechanics and Engineering},
	volume={392},
	pages={114718},
	year={2022},
	publisher={Elsevier}
}

@Article{huang2020highly,
	title={A highly efficient and accurate new scalar auxiliary variable approach for gradient flows},
	author={Huang, Fukeng and Shen, Jie and Yang, Zhiguo},
	journal={SIAM Journal on Scientific Computing},
	volume={42},
	number={4},
	pages={A2514--A2536},
	year={2020},
	publisher={SIAM}
}

@Article{shen2019new,
	title={A new class of efficient and robust energy stable schemes for gradient flows},
	author={Shen, Jie and Xu, Jie and Yang, Jiang},
	journal={SIAM Review},
	volume={61},
	number={3},
	pages={474--506},
	year={2019},
	publisher={SIAM}
}

@Article{shen2018scalar,
	title={The scalar auxiliary variable ({SAV}) approach for gradient flows},
	author={Shen, Jie and Xu, Jie and Yang, Jiang},
	journal={Journal of Computational Physics},
	volume={353},
	pages={407--416},
	year={2018},
	publisher={Elsevier}
}

@Article{qiu2025high,
	title={A high order ensemble algorithm for dual-porosity-{Navier-Stokes} flows},
	author={Qiu, Changxin and Hou, Jiangyong and Xia, Yinhua and Shan, Li},
	journal={Journal of Computational Physics},
	volume={529},
	pages={113833},
	year={2025},
	publisher={Elsevier}
}

@Article{jiang2024highly,
	title={Highly efficient ensemble algorithms for computing the {Stokes--Darcy} equations},
	author={Jiang, Nan and Yang, Huanhuan},
	journal={Computer Methods in Applied Mechanics and Engineering},
	volume={418},
	pages={116562},
	year={2024},
	publisher={Elsevier}
}

@Article{jiang2021artificial,
	title={An artificial compressibility {Crank--Nicolson} leap-frog method for the {Stokes--Darcy} model and application in ensemble simulations},
	author={Jiang, Nan and Li, Ying and Yang, Huanhuan},
	journal={SIAM Journal on Numerical Analysis},
	volume={59},
	number={1},
	pages={401--428},
	year={2021},
	publisher={SIAM}
}

@Article{jiang2021sav,
	title={{SAV} decoupled ensemble algorithms for fast computation of {Stokes--Darcy} flow ensembles},
	author={Jiang, Nan and Yang, Huanhuan},
	journal={Computer Methods in Applied Mechanics and Engineering},
	volume={387},
	pages={114150},
	year={2021},
	publisher={Elsevier}
}

@Article{he2020artificial,
	title={An artificial compressibility ensemble algorithm for a stochastic {Stokes-Darcy} model with random hydraulic conductivity and interface conditions},
	author={He, Xiaoming and Jiang, Nan and Qiu, Changxin},
	journal={International Journal for Numerical Methods in Engineering},
	volume={121},
	number={4},
	pages={712--739},
	year={2020},
	publisher={Wiley Online Library}
}

@Article{jiang2019efficient,
	title={An efficient ensemble algorithm for numerical approximation of stochastic {Stokes--Darcy} equations},
	author={Jiang, Nan and Qiu, Changxin},
	journal={Computer Methods in Applied Mechanics and Engineering},
	volume={343},
	pages={249--275},
	year={2019},
	publisher={Elsevier}
}

@Article{mohebujjaman2017efficient,
	title={An efficient algorithm for computation of {MHD} flow ensembles},
	author={Mohebujjaman, Muhammad and Rebholz, Leo G},
	journal={Computational Methods in Applied Mathematics},
	volume={17},
	number={1},
	pages={121--137},
	year={2017},
	publisher={De Gruyter}
}

@Article{fiordilino2018second,
	title={A second order ensemble timestepping algorithm for natural convection},
	author={Fiordilino, Joseph A},
	journal={SIAM Journal on Numerical Analysis},
	volume={56},
	number={2},
	pages={816--837},
	year={2018},
	publisher={SIAM}
}

@Article{jiang2021stabilized,
	title={Stabilized scalar auxiliary variable ensemble algorithms for parameterized flow problems},
	author={Jiang, Nan and Yang, Huanhuan},
	journal={SIAM Journal on Scientific Computing},
	volume={43},
	number={4},
	pages={A2869--A2896},
	year={2021},
	publisher={SIAM}
}

@Article{gunzburger2019efficient,
	title={An efficient algorithm for simulating ensembles of parameterized flow problems},
	author={Gunzburger, Max and Jiang, Nan and Wang, Zhu},
	journal={IMA Journal of Numerical Analysis},
	volume={39},
	number={3},
	pages={1180--1205},
	year={2019},
	publisher={Oxford University Press}
}

@Article{jiang2015higher,
	title={A higher order ensemble simulation algorithm for fluid flows},
	author={Jiang, Nan},
	journal={Journal of Scientific Computing},
	volume={64},
	number={1},
	pages={264--288},
	year={2015},
	publisher={Springer}
}

@Article{jiang2014algorithm,
	title={An algorithm for fast calculation of flow ensembles},
	author={Jiang, Nan and Layton, William},
	journal={International Journal for Uncertainty Quantification},
	volume={4},
	number={4},
	year={2014},
	publisher={Begel House Inc.}
}

@Article{owen2017comparison,
	title={Comparison of surrogate-based uncertainty quantification methods for computationally expensive simulators},
	author={Owen, Nicola E and Challenor, Peter and Menon, Prathyush P and Bennani, Samir},
	journal={SIAM/ASA Journal on Uncertainty Quantification},
	volume={5},
	number={1},
	pages={403--435},
	year={2017},
	publisher={SIAM}
}

@Article{reagana2003uncertainty,
	title={Uncertainty quantification in reacting-flow simulations through non-intrusive spectral projection},
	author={Reagana, Matthew T and Najm, Habib N and Ghanem, Roger G and Knio, Omar M},
	journal={Combustion and flame},
	volume={132},
	number={3},
	pages={545--555},
	year={2003},
	publisher={Elsevier}
}

@Article{ganis2008stochastic,
	title={Stochastic collocation and mixed finite elements for flow in porous media},
	author={Ganis, Benjamin and Klie, Hector and Wheeler, Mary F and Wildey, Tim and Yotov, Ivan and Zhang, Dongxiao},
	journal={Computer methods in applied mechanics and engineering},
	volume={197},
	number={43-44},
	pages={3547--3559},
	year={2008},
	publisher={Elsevier}
}

@Article{nobile2008sparse,
	title={A sparse grid stochastic collocation method for partial differential equations with random input data},
	author={Nobile, Fabio and Tempone, Ra{\'u}l and Webster, Clayton G},
	journal={SIAM Journal on Numerical Analysis},
	volume={46},
	number={5},
	pages={2309--2345},
	year={2008},
	publisher={SIAM}
}

@Article{babuvska2007stochastic,
	title={A stochastic collocation method for elliptic partial differential equations with random input data},
	author={Babu{\v{s}}ka, Ivo and Nobile, Fabio and Tempone, Ra{\'u}l},
	journal={SIAM Journal on Numerical Analysis},
	volume={45},
	number={3},
	pages={1005--1034},
	year={2007},
	publisher={SIAM}
}

@Article{xiu2005high,
	title={High-order collocation methods for differential equations with random inputs},
	author={Xiu, Dongbin and Hesthaven, Jan S},
	journal={SIAM Journal on Scientific Computing},
	volume={27},
	number={3},
	pages={1118--1139},
	year={2005},
	publisher={SIAM}
}

@Article{kuo2012quasi,
	title={Quasi-{Monte Carlo} finite element methods for a class of elliptic partial differential equations with random coefficients},
	author={Kuo, Frances Y and Schwab, Christoph and Sloan, Ian H},
	journal={SIAM Journal on Numerical Analysis},
	volume={50},
	number={6},
	pages={3351--3374},
	year={2012},
	publisher={SIAM}
}

@Article{barth2012multilevel,
	title={Multilevel {Monte Carlo} method with applications to stochastic partial differential equations},
	author={Barth, Andrea and Lang, Annika},
	journal={International Journal of Computer Mathematics},
	volume={89},
	number={18},
	pages={2479--2498},
	year={2012},
	publisher={Taylor \& Francis}
}

@Misc{pgfplots,
  author =	 {Christian Feuers\"anger},
  title =	 {Manual for Package \texttt{PGFPLOTS}},
  month =	 may,
  year =	 2015,
  url =		 {http://sourceforge.net/projects/pgfplots}
}

\end{document}